%% file: log-comp.tex
\documentclass{amsart}

\input{log-comp-preamble}

\begin{document}

\title{Regularized integrals and manifolds with log corners}

\author{Cl\'ement Dupont}
\address{Institut Montpelli\'erain Alexander Grothendieck, Universit\'e de Montpellier, CNRS,
Montpellier, France}
\email{clement.dupont@umontpellier.fr}

\author{Erik Panzer}
\address{Mathematical Institute, University of Oxford, UK}
\email{erik.panzer@maths.ox.ac.uk}

\author{Brent Pym}
\address{McGill University, Montr\'eal, Canada}
\email{brent.pym@mcgill.ca}
\maketitle

\begin{abstract}
We introduce a natural geometric framework for the study of logarithmically divergent integrals on manifolds with corners and algebraic varieties, using the techniques of logarithmic geometry.  Key to the construction is a new notion of morphism in logarithmic geometry itself, introduced by Howell, which allows us to interpret the ubiquitous rule of thumb ``$\lim_{\epsilon\to 0} \log \epsilon := 0$'' as the restriction to a submanifold.  Via a version of de Rham's theorem with logarithmic divergences, we obtain a functorial characterization of the classical theory of ``regularized integration'': it is the unique way to extend the ordinary integral to the logarithmically divergent context while respecting the basic laws of calculus (change of variables, Fubini's theorem, and Stokes' formula.)
\end{abstract}

\setcounter{tocdepth}{1}
\tableofcontents

\section{Introduction}

\subsection{Motivation}\label{sec:motivation}
The need to regulate logarithmic divergences of  integrals is a recurring theme in geometry, number theory, and mathematical physics.  For instance, to make sense of the divergent integral
\[
I_1 \defas
\int_0^a \dlog{x}
\]
one introduces a cutoff parameter $\epsilon > 0$, computes the integral
\[
\int_{\epsilon}^a \dlog{x} = \log(a) -\log(\epsilon)
\]
and formally discards the term that diverges as $\epsilon \to 0$, to obtain the ``regularized'' value
\begin{align}
I_1 = \int_0^a \dlog{x} \defas \log(a) - \cancelto{0}{\log\epsilon} \ \  = \log(a) \label{eq:I1-regularized}
\end{align}
The cost of assigning  a finite value to such a divergent integral is a dependence on a choice of coordinate: if we make a change of variables $x = g(\tilde x)$, with $g$ a diffeomorphism, the regularized integral changes to $\log(a) - \log(g'(0))$, giving a different value unless $g'(0)=1$.  To eliminate this ambiguity, one can fix a  nonzero tangent vector $\vec{v}$ at $x=0$, which for \eqref{eq:I1-regularized} is $\vec{v}=\partial_x|_{x=0}$, and require the coordinate transformation $g$ to preserve this vector; this forces $g'(0)=1$.  Thus we can imagine that the lower bound of integration is not the point $x=0$ per se, but rather the pair $(0,\vec{v})$, which Deligne calls a  ``tangential basepoint'' \cite{Deligne1989}. 

The notion of tangential basepoint is a useful tool, but it poses a basic challenge, in that it gives a notion of ``basepoint in $X$'' that does not actually correspond to a point in $X$.  As a result, it becomes difficult to interpret the regularized integral in a cohomological fashion, that is, as the result of the natural integration pairing between classes in de Rham cohomology and singular homology. This is reflected, e.g.~in Deligne and Goncharov's indirect construction of some special cases of ``motivic fundamental groups with tangential basepoints'' \cite{DeligneGoncharov}. 
In this paper, we will explain how to view a tangential basepoint as an ``actual point'' in a suitable sense, and use this framework to give a precise cohomological meaning to regularized integrals such as $I_1$, in any dimension.

In a similar vein, consider the integral
\[
I_2 \defas \iint_{\PP^1(\CC)}  \rbrac{\frac{\dd z}{z-a} - \frac{\dd z}{z-1}} \wedge \dlog{\zb}
\]
where $z$ is the standard holomorphic coordinate on the Riemann sphere $\PP^1(\CC)$.  (We have borrowed this example from \cite{BrownDupont}.)   This integral presents us with a subtly different problem.  Namely, despite the apparent singularities of the integrand
\[
\omega \defas  \rbrac{\frac{\dd z}{z-a} - \frac{\dd z}{z-1}} \wedge \dlog{\zb}
\]
at the points $0,1,a,\infty \in \PP^1(\CC)$, the integral 
$I_2$ is absolutely convergent, as can be seen by working in polar coordinates centred at each of these points.   More invariantly, we may pass to the real oriented blowup at these points, which is the compact oriented surface with boundary $\Sig$ obtained from $\PP^1(\CC)$ by replacing each point $0,1,a,\infty$ with a boundary circle; then $\omega$ extends to a smooth form on $\Sig$. But an issue arises when we try to \emph{compute} this integral. To do so, we can observe that the form
\[
\alpha \defas -\log|z|^2  \rbrac{\frac{\dd z}{z-a} - \frac{\dd z}{z-1}}
\]
is a primitive for $\omega$, and attempt to apply Stokes' formula
\[
I_2=\iint_\Sig\omega = \int_{\bdSig}i^*\alpha
\]
where $i\colon \bdSig \to \Sig$ is the inclusion of the boundary.  However, the right-hand side does not makes sense as written, because $\alpha$ has a logarithmic pole along $\bdSig$.  Indeed, if we switch to polar coordinates $z - a = r e^{\iu\theta}$, we find that
\[
\alpha = -\log|a|^2 \rbrac{\dlog{r} + \iu\, \dd \theta} + O(1)\,\dd r + O(r)\,\dd \theta \qquad \textrm{as }r\to 0
\]
so that $\alpha$ has a logarithmic pole $\dlog{r} = \dd\log(r)$ at $r=0$ and its restriction is ill-defined.  So, once again, we introduce a cutoff parameter $\epsilon$, excise a tubular neighbourhood of width $\epsilon$ around the boundary, compute the integral over the boundary of the resulting surface $\Sig_\epsilon$, and take the limit as $\epsilon \to 0$.  With a bit of care, one sees that this boils down to computing the residue of $\alpha$ at $z=a$.  Consequently, the same result can be achieved without introducing a cutoff parameter, by instead  defining a ``regularized restriction'' $\reg\, i^*$, in which we formally set $\dlog{r}$ to zero on the boundary, giving
\[
\iint_{\PP^1(\CC)} \omega = \int_{\bdSig} \reg\, i^*\alpha = -\log|a|^2 \cdot \int_{\bdSig} \left(\cancelto{0}{\reg\, i^* \dlog{r}} + \iu\, \dd\theta\right) = 2\pi\iu \log|a|^2.
\]
(The disappearance of the minus sign is explained by the orientation of $\partial\Sig$.) In this paper, we will explain how this ad hoc construction of $\reg\, i^*\alpha$ can be interpreted, in a precise sense, as the pullback of $\alpha$ along a morphism, resulting in a ``regularized Stokes theorem''.  Once again, the form $\reg\, i^* \alpha$ itself depends on tangential data---this time, a trivialization of the normal bundle of $\bdSig$---but in this case, the value of the integral is ultimately independent.

At first glance, the dependence on choices seems more like a mild nuisance than a serious issue, but the situation becomes more critical when one has \emph{many} integrals and wants to prove nontrivial relations between them using the basic laws of integration: change of variables, Fubini's theorem, and Stokes' formula.  These integrals may diverge (as in the case of $I_1$ above), but even if they converge, one may wish to make use of divergent forms in the calculations (as in the case of $I_2$). For instance, our own interest in the problem stems from Feynman-style integrals in quantum field theory, and related structures in quantum algebra, where one has a whole infinite collection of integrals indexed by graphs; they satisfy many relations amongst themselves, and logarithmic divergences abound.

Going further, one may seek to reformulate such identities between integrals in purely cohomological terms, which grants access to the powerful tools of algebraic geometry and Hodge theory, and reveals a hidden symmetry in the form of the action of a motivic Galois group.  In other words, one may wish to establish such identities in the ring of ``motivic periods''~\cite{KontsevichZagier,BrownMotivicPeriods} (or a logarithmic variant thereof).  The formalism we develop below in this paper provides exactly such an interpretation.  In future work, we will apply it to the integrals appearing in deformation quantization~\cite{KontsevichOperadsMotives,Kontsevich2003,AlekseevRossiTorossianWillwacher:Logs}, thus realizing Kontsevich's vision of a motivic Galois action in this context.

\subsection{Approach and results}\label{sec:results}
In this paper, we solve the problems above by introducing a new class of geometric objects, called \defn{manifolds with log corners}, which serve as the natural domains of integration for logarithmically divergent forms, in the same way that ordinary manifolds with corners are used for the integration of smooth forms.  In particular, there is a natural geometric notion of ``regularization'' for manifolds with log corners, which serves as a replacement for the ad hoc introduction of cutoffs and tubular neighbourhoods in the usual approach, retaining only the essential geometric data (such as tangential basepoints) that are needed to regulate the integral.

In the rest of the introduction, we shall give an overview of the main definitions and ideas.  As a preview of what is to come, let us summarize the main results of the paper as follows:

\begin{theorem}The following statements hold:
\begin{enumerate}
    \item Manifolds with log corners carry functorial sheaves of functions and differential forms with logarithmic divergences, satisfying versions of the Poincar\'e lemma and de Rham isomorphism.
    \item When such manifolds are equipped with the additional data of an orientation and a regularization, logarithmically divergent forms can be naturally integrated, in a way that respects the basic laws of calculus: change of variables, Fubini's theorem and Stokes' formula.
    \item A natural source of manifolds with log corners is provided by the ``Kato--Nakayama'' spaces of a class of logarithmic algebraic varieties, which we call ``varieties with log corners''.  The corresponding regularized integrals provide a cohomological theory of periods for varieties with log corners that is naturally compatible with Deligne's tangential basepoints.
\end{enumerate}
\end{theorem}

\subsubsection{Manifolds with log corners}
  The most basic example of a manifold with log corners is provided by a manifold with corners in the classical sense.  More generally, we may consider boundary faces of a manifold with corners, along with the natural data they carry on their normal bundles---specifically, the collection of ``positive'' normal vectors, i.e.~those which point into the interior.  These data play a crucial role in regularizing integrals.

While it is possible to describe manifolds with log corners in purely classical terms (see \autoref{sec:phantom-tangent}), such a description is, in our experience, cumbersome.  For this reason, we adopt the framework of ``positive log geometry'' of Gillam--Molcho \cite{GillamMolcho}---a differential geometry version of the logarithmic algebraic geometry of Fontaine--Illusie--Kato~\cite{Kato1989a, IllusieLogSpaces, Ogus}.  In other words, we consider a manifold with corners $\Sig$ endowed with the additional datum of a sheaf of monoids $\cM_\Sig$ mapping to the sheaf of non-negative smooth functions.  In the basic case of a manifold with corners, $\cM_\Sig$ is simply the sheaf of non-negative smooth functions that are locally monomial in coordinates near the boundary.  For a boundary face, we include additional elements that are ``phantoms'' of such functions on the normal bundle.   See \autoref{sec:log-corners} for the general definition.  For the introduction, it will suffice to have the following examples in mind.  

The first is the ``standard interval'' $\halfspc{}$, which we equip with the sheaf $\cM_{\halfspc{}}$ of functions of the form $g(r)r^j$ where $r$ is the standard coordinate, $g$ is a positive smooth function, and $j \in \NN$; this is a multiplicative submonoid of the sheaf $\Cinfge{\halfspc{}}$ of non-negative smooth functions.   

The second is its boundary, the ``standard end'' $[0) \defas \partial [0,\infty)$, consisting of the point $0$ equipped with the monoid $\cM_{[0)}$ of  monomials $\lambda t^j$ where $\lambda \in \RR_{>0}$ is a positive constant and $t$ is a formal coordinate corresponding to the derivative of $r$ at zero; this coordinate is not really a function, but rather a phantom thereof, in the sense that its ``value'' is defined by formally setting $t=0$.

A general manifold with log corners is then locally isomorphic to an open set in a product $\model{n}{k}$, and thus is covered by local charts consisting of ``basic coordinates'' $r_1,\ldots,r_n$ on the underlying manifold with corners, and ``phantom coordinates'' $t_1,\ldots,t_k$ that keep track of normal directions.

\subsubsection{Log morphisms and tangential basepoints}
Log geometry gives a natural notion of morphism between manifolds with log corners, expressed in terms of sheaves of monoids.  For the basic case of manifolds with corners, such morphisms have a very concrete description: they are the maps that, when written in coordinates, are locally monomial near the boundary.  The importance of such maps in differential geometry and geometric analysis was articulated by Melrose, who called them ``interior $b$-maps''~\cite{Melrose1992}.  A basic feature of these maps, as the name suggests, is that they preserve interiors of manifolds with corners; in particular, morphisms $* \to \Sig$, where $*$ is a single point, are in bijection with points in the interior $\Sig\setminus\bdSig$.

One of our key discoveries, made independently\footnote{In the first arXiv version of this paper, we were working with a notion of morphism that we called ``weak'', see \autoref{rem:virtual-vs-weak}. After learning about Howell's work, we decided to adopt his framework of ``virtual'' morphisms. The difference between the two notions is subtle and ends up being essentially irrelevant for the applications to regularized integrals in this paper. However, virtual morphisms in the sense of Howell are the right notion for other applications, such as formality of the little disks operad; see our article \cite{DPP:WeakMorphisms}.} by Howell \cite{Howell2017} in algebraic geometry, is that by a simple weakening of the axioms for a morphism in log geometry, we obtain a more flexible notion of ``virtual morphism'' which also allows points to land in the boundary; however, when they do so, they automatically come decorated with positive normal vectors---a $C^\infty$ counterpart of Deligne's tangential basepoints. In fact, tangential basepoints are the same thing as virtual morphisms from a point:

\begin{proposition}[see \autoref{prop:tangential-basepoints-as-weak-morphisms}]
For $\Sig$ a manifold with corners viewed as a manifold with log corners, virtual morphisms $*\to \Sig$ are in bijection with tangential basepoints of $\Sig$.
\end{proposition}

The category of manifolds with log corners and virtual morphisms is the natural venue for our theory of integration. 

\subsubsection{Functions and forms}

In \autoref{sec:log-functions} we develop a theory of functions with logarithmic singularities.  We construct, for each manifold with log corners $\Sig$, a sheaf $\Cinflog{\Sig}$ of ``logarithmic functions''.  This is done purely algebraically, by a simple generators-and-relations presentation, in which we formally adjoin logarithms for the elements of $\cM_\Sig$, subject to two natural relations.  The first relation is the obvious identity $\log(fg) = \log(f)+\log(g)$.  The second relation identifies a formal symbol $f \log(g)$ with the corresponding function on $\Sig$, provided that the latter is everywhere smooth; this ensures that we do not ``overcount'' the smooth functions.

This presentation makes it easy to show that the sheaf $\Cinflog{}$ is functorial for virtual morphisms, which will be crucial for our geometric interpretation of regularized integrals. 
 However, it obscures its analytic meaning.  To this end, we prove that this sheaf has the following classical description:
\begin{theorem}[see \autoref{thm:injectivity-of-j*}]\label{thm:log-functions}
The sheaf $\Cinflog{\Sig}$ is identified with the algebra of functions that in any system of basic coordinates $r_1,\ldots,r_n$ and phantom coordinates $t_1,\ldots,t_k$ on $\Sig$, admit finite expansions of the form
\[
f = \sum_{I,J} f_{I,J}(r)\log^I(r) \log^J(t)
\]
where $I,J$ are multi-indices and $f_{I,J}$ are smooth functions.
\end{theorem}

The crucial subtlety here---which is why there is something nontrivial to prove---is that the expansion in $\log(r)$ is not unique, due to the possibility of smooth function coefficients that are infinitely flat at the boundary.  The theorem shows that this ambiguity is completely captured by the elementary relation defining $\Cinflog{\Sig}$, so that we can manipulate the formal expansions as functions in the obvious way.

In the absence of phantom coordinates, the sections of $\Cinflog{\Sig}$ are examples of polyhomogeneous functions in the sense of Melrose~\cite{Melrose1992}.  Meanwhile the symbols $\log^J(t)$ are monomials in the formal logarithms in the phantom coordinates, which are related to Melrose's polyhomogeneous symbols; they keep track of the singular terms that arise when attempting to restrict polyhomogeneous functions to strata.

Even though functions may diverge on the boundary, we can assign finite values  at ``points'', simply by pulling them back along virtual morphisms $* \to \Sig$.  As we explain in \autoref{sec:log-functions}, this simple prescription exactly encapsulates the classical approach of choosing a tangential basepoint, and using it to define the regularized limit by discarding divergent terms as in \eqref{eq:I1-regularized} above.

\subsubsection{Differential forms and de Rham cohomology}
In a similar way, we obtain a functorial sheaf of differential forms $\cAlog{\Sig}$ with logarithmic singularities; it is the smallest extension of the dg algebra of smooth forms which contains all logarithmic functions $\cAlog[0]{\Sig} \defas \Cinflog{\Sig}$, and hence also their differentials.  In local coordinates, the latter have a basis given by the elements $\dd\log(r_i)=\dlog{r_i}$ and $\dd\log(t_j) = \dlog{t_j}$.  We prove the following logarithmic version of the Poincar\'e lemma and de Rham isomorphism, giving a topological interpretation for the cohomology of $\cAlog{\Sig}$.

\begin{theorem}[see Sections \ref{sec:poincare} and \ref{sec:relative-cohomology}]\label{thm:intro-deRham}
The natural map $\RR_{\Sig} \hookrightarrow \cAlog{\Sig}$ is a quasi-isomorphism of complexes of sheaves, and hence we have canonical isomorphisms
\[
\mathsf{H}^\bullet_{\mathrm{sing}}(\Sig;\RR)  \cong \HdR{\Sig}
\]
where $\HdR{\Sig}$ is the cohomology of the complex of global logarithmic forms.  Similar isomorphisms hold more generally for relative cohomology, with or without compact supports.  In particular, logarithmic de Rham cohomology is homotopy invariant and satisfies the K\"unneth formula.
\end{theorem}
The case of relative cohomology is particularly important, since in practice, domains of integration often have a boundary.

\subsubsection{Integration}

In \autoref{sec:integration} we explain how our de Rham isomorphism above can be used to construct a natural and cohomologically meaningful theory of integration on manifolds with log corners.  

Integration of forms in ordinary differential geometry requires an orientation to sort out the signs. In the logarithmic setting, we require the additional data of a \defn{regularization} of the manifold with log corners to control the divergences.  In classical terms, this amounts to the data $s$ of a mutually compatible collection of non-negative, locally monomial sections of the normal bundles of the faces, which allows for quite a lot of variation in the qualitative behaviour.  We package this succinctly as a collection of virtual morphisms respecting the natural combinatorics of the boundary (that of a ``symmetric semi-simplicial set''); see \autoref{def:regularization}. It is then immediate that if $(\Sig,s)$ is a regularized manifold with log corners, its boundary $\partial \Sig$ comes with a canonical regularization $\partial s$.  We establish the following.

\begin{theoremdefn}[see \autoref{sec:integration}]
There is a unique collection of functionals on compactly supported log forms
\[
\int_{(\Sig,s)} \colon \Alogc[n]{\Sig} \to \RR,
\]
one for each oriented and regularized manifold with log corners $(\Sig,s)$ whose underlying manifold has dimension $n \ge 0$, which reduces to the ordinary integral whenever the latter converges absolutely, and which satisfies the regularized Stokes formula
\[
\int_{(\Sig,s)}\dd \alpha = \int_{(\bdSig,\partial s)}\alpha.
\]
\end{theoremdefn}

We then show by a straightforward calculation in coordinates that $\int_{(\Sig,s)}$ reduces to the classical regularized integral, defined in terms of regularized limits of asymptotic expansions relative to tangential basepoints.  Calculations such as the integrals $I_1$ and $I_2$ (from \autoref{sec:motivation} above) then become direct applications of the definition and Stokes' formula, as desired.

By construction, the regularized integral descends to logarithmic de Rham cohomology, where it induces the Alexander--Lefschetz--Poincar\'e duality pairing between absolute cohomology, and cohomology relative to the boundary.  Note that this pairing is purely topological, and hence independent of the choice of regularization; this is reflected in the fact that the relative cohomology can be computed using logarithmic forms that vanish on the boundary.  Such forms are absolutely integrable, even if they are not smooth, so that their integral is independent of regularization.

\subsubsection{Regularized periods in algebraic geometry}

We conclude in \autoref{sec:periods} by connecting our setup with the study of periods in algebraic geometry, i.e.~integrals of algebraic differential forms over topological cycles.

There are natural algebro-geometric analogues of manifold with log corners, which we call \defn{varieties with log corners}; these are the logarithmic algebraic varieties over $\CC$ that arise as strata of normal crossing divisors in smooth varieties. The connection between these two worlds is obtained via Kato--Nakayama's construction~\cite{KatoNakayama} of a topological space $\KN{X}$ associated to a log scheme $X$ over the complex numbers.  In \emph{op.~cit.}, $\KN{X}$ is viewed as a topological space with no additional structure, but as explained by Gillam--Molcho~\cite{GillamMolcho}, it actually comes equipped with a natural positive log structure.

For the basic case in which $X = Y \log D$ is the log scheme associated to a smooth variety $Y$ equipped  with a normal crossing divisor $D$, the resulting space $\KN{X}$ is the real-oriented blowup of (the analytification of) $Y$ along $D$.  The result is a manifold with corners whose boundary components are circle bundles over the irreducible components of $D$, viewed as a basic example of a manifold with log corners with no phantoms.  

In general, one may also have phantom directions in a variety with log corners; the prototype is the ``standard log point'' given by the induced log structure on the origin in $\logcyl$. Its Kato--Nakayama space is the manifold with log corners $\unitcirc \times [0)$, with basic and phantom coordinates corresponding to the pullback of angular and radial coordinates on the complex line $\AF^1(\CC) = \CC$, respectively.    We have the following.

\begin{theorem}[see Sections \ref{sec:KN} and \ref{sec:real-KN}]
The Kato--Nakayama space construction defines a functor $X \mapsto \KN{X}$ from the category of varieties with log corners over $\CC$ and virtual morphisms, to the category of manifolds with log corners and virtual morphisms.  If, in addition, $X$ is defined over $\RR$, then its set of real points lifts to a natural embedded submanifold with log corners $\KNR{X}\subset \KN{X}$.
\end{theorem}

As an immediate consequence, there is a natural isomorphism
\[
\HdR{X}  \to \HdR{\KN{X}}\otimes_\RR \CC
\]
where $\HdR{X}$ is the \emph{algebraic} de Rham cohomology of Kato--Nakayama \cite{KatoNakayama} defined using logarithmic algebraic differential forms, and $\HdR{\KN{X}}$ is the cohomology of the forms with logarithmic singularities on the manifold with log corners $\Sig=\KN{X}$ as above.

More generally, we may consider the cohomology of a diagram $X_\bullet$ of varieties with log corners and virtual morphisms---for instance the diagram of boundary inclusions of a normal crossing divisor, which gives rise to a relative cohomology group.  Using this one can define an algebra of \defn{logarithmic periods} as the numbers obtained by the natural pairing between singular  (a.k.a.~Betti) homology and algebraic de Rham cohomology for diagrams of varieties with log corners defined over $\QQ$.  We expect that this algebra (and its ``motivic'' counterpart) equals the ring of ordinary (non-regularized) periods defined by Kontsevich--Zagier~\cite{KontsevichZagier}.

For example, as we explain in \autoref{sec:period-examples}, the integrals $I_1$ and $I_2$ above can both be viewed as periods of varieties with log corners $X$ and $Y$ where the domains of integration are cells of the real loci $\KNR{X}$ and $\KNR{Y}$.  We further explain an alternative approach to the single-valued integration and the double-copy formula from \cite{BrownDupont}, which gives a complementary viewpoint on the resulting motivic relation between $\log a$ and $2\pi\iu \log |a|^2$.

\subsection{Future work}

In the article \cite{DPP:WeakMorphisms} we develop further the notion of virtual morphism in logarithmic algebraic geometry, proving that it retains many good functoriality properties and applying it to the formality of the little disks operad.

In future work we will explain how the construction of real Kato--Nakayama spaces from log structures on the moduli space of stable genus zero curves can be used to give motivic meaning to the integrals appearing in deformation quantization, and their evaluation in terms of multiple zeta values~\cite{BPP}. In this way, we will explain that the subtle ``weight drop'' phenomenon observed analytically in \emph{op.~cit.} is, in fact, a consequence of the geometry (via mixed Hodge theory), and realize the aforementioned motivic Galois action.

\subsection{Relation to other work}

Many authors have approached regularized integration from various points of view.  For instance, Felder--Kazhdan~\cite{Felder2017,Felder2018} and Li--Zhou~\cite{Li2021,Li2023} have studied such integrals in the presence of a suitable conformal structure, by introducing cutoff functions and examining the asympotics of the integral as the cutoff tends to zero; they show explicitly that the ``finite part'' of the integral only depends on a trivialization of the outward-pointing normal bundle, via residues.  Our regularization procedure gives the same definition and dependence on choices without the need for cutoff functions or direct asymptotic analysis of the integral.  (The only place where asymptotics appear in our approach is to prove \autoref{thm:log-functions}, which establishes the structure of the sheaf $\Cinflog{}$ of logarithmic functions.) 

Meanwhile, in \cite{RuddatSiebert}, Ruddat--Siebert develop a formalism of regularized integration along the fibers of certain log smooth morphisms, which should also be compatible with ours, although we did not try to relate the two approaches. 

Regularized integrals also abound in the theory of automorphic periods (see, e.g. \cite{Zydor} and the references therein), for which the appearance of local toric models \cite{Sakellaridis} suggests that our approach via logarithmic geometry could apply.

In \cite{Brown2009}, Brown studied period integrals on the moduli space of genus zero curves. In the process, he proved a version of Stokes' theorem for certain forms with logarithmic singularities; it corresponds to the special case of our regularized Stokes' theorem in which the forms have no poles, and directly inspired our approach.

In \cite{AlekseevRossiTorossianWillwacher:Logs}, Alekseev--Rossi--Torossian--Willwacher formulated a regularized version of Stokes' theorem for manifolds with corners, equipped with suitable torus actions near the boundary strata.  This can be understood as the special case of our regularized Stokes formula, in which the regularization of $\Sig$ is chosen to be torus-invariant, and the form satisfies their ``regularizability'' criterion.  They then applied their approach to the integrals appearing in deformation quantization mentioned above; as we will explain in future work, the torus action in these examples is the natural phase rotation on the corresponding Kato--Nakayama spaces.

While our paper was in preparation, Kato--Nakayama--Usui posted an interesting preprint~\cite{KatoNakayamaUsui}, in which they study $C^\infty$ logarithmic functions on log complex analytic spaces by combining the Kato--Nakayama spaces with an additional ``space of ratios''.  While there are some formal similarities (e.g.~a Poincar\'e lemma), the aims, results, and approach appear to be quite different.

In \cite{turdorvault}, Tur-Dorvault used the formalism of the present paper to identify the periods of the motivic fundamental groupoids based at tangential basepoints.
 
Finally, our logarithmic de Rham theorem (\autoref{thm:intro-deRham}) has subtly different analogues in other contexts.  Firstly, there is an analogous result for holomorphic logarithmic de Rham complexes, due to Kato--Nakayama~\cite{KatoNakayama}.  Their proof is essentially algebraic, relying on the fact that $\log(z)$ is algebraically independent from holomorphic functions, which fails in the $C^\infty$ context. In contrast, our proof of \autoref{thm:intro-deRham} is quite close to the classical argument in differential geometry via contracting homotopies; this is made possible by the notion of virtual morphism.   Secondly, in a different direction, Mazzeo--Melrose~\cite[\S2.16]{Melrose1993} computed the cohomology of smooth forms on a manifold with corners with log poles on the boundary (but no log divergences in the coefficient functions). This gives a different answer, since absent the function $\log r$, the form $\dlog r$ is not exact near the boundary.

\subsection*{Conventions and notation}

All monoids are implicitly commutative and with a monoid law written multiplicatively. The main exception is the set $\NN$ of natural numbers (i.e.~non-negative integers, including zero), which we view as a monoid under addition.
    
\subsection*{Acknowledgements} We thank Piotr Achinger, Federico Binda, Damien Calaque, Danny Gillam, Samouil Molcho, Sam Payne, Yiannis Sakellaridis, and Matt Satriano for helpful discussions and correspondence, and the referee for their valuable comments. This work grew out of discussions held by the authors during a Research in Pairs activity at the Mathematisches Forschungsinstitut Oberwolfach (MFO).  We thank the MFO for providing excellent hospitality and working conditions.
C.~D.\ was supported by the Agence Nationale de la Recherche through grants ANR-18-CE40-0017 PERGAMO and ANR-20-CE40-0016 HighAGT.
E.~P.\ was funded as a Royal Society University Research Fellow through grant {URF{\textbackslash}R1{\textbackslash}201473}.
B.~P.\ was supported by Discovery Grant RGPIN-2020-05191 from the Natural Sciences and Engineering Research Council of Canada (NSERC), \'Etablissement de la rel\`eve professorale grant 313101 from the Fonds de recherche du Qu\'ebec -- Nature et technologies (FRQNT), and a startup grant at McGill University.

\section{Manifolds with corners}
There are several approaches to manifolds with corners in the literature; in this section we review the basic definitions and conventions used in this paper.  

\subsection{Definitions}

For subsets $A\subset \RR^m$ and $B\subset \RR^n$ we say that a map $\phi\colon A\to B$ is \defn{smooth} if it extends to a smooth $\RR^m$-valued function on an open neighbourhood of $A\subset \RR^n$.
 
A \defn{diffeomorphism} is a smooth map $\phi\colon A\to B$ that admits a smooth inverse.

A \defn{chart (with corners)} on a topological space $W$ is a pair $(U,\phi)$ where $U\subset W$ is an open set and $\phi\colon U\to \halfspc{n}$ is a continuous map that induces a homeomorphism from $U$ to an open subset of $\halfspc{n}$, for some $n\in\NN$.  We typically denote the coordinate functions in such a chart by $\phi=(r_1,\ldots,r_n)$, as they measure the distance from the origin in each direction.

Two charts $(U,\phi)$ and $(V,\psi)$ are said to be \defn{compatible} if the transition function $\psi\circ \phi^{-1}\colon \phi(U\cap V)\to \psi(U\cap V)$ is a diffeomorphism. An \defn{atlas (with corners)} on $W$ is a set of charts that are pairwise compatible and whose union is $W$.

\begin{definition}\label{def:mfld-with-corners}
A \defn{manifold with corners} is a second countable Hausdorff space equipped with a maximal atlas (with corners).
\end{definition}
For any point $x$ in a manifold with corners $W$, there is a unique $j \ge0$ such that $x$ lies in a chart $(U,\phi)$ for which $\phi(x) \in \{0\}^j \times (0,\infty)^{n-j}$, i.e.
\begin{align}
r_1(x)=\cdots=r_j(x)=0 \qquad \textrm{and} \qquad r_{j+1}(x),\ldots,r_n(x) > 0. \label{eq:good-coords}
\end{align}
We call this integer $j$ the \defn{depth} of $x$.

The \defn{interior} is the open set $W^\circ \subset W$ consisting of points of depth zero, i.e.~for which all coordinates are positive.  It is a smooth manifold (without corners).
\begin{remark}
Every manifold with corners is covered by open sets diffeomorphic to $\halfspc{j}\times\RR^{n-j}$ for some $0\le j\le n$, since $\halfspc{n}$ has a basis of such open sets.
\end{remark}

\begin{definition}
A \defn{smooth map} between manifolds with corners is a map that is smooth in every chart.  We denote by $\Cinf{W}$ the sheaf of smooth $\RR$-valued functions on $W$, and by $\Cinfpos{W}\subset\Cinfge{W}\subset\Cinf{W}$ the subsheaves of functions whose values are strictly positive, and non-negative, respectively.
\end{definition}
\begin{remark}\label{rem:Ck-topology}
There are different notions of smooth maps between manifolds with corners.  (For instance, the notion we use here is called \emph{weakly smooth} in \cite{Joyce:ManifoldsWithCorners}.)
By a theorem of Seeley~\cite{Seeley}  (see also \cite{Melrose1996}),  a function $f\colon [0,\infty)^n \to \RR$ is smooth in the present sense if and only if it is smooth in the interior $(0,\infty)^n$, with all partial derivatives continuous on $\halfspc{n}$.  Equipped with the Fr\'echet topology, the set of such functions is  a complete, locally convex topological vector space. 
\end{remark}

\subsection{Tangent structure}\label{sec:tangent-structure-manifolds-with-corners}

    For a manifold with corners $W$ and a point $x\in W$, the tangent space $\tb[x]{W}$, the cotangent space $\ctb[x]{W}$ and the differential $\dd f|_x \in \ctb[x]{W}$ of a smooth function are defined in the usual way, via derivations of $\Cinf{W}$.  On the boundary, only some vectors actually point into $W$; we single them out as follows.  

    The \defn{non-negative tangent space} of $W$ at $x$ is the closed subset
    \[
    \tbge[x]{W} \defas \set{v \in \tb[x]{W}}{\abrac{v,\dd f|_x} \ge 0 \textrm{ for all } f \in \Cinfge{W,x} \textrm{ such that $f(x)=0$}}
    \]
    Its boundary lies on a union of hyperplanes, called the \defn{boundary tangent hyperplanes}, defined by $\abrac{v,\dd r|_x}=0$ for $r$ a non-negative local coordinate at $x$. Its interior is the \defn{positive tangent space} $\tbpos[x]{W} \subset \tbge[x]{W}$.  In coordinates $(r_1,\ldots,r_n)$ satisfying \eqref{eq:good-coords}, a vector $v =a_1\cvf{r_1}|_x+\cdots +a_n\cvf{r_n}|_x \in \tb[x]{W}$ is non-negative (resp.~positive) if and only if  the first $j$ coefficients $a_1,\ldots,a_j$ are non-negative (resp.~positive), so that the coordinate basis gives identifications
    \[
    T_xW\cong \RR^n, \qquad T_x^{\ge 0}W \cong \halfspc{j}\times \RR^{n-j}, \qquad T_x^{>0}W\cong (0,\infty)^j\times \RR^{n-j}.
    \]
    The boundary tangent hyperplanes are identified with the first $j$ coordinate hyperplanes of $\RR^n$, i.e.~the vanishing sets of the linear functionals $\dd r_1|_x,\ldots,\dd r_j|_x\in \ctb[x]{W}$.

    The \defn{tangent face} of $W$ at $x$ is the intersection of all boundary tangent hyperplanes at $x$, given by
    \[
    F_x W = \set{v \in \tb[x]{W}}{\abrac{v,\dd f|_x} = 0\textrm{ for all } f\in  \Cinfge{W,x}\textrm{ such that }f(x)=0}
    \]
    The quotient $N_x W\defas T_x W/F_x W$ is called the \defn{normal space} of $W$ at $x$.  The \defn{positive} and \defn{non-negative normal spaces} $N_x^{>0}W\subset N_x^{\ge 0}W\subset N_xW$ are the projections of the corresponding subsets of $\tb[x]{W}$.  The coordinate basis above gives
    \[
    F_x W\cong \{0\}^j\times \RR^{n-j}
    \] and isomorphisms
    $$N_x W\cong \RR^j, \qquad N_x^{\ge 0} W \cong \halfspc{j}, \qquad \textrm{and}  \qquad N_x^{>0}  W \cong (0,\infty)^j.$$
    Note that when $x \in W^\circ$ is an interior point, we have $N_xW=N_x^{\ge0}=N_x^{>0} = \{0\}$, so that in this case only, the zero vector is ``positive''.

\subsection{Boundary}\label{sec:boundary-manifolds-with-corners}
  The \defn{topological boundary} of a manifold with corners $W$ is the complement of the interior $\partial_{\mathrm{top}} W := W \setminus W^\circ$.  It does not naturally have the structure of a manifold with corners in general.
  
  Rather, the correct notion of boundary has the following local picture.  In the orthant $\halfspc{n}$ with coordinates $(r_1,\ldots,r_n)$, we define the boundary $\partial\halfspc{n}$ as the disjoint union of the $n$ coordinate suborthants $\{r_j=0\}\cong \halfspc{n-1}$:
    \[
    \partial\halfspc{n} = \bigsqcup_{j=1}^n\,\halfspc{j-1}\times\{0\}\times \halfspc{n-j}.
    \]
    The natural map
    \[
    i\colon \partial\halfspc{n}\to \halfspc{n}
    \]
    is an immersion of manifolds with corners that fails to be injective if $n>0$, since the origin of $\halfspc{n}$ has $n$ preimages by $i$; see \autoref{fig:boundary-quadrant}.
    
    \begin{figure}
        \centering
        \begin{subfigure}[t]{0.3\textwidth}
        \centering
        \begin{tikzpicture}[scale=1]
            \draw[white,fill=white!90!black] (0,0) rectangle (1.9,1.9);
            \draw[thick,->] (0,0) -- (2,0);
            \draw[thick,->] (0,0) -- (0,2);
            \draw[fill] (0,0) circle (0.05);
        \end{tikzpicture}
        \caption{$\halfspc{2}$}
        \end{subfigure}
        \begin{subfigure}[t]{0.3\textwidth}
        \centering
        \begin{tikzpicture}[scale=1]
            \draw[thick,->] (0.15,0) -- (2,0);
            \draw[thick,->] (0,0.15) -- (0,2);
            \draw[fill] (0.15,0) circle (0.05);
            \draw[fill] (0,0.15) circle (0.05);
        \end{tikzpicture}
        \caption{$\partial \halfspc{2}$}
        \end{subfigure}
        \caption{The boundary of the quadrant $\halfspc{2}$ is the disjoint union of two copies of $\halfspc{}$.}
        \label{fig:boundary-quadrant}
    \end{figure}
    
    The \defn{boundary} $\partial W$ of a general manifold with corners is obtained by glueing this local construction. It is a manifold with corners equipped with an immersion 
    \[
    i\colon \partial W\to  W.
    \]
    One can view $\partial W$ as the set of all pairs $(x,b)$ where $x\in W$ and $b$ is a boundary tangent hyperplane of $ W$ at $x$. Note that $i$ need not be injective on connected components of $\partial W$, as shown by the example of the ``teardrop manifold'' in \autoref{fig:teardrop}. 

    \begin{figure}
        \centering
        \begin{subfigure}[t]{0.3\textwidth}
        \centering
        \begin{tikzpicture}[scale=1]
            \draw[thick,fill=white!90!black] (0,0) -- (1,0) arc (-90:180:1) -- (0,0);
            \draw[blue!50!red,fill] (0,0) circle (0.05);
        \end{tikzpicture}
        \caption{$ W$}
        \end{subfigure}
        \begin{subfigure}[t]{0.3\textwidth}
        \centering
        \begin{tikzpicture}[scale=1]
            \draw[thick] (0.15,0) -- (1,0) arc (-90:180:1) -- (0,0.15);
            \draw[blue,fill] (0.15,0) circle (0.05);
            \draw[red,fill] (0,0.15) circle (0.05);
        \end{tikzpicture}
        \caption{$\partial W$}
        \end{subfigure}
        \begin{subfigure}[t]{0.3\textwidth}
        \centering
        \begin{tikzpicture}[scale=1.5]
            \draw[blue,fill] (1,0.5) circle (0.05);
            \draw[red,fill] (0,1.5) circle (0.05);
        \end{tikzpicture}
        \caption{$\partial^2 W$}
        \end{subfigure}
        \caption{The ``teardrop'' manifold $ W$ is a compact surface with a single corner and a single boundary component.  Its double boundary consists of two points that are interchanged by the action of $\symgp{2}$, and correspondingly have the same image in $ W$.}
        \label{fig:teardrop}
    \end{figure}

    Iterating, we obtain for each $k>0$ a manifold with corners $\partial^{k} W:=\partial (\partial^{k-1}W)$, whose points are pairs $(x,(b_1,\ldots,b_k))$ consisting of a point $x\in W$ and an ordered list $(b_1,\ldots,b_k)$ of pairwise distinct boundary tangent hyperplanes of $ W$ at $x$. There is a free action of the symmetric group $\mathfrak{S}_k$ on $\partial^k W$ which permutes the boundary tangent hyperplanes $b_j$, and there are $k$ smooth maps $\partial^k W\to \partial^{k-1} W$ which forget one of the $b_j$. This structure makes the collection $\partial^\bullet W$ into a symmetric semi-simplicial object (in the sense of \cite{ChanGalatiusPayne}, see \S\ref{sec:semi-simplicial-objects}) in the category of manifolds with corners.
    
    The quotient $
    \overline{\partial}^k  W \defas \partial^k W/\mathfrak{S}_k$ is a manifold with corners, whose points are pairs $(x,\{b_1,\ldots,b_k\})$, where $x\in  W$ and  $\{b_1,\ldots,b_k\}$ is a set of pairwise distinct boundary tangent hyperplanes of $ W$ at $x$.  Its connected components are called the \defn{(boundary) faces of $W$}.  These objects are less practical to work with than the $\partial^k W$, because of the impossibility to define maps $\overline{\partial}^{k+1}  W\to \overline{\partial}^{k}  W$ for $k > 0$ in a choice-free manner. There is, however, a natural immersion $\overline{\partial}^k  W\to  W$ for each $k$, which identifies the interior $(\overline{\partial}^kW)^\circ$ with the set of depth-$k$ points in $W$, and the tangent space at $(x,\{b_1,\ldots,b_k\})\in (\overline{\partial}^kW)^\circ$ with the tangent face $F_xW\subset \tb[x]{W}$.  
    
\section{Manifolds with log corners}\label{sec:log-corners}
    
\subsection{Positive log  structures}\label{sec:log-structures}

\subsubsection{Definitions}

Let $\Sig$ be a manifold with corners.  Note that the subsheaves of positive and non-negative functions $\Cinfpos{\Sig} \subset \Cinfge{\Sig} \subset \Cinf{\Sig}$
are sheaves of submonoids with respect to the operation of multiplication, and that $\Cinfpos{\Sig}$ is the subgroup of invertible sections of the monoid $\Cinfge{\Sig}$. 

\begin{definition}
A \defn{positive pre-logarithmic} (or \defn{pre-log}) \defn{structure} on $\Sig$ is a sheaf of monoids $\cM_\Sig$ on $\Sig$ along with a morphism of sheaves of monoids 
\[
\alpha\colon \cM_\Sig\to \Cinfge{\Sig}.
\]
It is called a \defn{positive logarithmic} (or \defn{log}) \defn{structure} if the induced morphism
\[
\alpha^{-1}(\Cinfpos{\Sig})\to \Cinfpos{\Sig}
\]
is an isomorphism.  
\end{definition}
We will often abuse notation and simply write $\Sig$ for a triple $(\Sig,\cM_\Sig,\alpha)$ consisting of a manifold with corners equipped with a (pre-)log structure. When we want to distinguish between such an object and the underlying manifold with corners, we will use the notation $\uSig$ for the latter.

\begin{example}\label{ex:trivial}
    If $\Sig$ is a manifold with corners, then taking $\cM_{\Sig} = \Cinfpos{\Sig}$ with $\alpha\colon \cM_{\Sig} \hookrightarrow \Cinfge{\Sig}$ the inclusion, we obtain a positive log structure on $\Sig$, which we call the \defn{trivial positive log structure}.
\end{example}
    \begin{example}\label{ex:interval}
    On the manifold with corners $[0,\infty)$ with standard coordinate $r$, we define a sheaf of submonoids $\cM_{[0,\infty)}\defas \Cinfpos{[0,\infty)}r^{\NN}\subset \Cinfge{\halfspc{}}$ as the functions which can be written as a positive smooth function times a non-negative integer power of $r$. One readily checks that the inclusion
    \[
    \alpha\colon \Cinfpos{[0,\infty)}r^\NN \hookrightarrow  \Cinfge{[0,\infty)}
    \]
    defines a positive log structure $(\cM_{\halfspc{}},\alpha)$ on $[0,\infty)$.  We refer to $[0,\infty)$ with this positive log structure as the \defn{standard half-open interval}.
    \end{example}

For a positive log structure, we will tacitly identify $\Cinfpos{\Sig}$ with the submonoid $\alpha^{-1}(\Cinfpos{\Sig}) \subset  \cM_\Sig$, viewing $\alpha$ as a factorization of the inclusion $\Cinfpos{\Sig} \hookrightarrow\Cinfge{\Sig}$:
\[
\Cinfpos{\Sig}\hookrightarrow \cM_\Sig \stackrel{\alpha}{\longrightarrow} \Cinfge{\Sig}.
\]
Every positive pre-log structure has an \defn{associated positive log structure} $\cM_{\Sig}^{\log}$ defined as the pushout $\cM_{\Sig}^{\log} := \cM_{\Sig} \sqcup_{\alpha^{-1}(\Cinfpos{\Sig})} \Cinfpos{\Sig}$
with the induced morphism $\alpha^{\log }: \cM_{\Sig}^{\log} \to \Cinfge{\Sig}$; see~\cite[\S{}I.1.1 and \S{}III.1.1.3]{Ogus} for the analogous construction in the algebro-geometric setting.  This will be useful for the constructions below.

\subsubsection{Pullback and restriction}

Let $\Psi=(\underline{\Psi},\cM_\Psi,\alpha_\Psi)$ be a manifold with corners equipped with a positive log structure, let $\uSig$ be a manifold with corners, and let $\phi:\uSig\to \underline{\Psi}$ be a smooth map. The composition
\[
\begin{tikzcd}
\phi^{-1}\!\cM_{\Psi} \ar[r,"\alpha_\Psi"] & \phi^{-1}\Cinfge{\Psi} \ar[r,"\phi^*"] & \Cinfge{\Sig}
\end{tikzcd}
\]
defines a positive pre-log structure on $\uSig$, where the second map is the pullback of smooth functions along $\phi$. Its associated log structure
is denoted by $\phi^*(\cM_\Psi,\alpha_\Psi)$ and called the \defn{pullback} of the positive log structure. When $\phi$ is an immersion, we will refer to this operation as \defn{restriction}. We will implicitly equip any open subset $U\subset \underline{\Psi}$ with the restricted positive log structure.

    \begin{example}\label{ex:endpoint}
    The point $\{0\}\subset [0,\infty)$, equipped with the restriction of the positive log structure from \autoref{ex:interval}, is called the \defn{standard end} and denoted by the symbol $[0) = (\{0\},\cM_{[0)},\alpha_{[0)})$.  It behaves somewhat like a tubular neighbourhood of $0\in\halfspc{}$. Let $t:= r|_{[0)]} \in \cM_{[0)}$ denote the restriction of the standard coordinate on $\halfspc{}$. Then we have $\cM_{[0)}=\RR_{>0} t^\NN$, the product of  $\RR_{>0}$ with the free monoid generated by $t$, and $\alpha_{[0)}$ is the morphism of monoids defined by ``evaluation at $t=0$'':
    \[
    \alpha_{[0)}\colon \RR_{>0} t^\NN \to \RR_{\ge 0} \qquad \lambda t^j\mapsto \lambda 0^j \defas \begin{cases} \lambda & \mbox{ if } j=0\\ 0 & \mbox{ if } j>0. \end{cases}\qedhere
    \]
    \end{example}

\subsubsection{Products}

Let $\Sig=(\uSig,\cM_\Sig,\alpha_\Sig)$ and $\Psi=(\underline{\Psi},\cM_\Psi,\alpha_\Psi)$ be manifolds with corners equipped with positive log structures. Their \defn{product} $\Sig\times\Psi$ is the manifold with corners $\uSig\times\underline{\Psi}$ is equipped with the positive log structure associated to the positive pre-log structure 
\[
p_\Sig^{-1}\cM_\Sig \times p_\Psi^{-1}\cM_{\Psi} \longrightarrow \Cinfge{\Sig\times\Psi} \;\; , \;\; (f,g) \mapsto p_\Sig^*\alpha_\Sig(f)\cdot 
p_\Psi^*\alpha_\Psi(g),
\]
where $p_\Sig$ and $p_\Psi$ denote the projections from $\uSig\times\underline{\Psi}$ to $\uSig$ and $\underline{\Psi}$ respectively.

   \begin{example}\label{ex:std-corner}
    Let $n, k \in \NN$.  The \defn{standard log corner of dimension $(n,k)$} is the product $\model{n}{k}$. 
 Explicitly, the positive log structure $\cM_{\model{n}{k}}$  is defined as follows.     
    Let $r_1,\ldots,r_n$ denote the standard coordinates on $\halfspc{n}$, and let $t_1,\ldots,t_k$ be additional formal variables. We define
    \[
    \cM_{\model{n}{k}} \defas \Cinfpos{\halfspc{n}}r^\NN \cdot t^\NN
    \]
    as the product of the constant monoid $t^\NN\defas t_1^\NN\cdots t_k^\NN$ of monomials in the variables $t_1,\ldots,t_k$ with the subsheaf of monoids of $\Cinfge{\halfspc{n}}$ consisting of functions which can be written as a positive smooth function times a monomial in $r^\NN \defas r_1^\NN\cdots r_n^\NN$.
    The morphism of sheaves of monoids
    \[
    \alpha\colon  \Cinfpos{\halfspc{n}}r^\NN \cdot t^\NN \longrightarrow \Cinfge{\halfspc{n}}
    \]
    is defined as the identity on functions of $r$, and sends each $t_i$ to $0$. We note that $[0,\infty)^n\times [0)^k$ can also be obtained by pulling back the positive log structure of the product $[0,\infty)^{n+k}$  to the vanishing locus of the last $k$ coordinates.
    \end{example}

\subsection{Manifolds with log corners}
We now introduce our main objects of study.

    \subsubsection{Definition} Let $\Sig=(\uSig,\cM_\Sig,\alpha)$ be a manifold with corners equipped with a positive log structure.
    \begin{definition}\label{def:chart}
     A \defn{chart of dimension $(n,k)$} on  $\Sig$ is an isomorphism from an open set  $U \subset \Sig$ to an open set in the standard log corner $\model{n}{k}$.   We say that $\Sig$ is a \defn{manifold with log corners} if every point $x \in \Sig$ is contained in the domain of a chart of dimension $(n,k)$ for some $n,k\in\NN$ (which may depend on $x)$.
    \end{definition}
   If all charts have the same dimension $(n,k)$, which is automatic if $\Sig$ is connected, we say that $\Sig$ itself has dimension $(n,k)$.
   
    Thus a manifold with log corners of dimension $(n,k)$ is covered by charts consisting of \defn{basic coordinates} $r_1,\ldots,r_n \in \Cinfge{\Sig}$ on the underlying manifold with corners, and \defn{phantom coordinates} $t_1,\ldots,t_k \in \cM_{\Sig}$ with $\alpha(t_j) =0$.

        \begin{example}
    Every connected manifold with log corners of dimension $(0,k)$ is isomorphic to $[0)^k$.
        \end{example}

        \begin{example}
    If $\Sig_1,\Sig_2$ are manifolds with log corners of dimensions $(n_1,k_1)$ and $(n_2,k_2)$, then their product $\Sig_1\times\Sig_2$ is a manifold with log corners of dimension $(n_1+n_2,k_1+k_2)$.
        \end{example}

    \begin{remark}
    A manifold with corners equipped with the \emph{trivial} log structure (\autoref{ex:trivial}) is not a manifold with log corners, unless it is a manifold (i.e., without corners). Indeed, in our local model $\halfspc{n}$ the log structure is non trivial along the boundary.
    \end{remark}

    \subsubsection{(Ordinary) morphisms}

    There is an obvious notion of morphism between manifolds with log corners, identical to the corresponding notion in logarithmic algebraic geometry. Since we will later need a weaker notion of morphism, we will sometimes call the usual morphisms ``ordinary'' to distinguish them from the more general ``virtual'' morphisms defined below in \autoref{sec:weak-morphisms}.

    \begin{definition}\label{def:morph}
    Let $\Sig = (\uSig,\cM_\Sig,\alpha_\Sig)$  and $\Psi = (\uPsi,\cM_\Psi,\alpha_\Psi)$ be manifolds with log corners. An \defn{(ordinary) morphism} $\phi\colon \Sig \to \Psi$ is a pair $(\underline{\phi},\phi^*)$, where $\underline{\phi}\colon \uSig \to \uPsi$ is a smooth map and $\phi^* \colon \underline{\phi}^{-1}\cM_{\Psi} \to \cM_\Sig$ is a morphism of sheaves of monoids such the following diagram commutes, where the bottom horizontal arrow is the usual pullback of smooth functions along $\underline{\phi}$.
    \begin{equation}\label{eq:ordinary-commutativity}
    \begin{tikzcd}
    \underline{\phi}^{-1}\cM_{\Psi} \ar[r,"\phi^*"]\ar[d,"\alpha_\Psi"] & \cM_{\Sig} \ar[d,"\alpha_\Sig"] \\
    \underline{\phi}^{-1}\Cinfge{\Psi} \ar[r,"\underline{\phi}^*"] \ar[r] & \Cinfge{\Sig}
    \end{tikzcd}
    \end{equation}
    \end{definition}

    \begin{example}
    An example of morphism of manifolds with log corners is the inclusion $i\colon [0)\hookrightarrow [0,\infty)$ of the standard end (\autoref{ex:endpoint}) inside the standard half-open interval (\autoref{ex:interval}). Concretely, the smooth map $\underline{i}$ is the inclusion of $\{0\}$ inside $[0,\infty)$, and the  morphism of sheaves of monoids is
    \[
    \mapdef{i^*}{\underline{i}^{-1}\Cinfpos{[0,\infty)}r^\NN}{\RR_{>0}t^\NN}
    {g(r)r^j}{g(0)t^j,}
    \]
    which picks out the leading term in the Taylor expansion of a function at $r=0$.
    \end{example}

    \begin{example}
    More generally, consider the standard corner $\halfspc{n}$ with coordinates $r_1,\ldots,r_n$. We have a morphism of manifolds with log corners
    \[
    i\colon\halfspc{n-1}\times [0) \hookrightarrow \halfspc{n}
    \]
    whose underlying smooth map $\underline{i}$ is the inclusion of $[0,\infty)^{n-1}$ as the locus $\{r_n=0\}$ in $[0,\infty)^n$ and whose morphism of sheaves of monoids extracts the leading Taylor monomial in $r_n$:
    \[
    \mapdef{i^*}{\underline{i}^{-1}\Cinfpos{\halfspc{n}} r_1^\NN\cdots r_n^\NN}{\Cinfpos{\halfspc{n-1}}r_1^\NN\cdots r_{n-1}^\NN\cdot t^\NN}
{g(r_1,\ldots,r_n)r_1^{j_1}\cdots r_n^{j_n}}{g(r_1,\ldots,r_{n-1},0)r_1^{j_1}\cdots r_{n-1}^{j_{n-1}}\, t^{j_n}.}
    \]
    More generally, we have inclusion morphisms of manifolds with log corners
    \[
    i \colon \model{n}{k} \hookrightarrow \model{n+j}{k-j}
    \]
    for $0 \le j \le k$.
    \end{example}

    \begin{example}
    If $\Sig_1$, $\Sig_2$ are manifolds with log corners, then the projections $p_{\Sig_1}\colon\Sig_1\times \Sig_2\to \Sig_1$ and $p_{\Sig_2}\colon\Sig_1\times\Sig_2\to \Sig_2$ are morphisms of manifolds with log corners.
    \end{example}

    \subsubsection{Manifolds with corners viewed as manifolds with log corners}

    If $\uSig$ is a manifold with corners, we may endow it with the structure of a manifold with log corners $\Sigbas$ by requiring that every chart $\phi\colon U\to [0,\infty)^n$ on $\uSig$ in the sense of \autoref{def:mfld-with-corners} is a chart of dimension $(n,0)$ in the sense of \autoref{def:chart}. More precisely, let  $\cM_{\Sigbas} \subset \Cinfge{\Sig}$ be the subsheaf of monoids consisting of functions that can be written in local coordinates $(r_1,\ldots,r_n)$ in the form $g(r) r_1^{j_1}\cdots r_n^{j_n}$ where $g \in \Cinfpos{\Sig}$ and $j_1,\ldots,j_n\in\NN$, and let $\alpha_{\Sigbas} : \cM_{\Sigbas} \to \Cinfge{\Sig}$ be the natural inclusion.

    \begin{definition}\label{def:Sigbas}
    The manifold with log corners $\Sigbas = (\uSig,\cM_{\Sigbas},\alpha_{\Sigbas})$ is the \defn{basic 
    manifold with log corners associated to $\uSig$}.
    \end{definition}

    Note that if a smooth map $\uphi:\uSig \to \uPsi$ of manifolds with corners lifts to an ordinary morphism $\Sigbas \to \Psibas$, it does so in at most one way; this occurs if and only if, when expressed in local coordinates, $\uphi$ has the form
    \[
    \uphi \colon (r_1,\ldots,r_n) \longmapsto \rbrac{f_1 \prod_{i=1}^n r_i^{j_{i,1}},\ldots, f_{m} \prod_{i=1}^n r_i^{j_{i,m}}}
    \]
    for some strictly positive  smooth functions $f_1,\ldots,f_{m}$ of $(r_1,\ldots,r_n)$ and exponents $(j_{i,l}) \in \NN^{n\times m}$, i.e.~it is an ``interior b-map'' in the sense of Melrose~\cite{Melrose1992}.  Note that in this case, $\uphi$ preserves interiors, in the sense that $\uphi(\uSig^\circ) \subset \uPsi^\circ$.

\subsubsection{Tangent structure of a manifold with log corners}\label{sec:tangent-log-corners}
If $x\in \Sig$ is a point in a manifold with log corners, there is a natural structure of manifold with log corners on the non-negative tangent space $\tbge[x]{\uSig}$ of the underlying manifold with corners, that we denote by $\tbge[x]{\Sig}$. Concretely, in a chart of dimension $(n,k)$ whose coordinates satisfy \eqref{eq:good-coords}, we have an isomorphism
\[
\tbge[x]{\Sig} \cong [0,\infty)^j\times \RR^{n-j}\times [0)^k
\]
of manifolds with log corners, whose coordinates are the differentials $\dd r_1|_x,\ldots,\dd r_n|_x$ and $\dd t_1|_x,\ldots,\dd t_k|_x$. More intrinsically, the positive log structure of $\tbge[x]{\Sig}$ is the pullback of the positive log structure on $\Sig$ via any open embedding $\tbge{\uSig} \hookrightarrow \uSig$ with linearization equal to the identity.  This construction is evidently functorial: for $\phi\colon \Sig\to \Psi$ a morphism of manifolds with log corners, the derivative of $\phi$ gives at $x$ gives a morphism
$$\dd \phi_x\colon T_x^{\ge 0}\Sig \to T_{\phi(x)}^{\ge 0}\Psi.$$

\subsection{The global structure of manifolds with log corners}

\subsubsection{``Basic'' and  ``phantom'' sections}

 For a manifold with log corners $\Sig$, there is a fundamental dichotomy in the behaviour of sections $f \in \cM_{\Sig}$, generalizing the dichotomy between the basic coordinates $r_i$ and the phantom coordinates $t_j$:

    \begin{definition}
    Let $\Sig=(\uSig,\cM_\Sig,\alpha)$ be a manifold with log corners. A germ $f$ of a section of $\cM_\Sig$ is called \defn{basic} if $\alpha(f)\neq 0$ and a \defn{phantom} if $\alpha(f)=0$. We denote by $\cMbas_\Sig,\cMphan_\Sig \subset \cM_\Sig$ the subsheaves of basic and phantom sections.
    \end{definition}

\begin{example}
For the standard end $[0)$ from \autoref{ex:endpoint}, an element of $\cM_{[0)}$ has the form $\lambda t^j$ with $\lambda\in\RR_{>0}$ and $j\in\NN$. It is basic if $j=0$ and phantom if $j>0$, so that $\cMbas_{[0)}=\RR_{>0}$ and $\cMphan_{[0)}$ is the monoid ideal of $\cM_{[0)}$ generated by $t$.
\end{example}
An important subtlety in the definition is that the condition $\alpha(f) \neq 0$ refers to the germ of the function $\alpha(f)$, not the value at a point; thus it could well happen that the germs of $\alpha(f)$ are everywhere nonzero, so that $f$ is basic, even though the function $\alpha(f)$ has a nonempty vanishing set.

\begin{example}
For the standard interval half-open interval $\halfspc{}$ from \autoref{ex:interval}, a section of $\cM_{\halfspc{}}$ has the form $g(r)r^j$ with $g$ a positive smooth function and $j\in\NN$. Even though such a function vanishes at the origin when $j>0$, its germ is everywhere nonzero. Therefore, $\cMbas_{\halfspc{}}=\cM_{\halfspc{}}$ and $\cMphan_{\halfspc{}}=\varnothing$.
\end{example}

    Note that by definition, the stalk of $\cM_\Sig$ at any point decomposes as the disjoint union of the sets of basic and phantom elements, i.e.~we have a decomposition
    \[
    \cM_\Sig = \cMbas_\Sig \sqcup \cMphan_\Sig
    \]
    of sheaves of sets.  In local coordinates $(r,t)$, a phantom section is one that is a multiple of some phantom coordinate $t_j$, while the basic sections are given by
    $$\cMbas_{\model{n}{k}} =  \Cinfpos{\halfspc{n}}r^\NN.$$
    Therefore $\cMbas_\Sig\subset \cM_\Sig$ is a sheaf of submonoids while $\cMphan_\Sig\subset \cM_\Sig$ is a sheaf of monoid ideals.

For a manifold with corners $\uSig$, the associated basic manifold with log corners $\Sigbas$ (\autoref{def:Sigbas}) does not have phantoms, i.e.~the sheaf $\cMphan_{\Sigbas}$ is empty.  On the other hand, if $\Sig$ is \emph{any} manifold with log corners, then it is clear from the definition of a chart that the subsheaf $\cMbas_{\Sig}\subset \cM_\Sig$ of basic elements depends only on the underlying manifold with corners, and is canonically identified with $\cM_{\Sigbas}$ via $\alpha$.  Hence there is a canonical ordinary morphism
\[
\Sig \to \Sigbas
\]
to the underlying basic manifold with log corners. In local coordinates, it is simply the projection $\halfspc{n}\times[0)^k\to \halfspc{n}$. 

\begin{definition}
    A manifold with log corners is \defn{basic} if $\cMbas_{\Sig} = \cM_{\Sig}$, or equivalently the canonical map $\Sig \to \Sigbas$ is an isomorphism.
\end{definition}

\subsubsection{The boundary of a manifold with log corners}

Let $\Sig = (\uSig,\cM_\Sig,\alpha_\Sig)$ be a manifold with log corners. Then the pullback of the positive log structure on $\uSig$ to the boundary $\ubdSig$ (defined as in \autoref{sec:boundary-manifolds-with-corners}) gives the latter the structure of a manifold with log corners, and makes the boundary immersion into a morphism
\[
i\colon\partial\Sig\to \Sig.
\]
 whose local picture is given by the following example.  Note that if $\Sig$ has dimension $(n,k)$, then $\partial\Sig$ has dimension $(n-1,k+1)$; in particular, the boundary is not basic unless it is empty.

\begin{example}
For $\Sig=\halfspc{}$ the standard half-open interval (\autoref{ex:interval}), the positive log structure on $\partial\halfspc{}$ is the standard end $[0)$. More generally, one checks that for the standard log corner $\model{n}{k}$ we have an identification 
\[
\partial(\model{n}{k}) = \bigsqcup_{j=1}^n \, \halfspc{j-1}\times [0) \times \halfspc{n-j}\times [0)^k,
\]
which is the disjoint union of $n$ copies of $\model{n-1}{k+1}$.
\end{example}

Iterating this construction as in \autoref{sec:boundary-manifolds-with-corners}, we obtain manifolds with log corners $\bdSig[k]$ and  $\overline{\partial}^k\Sig = \bdSig[k]/\symgp{k}$ for $k \ge 0$, so that every face of $\Sig$ is also a manifold with log corners.  In fact, if $\Sig=\Sigbas$ is a basic manifold with corners of dimension $n$, the induced log structure on $\bdSig[k]$ depends only on the structure of its normal bundle, as we now explain.

Let us denote by
\[
i \colon \ubdSig[k] \to \uSig
\]
the natural immersion, and by
\[
N^{\ge0} \defas i^*\tbge{\uSig}/\tb{\ubdSig[k]}
\]
its non-negative normal bundle.  Concretely, near any point $p \in \ubdSig$ we can choose coordinates $(r_1,\ldots,r_n)$ in which the immersion $\ubdSig[k] \to \uSig$ is given by the inclusion of the boundary face $r_1=r_2=\cdots=r_k=0$ in $[0,\infty)^n$.  Then the fibre of $N^{\ge0}$ at $p$ is identified with the set of non-negative linear combinations of the basis normal vectors
\[
[\cvf{r_1}],\ldots,[\cvf{r_k}] \in i^*\tb{\uSig}/\tb{\ubdSig[k]}.
\]
We may then view $N^{\ge 0}$ as a basic manifold with log corners with induced coordinates $(\dd r_1,\ldots, \dd r_k, r_{k+1},\ldots,r_{n})$,
where we abuse notation and view the differentials $\dd r_1,\ldots,\dd r_k$ as functions on the normal bundle, and the functions $r_{k+1},\ldots, r_n$ on $\uSig$ as functions on the base $\ubdSig[k]$ by pullback.

We have the zero section embedding
\[
i_0 : \ubdSig \to N^{\ge 0},
\]
given locally by the inclusion of the face $\dd r_1=\cdots = \dd r_k = 0$, so that we may pull back the basic positive log structure on $N^{\ge 0}$ to make $\ubdSig[k]$ into a manifold with log corners $\widehat{\bdSig[k]}$, with induced coordinates $(\dd t_1,\ldots, \dd t_k,r_{k+1},\ldots, r_n)$, where the coordinates $\dd t_j \defas i_0^* \dd r_j$ are phantoms.  By construction, $\widehat{\bdSig[k]}$ depends only on the non-negative normal bundle $N^{\ge0}$.

On the other hand, the pullback log structure $i^*\cM_{\Sig}$ defining $\bdSig[k]$ as before has coordinates $(t_1,\ldots,t_k,r_{k+1},\ldots,r_n)$ with phantoms $t_j := i^* r_j$.  It is straightforward to check that the identifications
\begin{align}(t_1,\ldots,t_k,r_{k+1},\ldots,r_n)\leftrightarrow 
 (\dd t_1,\ldots, \dd t_k,r_{k+1},\ldots,r_n)   \label{eq:symbol-coords}
\end{align}
are independent of the choice of coordinates $(r_1,\ldots,r_n)$ on $\Sig$, and therefore gives a canonical isomorphism $\widehat{\bdSig[k]} \cong \bdSig[k]$.

More invariantly, we may note that sections of $\cM_{\widehat{\bdSig[k]}}$ are identified with fibrewise monomial functions on $N^{\ge0}$, with coefficients in the basic functions $\cMbas_{\bdSig[k]}$. Furthermore,  if $f \in \cM_{\Sig}$ is a locally monomial function on $\Sig$, the leading term in its Taylor expansion  along $\bdSig[k]$ is exactly such a locally monomial function on the normal space. This gives a map of positive pre-log structures
\begin{align}
\sigma \colon
i^{-1} \cM_{\Sig} \to \cM_{\widehat{\bdSig[k]}} \label{eq:symbol-map}
\end{align}
which we call the \defn{symbol map}.  It has the local expression
\[
f(r_1,\ldots,r_n)r_1^{j_1}\cdots r_n^{j_n} \mapsto f(0,\ldots,0,r_{k+1},\ldots,r_n) (\dd t_1)^{j_1} \cdots (\dd t_k)^{j_k} \cdot r_{k+1}^{j_{k+1}}\cdots r_n^{j_n}
\]
and hence, upon logification, it induces the coordinate mapping \eqref{eq:symbol-coords} above.

\begin{proposition}\label{prop:normal-cone}
    For $\Sig=\Sigbas$ basic, the symbol map induces an isomorphism 
    \[
    \cM_{\bdSig[k]} \cong \cM_{\widehat{\bdSig[k]}}.
    \]
    of positive log structures, for all $k \ge 0$.  Hence the log structure on $\bdSig[k]$ depends only on the non-negative normal bundle of $\ubdSig[k]$ in $\uSig$.
    \end{proposition}

\begin{corollary}\label{cor:normal-monoid}
For every $x \in \Sig$, the pullback log structure $\cM_{\Sig}|_x$ is canonically identified with the multiplicative monoid of non-negative monomial functions on the non-negative normal space $N^{\ge0}_x\Sig$.
\end{corollary}

\subsubsection{The phantom tangent bundle}\label{sec:phantom-tangent}

\autoref{prop:normal-cone} shows that the manifolds with log corners that arise as boundary faces of a basic manifold with log corners have a special form: they are determined by a vector bundle (the normal bundle) equipped with a subbundle of orthants.  Our aim now is to show that in fact, every manifold with log corners has such a form, and may thus be viewed as a union of faces of a basic manifold with log corners in a canonical way. For this, we need to replace the normal bundle of the embedding with the following intrinsic notion.
 \begin{definition}
 Let $\Sig$ be a manifold with log corners, and let $x \in \Sig$ be a point.  A \defn{phantom tangent vector at $x$} is a map 
     \[
     v \colon \cMphan_{\Sig,x} \to \RR
     \]
     that is $\cM_{\Sig,x}$-linear, in the sense that
     \[
     v(fg) = \alpha(f)|_x \cdot v(g)
     \]
     for all $f \in \cM_{\Sig,x}$ and $g \in \cMphan_{\Sig,x}$. 
 A phantom tangent vector is \defn{non-negative} (resp.~\defn{positive}) if it takes values in $\RR_{\ge0}$ (resp.~$\RR_{>0}$).
 \end{definition}

  The set of all phantom tangent vectors on $\Sig$ is naturally a vector bundle over $\uSig$, which we call the \defn{phantom tangent bundle $\tbphan{\Sig}$}.  The positive and non-negative vectors give subbundles
\[
\tbposphan{\Sig}\subset \tbgephan{\Sig} \ \subset\  \tbphan{\Sig}
\]
with fibres isomorphic to $(0,\infty)^k \subset \halfspc{k}\subset \RR^k$.  Concretely, if $t_1,\ldots,t_k \in \cMphan_{\Sig,x}$  is a set of phantom coordinates we may define phantom tangent vectors  $\cvf{t_1},\ldots,\cvf{t_k} \in \tbphan[x]{\Sig}$ by $\cM_{\Sig,x}$-linear extension of the formula
 \[
 \cvf{t_i}(t_j) = \begin{cases}
     1 & i=j \\ 
     0 & i \neq j.
 \end{cases}
 \]
 Using that $\cMphan_{\Sig,x}$ is freely generated over $\cM_{\Sigbas,x}$ by $t_1,\ldots,t_k$, it is straightforward to verify that these phantom tangent vectors are well-defined, and form a local basis for $\tbphan{\Sig}$ such that the positive (resp.~non-negative) vectors are the positive (resp.~non-negative) linear combinations of the basis elements.
 
In particular, $\tbgephan{\Sig}$ is a manifold with corners, which we equip with its basic positive log structure.  This has the effect of turning local phantom coordinates $t_1,\ldots,t_k$ around a point $x\in \Sig$ into actual coordinates on the fibres of $\tbgephan{\Sig}$.  Consequently, the zero section lifts canonically to an embedding $\Sig \to \tbgephan{\Sig}$, identifying the positive log structure on $\Sig$ with the pullback of the basic positive log structure on $\tbgephan{\Sig}$.  The global structure of a manifold with log corners is thus summarized by the following statement.
\begin{proposition}\label{prop:embedding-in-phantom-tangent}
    Let $\Sig$ be a manifold with log corners.  Then $\Sig$ is canonically identified with the zero section in the basic manifold with log corners $\tbgephan{\Sig}$, with the induced log structure, giving a commutative diagram
    \[
    \begin{tikzcd}
        \Sig \ar[rr,hook]\ar[rd] && \tbgephan{\Sig}\ar[ld] \\
        & \Sigbas
    \end{tikzcd}
    \]
\end{proposition}

Note that by construction, every transition function for $\tbgephan{\Sig}$ has to map a phantom coordinate $t_j$ to $\lambda t_{j'}$ for some $\lambda>0$ and some index $j'$.  We therefore get a reduction of the structure group of the vector bundle $\tbphan{\Sig}$ from $\GL{k}{\RR}$ to the subgroup $S_k \ltimes \RR_{>0}^{k}$ generated by permutation matrices and diagonal matrices with positive entries. This means that locally, $\tbphan{\Sig}$ has a canonical decomposition as a sum of line bundles corresponding to the phantom coordinates $t_1,\ldots,t_k$, but this decomposition need not be globally well-defined, as the following example shows.

\begin{example}
    Let $\Sig_0 = \RR \times [0)^2$ and consider the $\ZZ$-action on $\Sig_0$ generated by the automorphism $(x,t_1,t_2) \mapsto (x+1,t_2,t_1)$, which acts as a translation on the underlying manifold, an swaps the phantom coordinates.  Then $\Sig\defas \Sig_0/\ZZ$ is a manifold with log corners whose underlying manifold is the circle $\RR/\ZZ$, and for which the pair of lines decomposing the fibres of $\tbphan{\Sig}$ are interchanged as we go around the circle.
\end{example}

\subsection{Fixed points of group actions} In this subsection, we briefly discuss the behaviour of compact Lie group actions on manifolds with log corners.   

\begin{definition}
Let $\Sig$ be a manifold with log corners, and $G$ a Lie group, which we view as a manifold with corners equipped with the trivial positive log structure.  An \defn{action of $G$ on $\Sig$} is a morphism $G \times \Sig \to \Sig$ of manifolds with log corners satisfying the unit and associativity conditions in the category of manifolds with log corners.
\end{definition}

 If $x \in \uSig$ is a $G$-fixed point, then by functoriality, we obtain a linear action
\[
G \times \tbge[x]{\Sig} \to \tbge[x]{\Sig}
\]
on the non-negative tangent space, viewed as a manifold with log corners as in \autoref{sec:tangent-log-corners}.  When $G$ is compact, this gives a local model for the action, thanks to the following mild extension of Bochner's linearization theorem~\cite{Bochner1945}.
\begin{theorem}
    If $G$ is compact, there exists a $G$-equivariant isomorphism of manifolds with log corners from an open neighbourhood of the fixed point $x \in \Sig$ to an open neighbourhood of $0 \in \tbge[x]{\Sig}$, whose derivative at $x$ is the identity map.
\end{theorem}
\begin{proof}
    The proof is a straightforward adaptation of Bochner's argument~\cite{Bochner1945} as presented in \cite[Section 2.2]{DuistermaatKolk}; we will simply indicate the necessary adjustments.

    By functoriality, the action of $G$ lifts to a fibre-wise linear action on the phantom tangent bundle preserving the embedding $\Sig\hookrightarrow \tbgephan{\Sig}$ from \autoref{prop:embedding-in-phantom-tangent}.  We may thus assume without loss of generality that $\Sig = \Sigbas$ is basic.
    
    From here, the proof prooceeds as in \emph{op.~cit.}. By compactness of $G$, there exists a $G$-invariant open neighbourhood $U$ of $x$ in $\Sig$. 
    Choose an arbitrary open embedding $\chi\colon U \hookrightarrow \tbge[x]{\Sig}$ whose derivative at $x$ is the identity.  Let $\chi_g \defas g \chi g^{-1}$ be the conjugation of $\phi$ by the given action of $g$ on $U$ and the induced linear action on $\tbge[x]{\Sig}$.  We may view each map $\chi_g$ as a vector in the vector space of smooth maps $U \to \tb[x]{\Sig}$. 
    The latter is a complete locally convex topological vector space (cf.~\autoref{rem:Ck-topology}), so that we may define the average $\overline{\chi} \defas \int_{g \in G} \phi_g$ with respect to the Haar measure on $G$.  As in \emph{op.~cit.}, $\overline{\chi}$ is a smooth, $G$-invariant map $U \to \tb[x]{\Sig}$ whose derivative at $x$ is the identity.  In addition, since $\tbge[x]{\Sig}\subset\tb[x]{\Sig}$ and all of its boundary faces are preserved by non-negative linear combinations, we conclude that the image of $\overline{\chi}$ is an open neighbourhood of the origin in $\tbge[x]{\Sig}$.  Thus $\overline{\chi}$ gives the desired $G$-equivariant isomorphism.  
\end{proof}

\begin{corollary}\label{cor:local-action}
    If $G$ is compact, then near any fixed point, $\Sig$ is $G$-equivariantly isomorphic to a product $\RR^l \times \model{n}{k}$ with a diagonal $G$-action, where $G$ acts linearly on $\RR^l$, and acts by permutations of the coordinates on $\model{n}{k}$.
\end{corollary}

\begin{proof}
    Endowing the tangent space with an invariant inner product, we can decompose it as an orthogonal $G$-invariant direct sum of the tangent space of the stratum through $x$ and its normal directions.  The former is a copy of $\RR^m$ with a linear $G$-action, and the latter are grouped into basic and phantom directions for which the action reduces to permutations by \autoref{lem:finite-symmetries} below.
\end{proof}

\begin{lemma}\label{lem:finite-symmetries}
    Let $G \subset \GL{n+k}{\RR}$ be a group of linear transformations of $\RR^{n+k}$ that restrict to automorphisms of the manifold with log corners $\model{n}{k}$.  Then $G$ lies in the subgroup $(\symgp{n}\times\symgp{k})\ltimes \RR_{>0}^{n+k}$ generated by permutations of the first $n$ and last $k$ coordinates, and the diagonal matrices with positive eigenvalues.  If, in addition, $G$ is compact, then $G$ is conjugate to a subgroup of $\symgp{n}\times\symgp{k}$.
\end{lemma}

\begin{proof}
Let $e_1,\ldots,e_{n+k} \in \RR^{n+k}$ be the standard basis.  
The action of $G$ on $\RR^{n+k}$ must permute the one-dimensional boundary faces of $\halfspc{n+k}$, but the latter are exactly the rays spanned by the basis vectors.  Hence any $g \in G$ acts by $g \cdot e_i = \lambda_i e_{\sigma^{-1}(i)}$ for some $\lambda_i > 0$ and $\sigma \in \symgp{n+k}$. 
 Furthermore, the action must preserve the grouping of the coordinates into the $n$ basic and $k$ phantom coordinates, so that $G < (\symgp{n}\times\symgp{k})\ltimes \RR_{>0}^{n+k}$ as claimed.
 
Now suppose that $G$ is compact; we wish to show that $G$ is conjugate to a subgroup of $\symgp{n}\times\symgp{k}$.  Since $\symgp{n}\times\symgp{k}$ is a subgroup of $\symgp{n+k}$ it is enough to treat the case $k=0$. We first treat the case where the action of $G$ on $\{1,\ldots,n\}$, induced by the map $G\to \symgp{n}$, is transitive. In this case, for all $i\in\{1,\ldots,n\}$, there exists an element $g_i\in G$ such that $g_i \cdot e_1=\alpha_ie_i$ for some $\alpha_i\in \RR_{>0}$, where we assume that $g_1=\id{}$ and $\alpha_1=1$. Performing the change of basis given by $e'_i=\alpha_i e_i$ now gives $g_i\cdot e'_1=e'_i$ for all $i\in\{1,\ldots,n\}$. Let $g\in G$ and let us write $g\cdot e'_i=\lambda e'_j$. Now the element $g_ig_j^{-1}g\in G$ sends $e'_i$ to $\lambda e'_i$. Since $G$ is a compact subgroup of $\symgp{n}\ltimes \RR^n_{>0}$, this element generates a compact subgroup and therefore $\lambda=1$. We conclude that $G$ acts by permuting the basis elements $e'_1,\ldots,e'_n$. In other words, the conjugate subgroup $\underline{\alpha} G\underline{\alpha}^{-1}$ is contained in $\symgp{n}$. For the general case (with $k=0$), the same argument on each orbit of the action of $G$ on $\{1,\ldots,n\}$ gives the claim.
\end{proof}

\begin{corollary}
    The fixed point set $\uSig^G \subset \uSig$ is an embedded submanifold with corners.
\end{corollary}

\begin{proof}
    By the previous corollary, the fixed point set decomposes locally as the product of a linear subspace of $\RR^l$ and diagonals in $\halfspc{n}$. The former is an ordinary smooth manifold, and the latter are products of corners, with coordinates given by restricting subsets of the coordinates on $[0,\infty)^n$. 
\end{proof}

We equip the manifold with corners $\uSig^G$ with a positive log structure $\cM_{\Sig^G}$  as follows.  Let $i\colon \Sig^G \hookrightarrow\Sig$ be the inclusion, and note that the pullback log structure $i^*\cM_{\Sig}$ carries a residual action of $G$ by automorphisms, for which the morphism $\alpha\colon i^*\cM_{\Sig} \to \Cinfge{\Sig^G}$ is invariant.  We may therefore take the quotient of the monoid by conguence generated by the $G$-action (see \cite[\S{}I.1.1]{Ogus}) to obtain the following.
\begin{definition}
     The \defn{fixed locus $\Sig^G$} is the fixed point set equipped with the positive log structure $\cM_{\Sig^G}\defas (i^*\cM_{\Sig})/G$.
\end{definition}

\begin{example}\label{ex:quotient-basic}
Consider the symmetric group $\symgp{n}$ acting on the manifold with log corners $\halfspc{n}$ by permuting the coordinates $r_1,\ldots,r_n$.  The fixed-point set is the image of the diagonal $i\colon[0,\infty)\hookrightarrow\halfspc{n}$. The pullback log structure is given by
\[
\mapdef{\alpha}{i^*\cM_{\halfspc{n}}}{\Cinfge{\Sig}}
{g(r)r_1^{j_1}\cdots r_n^{j_n}}{g(r)r^{j_1+\cdots+j_n}}
\]
where $r$ is the coordinate on  $\halfspc{}$ and $g \in \Cinfpos{\halfspc{}}$.  This identifies the quotient $i^*\cM_{\halfspc{n}}/\symgp{n}$ with the standard positive log structure on $[0,\infty)$.
\end{example}

\begin{example}\label{ex:quotient-phantom}
Consider the symmetric group $\symgp{k}$ acting on the manifold with log corners $[0)^k$ by permuting the (phantom) coordinates. The inclusion $i$ of the fixed-point set is the identity of $\{0\}$, and $i^*\cM_{[0)^k}$ is the monoid consisting of monomials $\lambda t_1^{j_1}\cdots t_k^{j_k}$ with $\lambda\in\RR_{>0}$ and $j_1,\ldots,j_k\in\NN$, with the action of $\symgp{k}$ permuting the (phantom) coordinates. The quotient $i^*\cM_{[0)^k}/\symgp{k}$ is therefore identified with the standard positive log structure on $[0)$.
\end{example}

The following result shows that $\Sig^G$ is the categorically correct version of the fixed point set. 

\begin{theorem}\label{thm:fixed-locus}
    Let $G$ be a compact Lie group acting on a manifold with log corners $\Sig$.  Then $\cM_{\Sig^G}$ gives $\Sig^G$ the structure of a manifold with log corners and $i\colon \Sig^G\hookrightarrow \Sig$ is a morphism of manifolds with log corners. Every $G$-invariant morphism $\Psi \to \Sig$ of manifolds with log corners factors uniquely through $i$.
\end{theorem}

\begin{proof}
    For any fixed point, we may choose a chart on $\Sig$ as in \autoref{cor:local-action}.  The action of $G$ therefore partitions the coordinates on $\model{n}{k}$ into orbits, and the first claim follows from \autoref{ex:quotient-basic} and \autoref{ex:quotient-phantom}.
    
    For the universal property, note that any $G$-invariant map of manifolds with log corners must map entirely to fixed points, and therefore factor through the pullback positive log structure $i^*\cM_\Sig$.  But the map on monoids is also $G$-invariant, and hence it must further factor through $(i^*\cM_\Sig)/G$.
\end{proof}

\section{Virtual morphisms, tangential basepoints, and scales}\label{sec:weak-morphisms}

\subsection{Definition and examples}

    The notion of (ordinary) morphism between manifolds with log corners is too restrictive for some purposes, and we will need a weaker notion, as follows. For a sheaf of (commutative) monoids $\cM$, recall that its \defn{group completion} $\cMgp$ is the universal sheaf of (abelian) groups equipped with a homomorphism from $\cM$. For the standard log corner of dimension $(n,k)$ from \autoref{ex:std-corner},
    $$\cMgp_{\model{n}{k}} = \Cinfpos{\halfspc{n}}r^{\ZZ}\cdot t^{\ZZ}$$
    is the product of the constant abelian group $t^{\ZZ}:=t_1^{\ZZ}\cdots t_k^{\ZZ}$ of Laurent monomials in the variables $t_1,\ldots,t_k$ with the subsheaf of groups of $\Cinfpos{(0,\infty)^n}$ consisting of functions which can be written as a positive smooth function on $\halfspc{n}$ times a Laurent monomial in $r^{\ZZ}:=r_1^{\ZZ}\cdots r_n^{\ZZ}$. This description implies that for every manifold with log corners $\Sig$, the natural morphism $\cM_\Sig\to \cMgp_\Sig$ is injective.

    The following definition is inspired by Howell's analogous notion in logarithmic algebraic geometry \cite{Howell2017}; see \cite{DPP:WeakMorphisms} for more details.
        
    \begin{definition}\label{def:weak-morph}
    Let $\Sig = (\uSig,\cM_\Sig,\alpha_\Sig)$  and $\Psi = (\uPsi,\cM_\Psi,\alpha_\Psi)$ be manifolds with log corners. A \defn{virtual morphism} $\phi\colon\Sig\to \Psi$ is a pair $(\underline{\phi},\phi^*)$ where $\underline{\phi}\colon \uSig \to \uPsi$ is a smooth map and $\phi^*\colon \underline{\phi}^{-1}\cMgp_\Psi \to \cMgp_\Sig$ is a morphism of sheaves of groups such that the following diagram commutes.
    \[
    \begin{tikzcd}
        \uphi^{-1}\Cinfpos{\Psi} \ar[r,"\uphi^*"] \ar[d,hook] & \Cinfpos{\Sig} \ar[d,hook] \\
        \uphi^{-1}\cMgp_{\Psi} \ar[r,"\phi^*"] & \cMgp_{\Sig}
    \end{tikzcd}
    \]
\end{definition}

    \begin{remark}
    By the universal property of group completion, the datum of $\phi^*$ is equivalent to the datum of a morphism $\uphi^{-1}\cM_\Psi \to \cMgp_\Sig$.
    \end{remark}

    There are two notable differences between the notions of ordinary and virtual morphisms, illustrated by examples below.
    \begin{enumerate}
    \item Firstly, the notion of virtual morphism does not require compatibility of pullback with the maps $\alpha$, only with positive smooth functions. For an ordinary morphism $\phi \colon \Sig \to \Psi$, the commutativity of \eqref{eq:ordinary-commutativity} implies that every phantom $f \in \uphi^{-1}\cM_\Psi$ pulls back to a phantom $\phi^*f \in \cM_\Sig$, and more generally if $f \in \uphi^{-1}\cM_\Psi$ is such that $\uphi^*\alpha_\Psi(f)= 0$ (i.e.\ $f$ is a phantom ``relative to $\uphi$''), then $\phi^*f\in\cM_\Sig$ is a phantom.  For a virtual morphism, this is no longer true and therefore virtual morphisms are allowed to ``breathe life into phantoms'' by pulling them back to non-phantoms---but the axiom asserts that they cannot change the values assigned to positive smooth functions.
    \item Secondly, the pullback happens at the level of the group completions of the sheaves of monoids and produces formal quotients of monoid elements which are no longer associated to functions on the underlying manifold with corners---this explains the terminology ``virtual'' (as in ``virtual representation of a group''). For instance, for the standard end $[0)$ from \autoref{ex:endpoint}, the monoid element $t$ corresponds to the function $\alpha(t)=0$ on the point, but its inverse $t^{-1}$ does not correspond to any function.
    \end{enumerate}

    By definition, virtual morphisms satisfy the equality
    \begin{align}
    \phi^*f = \uphi^*f. \label{eq:weak-mor-equality}
    \end{align}
    for all positive smooth functions $f \in \Cinfpos{\Psi}$.  In fact, the equality holds more generally thanks to the following proposition, that we call the ``continuity principle'', and whose proof we delay to the end of this subsection:

    \begin{proposition}\label{prop:weak-mor-basic-sections}
        If $\phi$ is a virtual morphism then \eqref{eq:weak-mor-equality} holds for all $f \in \cMbas_{\Psi}$ such that $\uphi^*f  \neq 0$.
        In other words, for all $f\in\cM_\Psi$ such that $\uphi^*\alpha_\Psi(f)$ is not the zero function, we have that $\phi^*f\in \cM_\Sig$ and
        $$\uphi^*\alpha_\Psi(f) = \alpha_\Sig(\phi^*f).$$
    \end{proposition}

With the obvious notion of composition, virtual morphisms form a category.  Every ordinary morphism is a virtual morphism, but not conversely.  The following examples illustrate the key similarities and differences between these notions.

    \begin{example}
    Let $*$ denote the point equipped with the trivial positive log structure.  Then for any $\Sig$, the projection $\Sig \to *$ is a (ordinary or virtual) morphism in a unique way.  Hence $*$ is a terminal object in the category of either ordinary or virtual morphisms.
    \end{example}

\begin{example}\label{ex:weak-mor-pt-to-end}
    Let $*$ be the point as in the previous example, and let $[0)$ be the standard end, given by the point with the positive log structure $\cM_{[0)} = \RR_{>0}t^\NN$ from \autoref{ex:endpoint}.  There is a unique map of the underlying manifolds $\underline{*}\cong\underline{[0)}$ since both consist of a single point. A virtual morphism 
    \[
    s\colon * \to [0)
    \]
    is determined by a morphism of monoids $s^*\colon\RR_{>0}t^\NN \to \RR_{>0}$ which acts as the identity on $\RR_{>0}$, but may send the generator $t$ to an arbitrary positive real number.  Thus virtual morphisms $s\colon * \to [0)$ are in bijection with $\RR_{>0}$, and are all (virtual) sections of the unique morphism $p\colon [0)\to *$. Note, however, that there are no ordinary morphisms from $*$ to $[0)$.  Indeed, since $\alpha(t)=0$, an ordinary morphism would have to send $t$ to a phantom in $\RR_{>0}$, of which there are none.
\end{example}

\begin{example}\label{ex:weak-mor-retract-interval}
A virtual morphism 
\[
q\colon \halfspc{}\to [0)
\]
is the datum of a morphism of monoids $q^*\colon \RR_{>0} t^{\NN}\to \sect{\halfspc{},\cMgp_{\halfspc{}}}$ which acts as the identity on the basic elements $\RR_{>0}$, but may send the generator $t$ to an arbitrary element $g(r)r^j$, with $g$ a positive smooth function on $\halfspc{}$ and $j\in\ZZ$. If $g(0)=1$ and $j=1$ then $f$ is a (virtual) retract of the ordinary morphism $i\colon [0)\hookrightarrow \halfspc{}$. Note, however, that there are no ordinary morphisms from $\halfspc{}$ to $[0)$ because $\cM_{\halfspc{}}$ has no phantoms.
\end{example}

\begin{example}
More generally, let $\Sig=(\uSig,\cM_\Sig,\alpha)$ be a manifold with log corners. A virtual morphism $f:\Sig\to [0)$ is the datum of a morphism of monoids $f^*:\RR_{>0}t^\NN\to \cMgp(\Sig)$  which acts as the identity on $\RR_{>0}$, where $\cMgp(\Sig)\defas \sect{\Sig,\cMgp_\Sig}$ denotes the monoid of global sections of $\cMgp_\Sig$. Thus virtual morphisms $\Sig\to [0)$ are in bijection with $\cMgp(\Sig)$. In contrast, ordinary morphisms $\Sig\to [0)$ are in bijection with $\cMphan(\Sig)\defas\sect{\Sig,\cMphan_\Sig} \subset \cM(\Sig)\subset \cMgp(\Sig)$.
\end{example}

\begin{proof}[Proof of \autoref{prop:weak-mor-basic-sections}]
    The statement being local, we can assume that $\uSig$ is connected, and up to shrinking $\uPsi$, that $f$ is a global section of $\cMbas_\Psi$.

    Consider the open sets $V=\{f>0\} \subset \uPsi$ and $U=\uphi^{-1}(V)=\{\uphi^*f>0\} \subset \uSig$, which are non-empty by assumption. By definition of a virtual morphism we have $\phi^*(f|_{V})=\uphi^*(f|_{V})$, or in other words
    \begin{equation}\label{eq:weak-morphism-first-prop}
    (\phi^*f)|_{U}=(\uphi^*f)|_{U}.
    \end{equation}
    Note that since $U$ is nonempty, $\phi^*f$ is necessarily a section of $(\cMbas_\Sig)^{\mathrm{gp}}$, hence is strictly positive on $\uSigo$. (By the description of $\cMgp_{\Sig}$, being a section of $(\cMbas_\Sig)^{\mathrm{gp}}$ is a property that can be checked at a point in a connected manifold with log corners.)
    
    We claim that $\uSigo\subset U$, which amounts to showing that $U\cap \uSigo$ is closed in $\uSigo$. For this, take a sequence $x_n\to x$ in $\uSigo$ where all $x_n\in U$. By \eqref{eq:weak-morphism-first-prop} we have, for all $n$ that
    $$(\phi^*f)(x_n)=(\uphi^*f)(x_n),$$
    and passing to the limit we get $(\uphi^*f)(x)=(\phi^*f)(x) > 0$.  Hence $x\in U$, and so $U\cap\uSigo$ is closed in $\uSigo$ and $\uSigo\subset U$ as desired.

    Now $\phi^*f$ agrees with $\uphi^*f$ on $\uSigo$ and $\uphi^*f$ extends to a smooth function on $\uSig$, hence $\phi^*f$ cannot blow up along $\partial\uSig$ and is a section of $\cMbas_\Sig$. The equality $\phi^*f=\uphi^*f$ follows on the whole of $\uSig$ by density of $\uSigo$ in $\uSig$.
\end{proof}

\begin{remark}\label{rem:virtual-vs-weak}
There is an obvious variant of the notion of virtual morphism where the pullback happens at the level of the sheaves of monoids, without group completing: $\phi^*:\uphi^{-1}\cM_{\Psi}\to \cM_\Sigma$. Such morphisms were called ``weak'' in the first arXiv version of this paper.
\end{remark}

\subsection{Tangential basepoints}\label{sec:tangential-basepoints}

As we saw in \autoref{ex:weak-mor-pt-to-end}, a point, equipped with the trivial log structure, does not admit any ordinary morphism to $[0)$.  More generally,  a point does not admit any ordinary morphism to a manifold with log corners that has phantoms; this includes the boundary of a basic manifold with corners.  However, it admits many virtual morphisms to the boundary.  As we now explain, these correspond to a $C^\infty$ analogue of Deligne's notion of a tangential basepoint in algebraic geometry \cite[\S 15]{Deligne1989}; the precise relationship with the latter is discussed in \autoref{sec:KN-basepoints} below.

\begin{definition}
Let $\uSig$ be a manifold with corners. A \defn{tangential basepoint} of $\uSig$ is a pair $(x,v)$ where $x\in\uSig$ and $v\in N_x^{>0}\uSig$ is a positive normal vector at $x$ in the sense of \autoref{sec:tangent-structure-manifolds-with-corners}.
\end{definition}
 Recall that if $x$ lies in the interior of $\uSig$, the set $N_x^{>0}\uSig$ has a unique element; hence, in the interior, a tangential basepoint is just an ordinary point.  However, the notion becomes nontrivial over the boundary: at a point $x \in \uSig$ of depth $j$, there is a $j$-dimensional space of tangential basepoints; see \autoref{fig:basepoints}. 
 These can be viewed as ``virtual points'' of $\Sig$ thanks to the following result.

\begin{figure}
\begin{tikzpicture}[scale=1.5] \draw[thick,fill=white!90!black] (0,0) -- (1,0) arc (-90:180:1) -- (0,0);
\draw[blue,->] (0,0) -- (30:0.4);
\draw[fill,blue] (0,0) circle (0.03);
\draw[blue,->] (0,0) -- (60:0.6);
\draw[blue,->] (0,0) -- (47:0.7);
\draw[blue,->] (1,0) -- (1,0.3);
\draw[fill,blue] (1,0) circle (0.03);
\begin{scope}[shift={(1,1)}]
\draw[blue,->] (-40:1) -- (-40:0.5);
\draw[fill,blue] (-40:1) circle (0.03);
\draw[blue,->] (25:1) -- (25:0.7);
\draw[fill,blue] (25:1) circle (0.03);
\draw[blue,->] (75:1) -- (75:0.6);
\draw[fill,blue] (75:1) circle (0.03);
\draw[blue,->] (180:1) -- (180:0.5);
\draw[fill,blue] (180:1) circle (0.03);
\end{scope}
\draw[fill,blue] (1.2,1.2) circle (0.03);
\draw[fill,blue] (1.1,0.5) circle (0.03);
\draw[fill,blue] (0.3,1.3) circle (0.03);
\end{tikzpicture}
\caption{Some tangential basepoints on the teardrop manifold.}\label{fig:basepoints}
\end{figure}

 \begin{proposition}\label{prop:tangential-basepoints-as-weak-morphisms}
 Let $\uSig$ be a manifold with corners and let $\Sig = \Sigbas$ be the associated basic manifold with log corners.  Then there is a natural bijection between the set of virtual morphisms $* \to \Sig$ and the set of tangential basepoints of $\uSig$.
 \end{proposition}

 \begin{proof}
 Every virtual morphism $* \to \Sig$ factors uniquely through the pullback log structure on its underlying image point $x \in \uSig$.  (This is the standard universal property for pullback log structures, which works equally well in the context of virtual morphisms.) But by \autoref{cor:normal-monoid} the pullback log structure $\cM_{\Sig}|_x$ is naturally isomorphic to the monoid of non-negative monomial functions on the non-negative normal space $N^{\ge0}_x\uSig$.  Evaluating such functions at points in $N^{>0}_x \uSig$ gives a bijection between tangential basepoints and monoid homomorphisms $\cM_{\Sig}|_x \to \RR_{>0}$ that act as the identity on the constants $\RR_{>0} \subset \cM_{\Sig}|_x$, or equivalently virtual morphisms $* \to (x,\cM_{\Sig}|_x,\alpha_\Sig|_x)$, as desired.
  \end{proof}

Concretely, in local coordinates $(r_1,\ldots, r_n)$ such that $r_1(x)=\cdots = r_j(x) = 0$ and $r_i(x) > 0$ for $i > j$ we may write a tangential basepoint as
\[
v = v_1 \cvf{r_1}|_x + \cdots + v_j\cvf{r_j}|_x
\]
where $v_1,\ldots,v_j \in \RR_{>0}$.  The corresponding virtual morphism $s\colon * \to \Sigbas$ is thus determined by the positive constants
\[
s^*r_i = \begin{cases}
    v_i & 1 \le i \le j \\
    r_i(x) & j < i \le n.
\end{cases}
\]
In light of this result we shall often denote the corresponding virtual morphism $s \colon * \to \Sigbas$ simply by $\sum v_i \cvf{r_i}|_x$.

\begin{remark} 
    In the literature, it is common to consider paths between tangential basepoints $v, w$ on $\Sig$: these are usually defined as smooth maps $\gamma : [0,1] \to \Sig$ that send the interior $(0,1)$ to $\intSig$, and whose initial and final velocities in the normal directions are given by 
    \begin{equation}\label{eq:initial-final-velocity}
    \gamma'(0)=v\qquad \mbox{ and } \qquad \gamma'(1)=-w.
    \end{equation}
    This can be phrased in the language of virtual morphisms as follows. Consider the interval $[0,1]$ with coordinate $t$, viewed as a basic manifold with corners, equipped with its two ``canonical'' tangential basepoints $s_0:=\partial_t|_0:*\to [0,1]$ and $s_1:=-\partial_t|_1:*\to [0,1]$. Then \eqref{eq:initial-final-velocity} implies the equalities
    \begin{equation}\label{eq:initial-final-compatibility-path}
    \gamma\circ s_0=v \qquad \mbox{ and } \qquad \gamma\circ s_1=w
    \end{equation}
    of virtual morphisms $*\to\Sig$. In other words, we have the following commutative diagram of virtual morphisms.
    \begin{equation}\label{eq:diagram-path-between-tangential-basepoints}
    \begin{tikzcd}
        *\sqcup * \ar[r,"s_0{,}s_1"] \ar[d,equal] &[.5cm] [0,1] \ar[d,"\gamma"] \\
        *\sqcup * \ar[r,"v{,}w"'] & \Sig
    \end{tikzcd}
    \end{equation}
    Note that \eqref{eq:initial-final-compatibility-path} holds more generally if we only require the leading Taylor coefficients at $0$ and $1$ to be $v$ and $-w$ respectively. For instance, for $(\Sig,v,w)=([0,1],s_0,s_1)$, this is the case for any smooth map $\gamma:[0,1]\to [0,1]$ sending $(0,1)$ to $(0,1)$ and satisfying $\gamma(t)\sim_0 t^a$ and $\gamma(t)\sim_1(1-t)^b$ for some $a,b\in\NN^*$, even though only those with $(a,b)=(1,1)$ are classically referred to as paths between tangential basepoints.
\end{remark}

\subsection{Scales}\label{sec:scales}
There are natural higher-dimensional counterparts of tangential basepoints, given by sections of normal bundles of strata, or more intrinsically, the phantom tangent bundle.  As we shall see, these correspond to the following notion.
\begin{definition}\label{def:scale}
     A \defn{scale} for a manifold with log corners $\Sig$ is a virtual morphism $s \colon \Sigbas \to \Sig$ that is a section of the natural projection $\Sig\to \Sigbas$.  We say that $s$ is \defn{regular} if the pullback of every phantom function is a section of $\cMbas_\Sigma\subset (\cMbas_\Sigma)^{\mathrm{gp}}$, i.e.~$s^*\cMphan_\Sig\subset \cMbas_\Sig$. We say that $s$ is \defn{nondegenerate} if the pullback of every phantom function is a strictly positive smooth function, i.e.~$s^*\cMphan_{\Sig}\subset \Cinfpos{\Sig}$.
\end{definition}

Concretely, in local coordinates $(r,t)$, the sheaf $\cM_\Sig$ is identified with the monoid $\cMbas_\Sig t^\NN$ freely generated by the phantom coordinates $t_1,\ldots,t_k$ over $\cMbas_\Sig$, and thus a scale is equivalent to the data of the $k$-tuple of nonnegative locally monomial functions
\[
s^*t_1,\ldots,s^*t_k \in (\cMbas_{\Sig})^{\mathrm{gp}}\subset j_*\Cinfpos{\intSig}.
\]
In general, the functions $s^*t_1,\ldots, s^*t_k$ may have poles on the boundary of $\Sig$. The scale $s$ is regular in this chart if and only if these functions extend to (non-negative) smooth functions on $\Sig$. It is nondegenerate if and only if they extend to positive smooth functions on $\Sig$. 

Invariantly, a regular scale is uniquely determined by its restriction to the phantoms, giving a morphism $s|_{\cMphan_\Sig} \colon \cMphan_{\Sig} \to \cMbas_{\Sig}\subset \Cinfge{\Sig}$, which defines a non-negative section of the phantom tangent bundle. We thus have the following. 
\begin{lemma}
    There is a natural bijection between regular scales (respectively, nondegenerate scales) on a manifold with log corners $\Sig$,  and morphisms of basic manifolds with log corners $\Sigbas \to \tbgephan{\Sig}$ (resp.~$\Sigbas \to \tbposphan{\Sig})$ that are sections of the natural projection. 
\end{lemma}

This explains the terminology: a scale assigns a notion of ``unit length vectors'' in the phantom directions.

We will tacitly identify a scale with the corresponding phantom vector field, writing a scale with components $s_j(r) \defas s^*t_j \in \cMbas_{\Sig}$ in local coordinates as
\[
s = \sum s_j(r)\cvf{t_j}.
\]

\begin{example}
By \autoref{ex:weak-mor-pt-to-end}, a scale for $[0)$ with phantom coordinate $t$ is the same thing as a positive number $\lambda \defas s^*t \in \RR_{>0}$; we write $s = \lambda \cvf{t}$.  Note that $s$ is automatically nondegenerate.
\end{example}
\begin{example}
A scale for $[0,\infty)\times [0)$ with coordinates $(r,t)$ is the same thing as a function $s^*(t)=g(r)r^j$, where  $g$ is a positive smooth function on $\halfspc{}$ and $j\in\ZZ$; we write $s = g(r)r^j\cvf{t}$.  Then $s$ is regular iff $j\ge 0$ and non-degenerate iff $j=0$.
\end{example}

\begin{proposition}\label{prop:scales-exist}
    Every manifold with log corners admits a nondegenerate scale.
\end{proposition}

\begin{proof}
    Recall that the bundle $\tbposphan{\Sig} \to \Sig$ has structure group $\symgp{k}\ltimes \RR_{>0}^k$. 
 Let $P \to \uSig$ be the associated $\symgp{k}$-bundle;  its fibres are in bijection with the boundary tangent hyperplanes along the zero section in $\tbgephan{\Sig}$.  Note that $P$ is determined up to isomorphism by the monodromy representation $\pi_1(\uSig) \to \symgp{k}$, and hence there exists a finite cover $\pi \colon \uSig' \to \uSig$ such that $\pi^*P$ is trivial.  The bundle $\pi^*\tbposphan{\Sig}$ then admits a further reduction of structure to the group $(\RR_{>0}^k,\cdot)\cong (\RR^k,+)$, and is thus classified by an element in the sheaf cohomology group $H^1(\uSig',\Cinf{\uSig'})^{\oplus k} = 0$.  Therefore  $\pi^*\tbposphan{\Sig}$ is a trivial bundle with fibre $(0,\infty)^k$, and hence it admits a section, say $s'$.  Averaging $s'$ over the action of $\symgp{k}$, we may assume without loss of generality that $s'$ is $\symgp{k}$-invariant, and  hence descends to a section of $\tbposphan{\Sig}$, giving the desired nondegenerate scale $s$.
\end{proof}

\begin{definition}\label{def:scale-preserving}
    Let $(\Sig,s)$ and $(\Psi,\tilde s)$ be manifolds with log corners equipped with scales.  A virtual morphism $\phi \colon \Sig \to \Psi$ is \defn{scale-preserving} if there exists a virtual morphism $\phi' \colon \Sigbas \to \Psibas$ making the following diagram commute:
    \begin{equation}\label{eq:scale-preserving}
    \begin{tikzcd}
        \Sig \ar[r,"\phi"] & \Psi \\
        \Sigbas \ar[r,"\phi'"]\ar[u," s"] \ar[r] & \Psibas\ar[u,"\tilde s"].
    \end{tikzcd}
    \end{equation}
\end{definition}
Clearly compositions of scale-preserving virtual morphisms are scale-preserving.  Note that the morphism $\phi'$, if its exists, is unique, being given by the formula
\begin{equation}\label{eq:scale-preserving-concrete}
\phi' = p \circ \phi \circ s
\end{equation}
where $p \colon \Psi \to \Psibas$ is the canonical projection. Therefore $\phi$ is scale-preserving if and only if $\tilde{s}\circ p\circ\phi\circ s=\phi\circ s$.

\begin{remark}
Note that for a virtual morphism $\phi:\Sig\to\Psi$ there rarely exists a virtual morphism $\phi':\Sigbas\to \Psibas$ that fits into a commutative diagram with $\phi$ and the natural projections $\Sig\to\Sigbas$ and $\Psi\to\Psibas$. For instance, there is no virtual morphism $i'$ that makes the following diagram commute, where $i$ is the inclusion and $p$ the projection.
\begin{equation*}
\begin{tikzcd}
    {[0)} \ar[r,"i",hook] \ar[d,"p"] & {[0,\infty)} \ar[d,equal] \\
    \{0\} \ar[r,"i'"]& {[0,\infty)} 
\end{tikzcd} 
\end{equation*}
Indeed, $i^*$ sends the basic coordinate $r$ to the phantom coordinate $t$ whereas $p^*$ is the inclusion of $\RR_{>0}$ inside $\RR_{>0}t^\ZZ$ and does not contain $t$ in its image.
\end{remark}

\section{Functions with logarithmic singularities}
\label{sec:log-functions}
Let $\Sig$ be a manifold with log corners. In this section, we construct a natural sheaf $\Cinflog{\Sig}$ of functions with logarithmic singularities on $\Sig$, by adding formal logarithms $\log(f)$ for every $f \in \cM_{\Sig}$, in such a way that $\log(f)$ agrees with the usual logarithm of the function $\alpha(f)$ whenever the latter is not identically zero.  To make the functoriality of the construction clear, we will start by defining $\Cinflog{\Sig}$ abstractly using generators and relations; we then spell out what this means concretely in local coordinates.  

Throughout the present \autoref{sec:log-functions} we will denote formal logarithms by $\ulog(f)$, reserving $\log(f)$ for the logarithm of an actual function, in order to avoid confusion.  This notation is temporary and from \autoref{sec:vector-fields} onwards we shall simply write $\log(f)$ for both the formal and actual logarithm;  \autoref{thm:injectivity-of-j*} below will guarantee that this does not introduce any ambiguity.

\subsection{Construction and functoriality} Our sheaf of functions is defined as follows.
\begin{definition}\label{def:log-functions}
    The \defn{sheaf of logarithmic functions on $\Sig$} is the sheaf $\Cinflog{\Sig}$ of $\Cinf{\Sig}$-algebras generated by formal symbols $\ulog(f)$ where $f \in \cMgp_\Sig$, subject to the relations:
\begin{enumerate}
    \item For every $f_1,f_2 \in \cMgp_\Sig$, we have
    \[
        \ulog(f_1f_2) = \ulog(f_1)+\ulog(f_2).
    \]
    \item If $f \in \cMbas_\Sig$ and $g \in \Cinf{\Sig}$ are such that $g\log(\alpha(f))$ is smooth on $\Sig$, then 
    \[
        g\ulog(f) = g\log(\alpha(f)),
    \]
    where $\log(\alpha(f))$ denotes the real-valued logarithm of the nonnegative function $\alpha(f)$.
\end{enumerate}
\end{definition}

\begin{remark}
The first relation implies that $\ulog(1)=0$ and $\ulog(f^{-1})=-\ulog(f)$ for all $f\in\cM_\Sig$, and therefore $\Cinflog{\Sig}$ is generated by the formal symbols $\ulog(f)$ for $f\in \cM_\Sig\subset \cMgp_\Sig$.
\end{remark}

\begin{remark}
    By ``$g\log(\alpha(f))$ is smooth on $\Sig$'', we mean the following: since $f$ is basic, it is expressed in local coordinates as $f=h(r)r_1^{j_1}\cdots r_n^{j_n}$ for some $j_i\in\NN$ and a positive smooth function $h$. Therefore, $g\log \alpha(f)$ is a well-defined smooth function on the interior of $\Sig$, and the condition is that this function extends as a smooth function on $\Sig$. The right-hand side of the second relation above refers to this extension, as a section of $\Cinf{\Sig}$ (or rather its image in $\Cinflog{\Sig}$).
\end{remark}

The meaning of the first relation is self-evident.  The second relation ensures that we do not overcount the smooth functions.  For instance,  taking $g=1$ and identifying a positive function $f$ with the corresponding section of $\cM_\Sig$, we have the following.

\begin{lemma}\label{lem:formal-vs-actual-log}
    If $f \in \Cinfpos{\Sig}$, then $\ulog(f) = \log(f) \in \Cinflog{\Sig}$.
\end{lemma}

The second relation in \autoref{def:log-functions} is more subtle when $\alpha(f)$ has zeros, since then $\log( \alpha(f))$ is not smooth.  For instance, we have the following useful property.

\begin{lemma}\label{lem:flat-times-log}
    Suppose that $f \in \cMbas_\Sig$ and $g \in \Cinf{\Sig}$ are such that the function $g\log(\alpha(f))$ is continuous.  If $x \in \Sig$ is any point such that the function $\alpha(f)$ vanishes at $x$, then $g$ and $g\log(\alpha(f))$ also vanish at $x$.
\end{lemma}

\begin{proof}
    The problem is local, so we may work in a chart with coordinates $(r,t)$.  Then $f$ has the form $f = f_0(r) r^J$ where $f_0(r)$ is a positive smooth function and $J =(j_1,\ldots,j_n)$ is a multi-index.  In particular the vanishing set of $f$ is the union of the boundary facets $r_i=0$ for indices $i$ such that $j_i> 0$, so it suffices to show that $g$ also vanishes there. But the function
    \[
        g \log(\alpha(f)) = g \cdot \rbrac{\log(f_0) + \sum_{i=1}^n j_i \log(r_i)}.
    \]
    is continuous, hence bounded, which implies that $g \to 0$ as $r_i \to 0$, to compensate for the divergence of $\log(r_i)$. 
 Hence, by Hadamard's lemma, $g$ is divisible by $r_i$ in the algebra of smooth functions.  Since $r_i\log r_i \to 0$  as $r_i\rightarrow 0$, it follows that $g\log\alpha(f)$ also vanishes in this limit.
\end{proof}
\begin{remark}
If  $g \log \alpha(f)$ is not just continuous, but actually smooth, then $g$ must vanish to infinite order on the vanishing set of $\alpha(f)$; see \autoref{prop:j^*-injective-induction} below.
\end{remark}

The sheaves $\Cinflog{}$ are functorial with respect to virtual morphisms, in the following sense.

\begin{lemma}\label{lem:functoriality-of-log-functions}
    If $\phi\colon \Sig \to \Psi$ is a virtual morphism, then the formula
    \[
    \phi^*(\ulog(f))\defas\ulog(\phi^*(f)) \quad (f\in\uphi^{-1}\cMgp_\Psi)
    \]
    uniquely defines a morphism $\phi^*\colon \uphi^{-1}\Cinflog{\Psi} \to \Cinflog{\Sig}$ of sheaves of $\uphi^{-1}\Cinf{\Psi}$-algebras.
\end{lemma}

\begin{proof}
Uniqueness is clear because the monoid elements generate $\Cinflog{}$ over $\Cinf{}$. It remains to show that the pullback is well-defined, that is, we must show that the map $\phi^*\colon \uphi^{-1}\cM_{\Psi} \to \cM_\Sig$ preserves the ideals generated by the two types of relations in \autoref{def:log-functions}. Compatibility with the first relation is immediate since $\phi^*$ is a monoid homomorphism. For the second relation, let $x \in \Sig$ be a point, let $f \in \uphi^{-1}(\cMbas_{\Psi})_x = \cMbas_{\Psi,\uphi(x)}$ be a germ of a section and let $g \in \Cinf{\Psi,\uphi(x)}$ be a germ of a function such that $g \log(\alpha_\Psi(f))$ is smooth on $\Psi$.  We must show that
\begin{align}\label{eq:log-pullback-relation-check}
    \uphi^*g\cdot \ulog(\phi^*f) = \uphi^*\rbrac{g\log\rbrac{\alpha_\Psi(f)}} \in \Cinflog{\Psi,x}.
\end{align}
There are two possibilities: either the germ $\uphi^*(\alpha_{\Psi}(f))$ is nonzero, or it is zero.  If $\uphi^*(\alpha_{\Psi}(f))$ is nonzero, then by \autoref{prop:weak-mor-basic-sections}, $\phi^*f\in \cM_\Sig$ and $\alpha_{\Sig}(\phi^*f)=\uphi^*(\alpha_\Psi(f))$, hence \eqref{eq:log-pullback-relation-check} is one of the defining relations for $\Cinflog{\Sig,x}$.  Otherwise, $\uphi^*(\alpha_\Psi(f)) = 0$ says that $\uphi$ maps a neighbourhood of $x$ to the vanishing set of $\alpha_\Psi(f)$. Hence by \autoref{lem:flat-times-log}, we have $\uphi^*g = \uphi^*(g \log \alpha_\Psi(f)) =0$, so that both sides of \eqref{eq:log-pullback-relation-check} are identically zero.
\end{proof}

\begin{example}
As an illustration of the subtle point in the proof of the lemma, consider the virtual morphism $i\colon\{0\}\to \halfspc{}$ corresponding to a tangential basepoint $c\,\cvf{r}$ at $0$ with $c>0$, i.e.~defined by $i^*(r)=c$. Let $g(r)$ be a smooth function on $\halfspc{}$ such that $g(r)\log(r)$ is smooth. Then we have the relation
\[
g(r)\ulog(r)=g(r)\log(r)
\]
in $\Cinflog{\halfspc{}}$ and we need to prove that the relation
\[
g(0)\ulog(c)\stackrel{?}{=}\big[g(r)\log(r)\big]_{r=0}
\]
is satisfied in $\Cinflog{\{0\}}=\RR$. The fact that $g(r)\log(r)$ is smooth implies, thanks to \autoref{lem:flat-times-log}, that both $g(r)$ and $g(r)\log(r)$ vanish at $r=0$, and therefore the latter relation reads $0=0$.
\end{example}

Recall that the boundary immersion of $\ubdSig \to \uSig$ a manifold with corners is locally modelled by the inclusion $\{0\} \times \halfspc{{n-1}} \to \halfspc{n}$ of the boundary facet $\{r_1=0\}$ in $\halfspc{{n}}$, so that $\Cinf{\ubdSig}$ is locally identified with the quotient
\[
\Cinf{\{0\}\times\halfspc{n-1}}
\cong \Cinf{\halfspc{n}}/(r_1),
\]
where $(r_1)$ denotes the principal ideal generated by $r_1$.  Using the functoriality above, we can likewise show that for a manifold with log corners $\Sig$, the algebra $\Cinflog{\bdSig}$ is locally identified with $\Cinflog{\halfspc{n}\times[0)^k}/(r_1)$.  In coordinate-free terms, we have the following following.
\begin{lemma}\label{lem:boundary-log-principal-ideal}
Let $\Sig$ be a manifold with log corners, let $i : \bdSig \to \Sig$ be the natural immersion of the boundary, and let $\mathscr{I} = \ker(\underline{i}^* : i^{-1}\Cinf{\uSig} \to \Cinf{\ubdSig})$ be the ideal defining the underlying immersion $\underline{i}: \ubdSig \to \uSig$ of manifolds with corners.  Then the pullback $i^* : i^{-1}\Cinflog{\Sig} \to \Cinflog{\bdSig}$ induces an isomorphism
\[
\Cinflog{\bdSig} \cong \frac{i^{-1}\Cinflog{\Sig}}{\mathscr{I}\cdot i^{-1}\Cinflog{\Sig}}
\]
\end{lemma}
\begin{proof}
    The problem is local in $\bdSig$, so it suffices to prove the corresponding statement when $\Sig = \halfspc{} \times \Psi$ for some manifold with log corners $\Psi$, and the boundary immersion $\bdSig \to \Sig$ is replaced with the natural immersion
    \[
    i : [0) \times \Psi \hookrightarrow \halfspc{} \times \Psi
    \]
    Let $r$ be the standard coordinate on $\halfspc{}$ and $t = i^*r$ its pullback to $[0)$.  Then $\mathscr{I} \defas \ker{\underline{i}^*} = (r) < \Cinf{\halfspc{}\times \uPsi} $ is the principal ideal generated by $r$.  We must show that $\ker i^* = r \cdot \Cinflog{[0)\times\Psi}$.  
    The inclusion $r \cdot \Cinflog{[0)\times\Psi} \subset \ker i^*$ is immediate, since $i^*$ is a morphism of $\Cinf{\Sig}$-algebras and $i^*r = 0$ as smooth functions. 
    
    For the reverse inclusion, set
    \[
    p = q \times \id{\Psi} : \halfspc{} \times {\Psi} \to [0) \times \Psi
    \]
    where $q : \halfspc{} \to [0)$ is the virtual retraction from  \autoref{ex:weak-mor-retract-interval}, defined by $q^*t = r$.    Since $\cM_\Sig$ is generated over $\Cinfpos{\Sig}$ by $r$ and  $\cM_{\Psi}$, any element $f \in \Cinflog{\Sig}$ can be locally represented by a finite sum
    \[
    f = \sum_j g_j (\ulog r)^{l_j} \ulog h_{i,1} \cdots \ulog h_{i,m_i}
    \]
    where $g_j \in \Cinf{\halfspc{}\times\uPsi}$ are ordinary smooth functions, $l_j,m_j \in \ZZ_{\ge 0}$, and $h_{i,m} \in \cM_{\Psi}$ for all $i,m$.  Note that if $g$ is a smooth function, then $g - p^*i^*g$ is the difference between $g$ and its restriction to $r=0$, and is therefore divisible by $r$ thanks to Hadamard's lemma.  On the other hand, $p^*i^* \ulog r = \ulog p^*i^* r = \ulog p^* t = \ulog r$ and similarly $p^*i^*\ulog  h = \ulog h$ for all $h \in \cM_{\Psi}$. Therefore, if $f \in \ker i^*$, we have that
    \[
    f = f - p^*i^*f = \sum_j (g_j-p^*i^*g_j) (\ulog r)^{l_j}  \ulog h_{j,1} \cdots \ulog h_{j,m_j}
    \]
    is divisible by $r$, as desired.
\end{proof}

\subsection{Local structure of logarithmic functions}
Our goal now is to prove \autoref{thm:log-functions} from the introduction, which describes the local structure of $\Cinflog{\Sig}$.  This follows from \autoref{lem:phantom-log-functions} and \autoref{thm:injectivity-of-j*} below which treat the contributions from phantom and basic directions, respectively.

\subsubsection{Phantom logarithms}
After the logarithms of positive functions---which are just smooth functions---the next simplest elements of $\Cinflog{}$ to understand are those of the form $\ulog f$, where $f$ is a phantom. 
\begin{example}
     Let $\Sig = [0)^k$ with phantom coordinates $t_1,\ldots,t_k$.  Since $\cM_\Sig =\RR_{>0}t^\NN$ is freely generated over the positive constants by $t_1,\ldots,t_k$, the logarithmic functions form a free commutative algebra generated by $\ulog t_1,\ldots,\ulog t_k$, i.e.
     \[
     \Cinflog{\Sig} \cong \RR[ \ulog t_1,\ldots,\ulog t_k],
     \]
     is a polynomial ring in the formal logarithms of the phantom coordinates.
\end{example}

More generally, note that the  projection $p \colon \Sig \to \Sigbas$ gives a canonical map
\[
p^*\colon \Cinflog{\Sigbas}\to\Cinflog{\Sig}
\]
of $\Cinf{\Sig}$-algebras.  We have the following:
\begin{lemma}\label{lem:phantom-log-functions}
    If $t_1,\ldots,t_k$ is a system of phantom coordinates in a neighbourhood of any point $x \in \Sig$, then $p^*$ gives a canonical isomorphism
    \[
    \Cinflog{\Sigbas,x}[\ulog t_1,\ldots,\ulog t_k] \stackrel{\sim}{\longrightarrow} \Cinflog{\Sig,x}
    \]
    of $\Cinf{\Sig,x}$-algebras.
\end{lemma}

\begin{proof}
    By definition, $\cM_\Sig$ is freely generated over $\cMbas_{\Sig}$ by the phantom coordinates, hence the map is surjective.  Since the second relation in \autoref{def:log-functions} only involves basic elements, the only relations involving $t_1,\ldots,t_k$, are those of the first type, which define the monoid algebra of $\cM_\Sig$.  The result follows since the monoid algebra of a free monoid is a polynomial ring.
\end{proof}

\subsubsection{Structure of the basic logarithms}
In light of \autoref{lem:phantom-log-functions}, it remains to understand the structure of  the subalgebra $\Cinflog{\Sigbas}\subseteq \Cinflog{\Sig}$ generated by basic elements.  Therefore, we now assume without loss of generality that $\Sig = \Sigbas$ is basic.

Let $r_1,\ldots,r_n$ be coordinates in a neighbourhood of $x \in \Sig$, so that $\cM_{\Sig,x}=\Cinfpos{\Sig,x}r^\NN$. Using the first relation in \autoref{def:log-functions}, we may write every $\uf \in \Cinflog{\Sig}$ locally as a finite sum
\begin{align}
\uf = \sum_{I} f_I(r) \ulog{}^I(r)\label{eq:coord-log-function-formal}
\end{align}
over multi-indices $I=(I_1,\ldots,I_n) \in \NN^n$ where $f_I(r)$ are smooth functions on the underyling manifold with corners, and 
\[
\ulog{}^I(r) \defas \ulog{}^{I_1}(r_1) \cdots \ulog{}^{I_n}(r_n) 
\]
denotes the corresponding monomial in the formal logarithms of the coordinates.  Our aim now is to show that the relations defining $\Cinflog{\Sigbas}$ amount to the statement that the formal symbols $\ulog r_i$ can be identified with the actual logarithm functions $\log r_i$, so that the expressions $\uf$ as above can be manipulated like ordinary functions of $r$; this is the content of \autoref{thm:injectivity-of-j*} below but to formulate the statement we need to set some notation.

Let $j \colon \intSig \hookrightarrow \Sig$ be the inclusion of the interior,  where all coordinates $r_i$ are positive. By functoriality, we have a canonical map of sheaves of algebras
\begin{equation*}
    j^*\colon \Cinflog{\Sig} \to j_*\Cinflog{\intSig},
\end{equation*}
given by restriction of functions to the interior.  Since $\Sig$ is assumed basic, we have $\cM_{\intSig} = \Cinfpos{\intSig}$, so we deduce from \autoref{lem:formal-vs-actual-log} that $\Cinflog{\intSig} = \Cinf{\intSig}$, and $j^*$  sends the expression $\uf$ from \eqref{eq:coord-log-function-formal} to the smooth function
\begin{equation}\label{eq:coord-log-function-interior}
    j^*\uf \defas \sum_{I} f_{I}(r) \log^I(r)
\end{equation}
obtained by replacing each formal logarithm $\ulog r_i$ with the corresponding smooth function $\log r_i$ defined in the interior. 

\begin{theorem}\label{thm:injectivity-of-j*}
    For a basic manifold with log corners $\Sig = \Sigbas$, the map $j^*$ is injective, and identifies $\Cinflog{\Sigbas}$ with  the sheaf of smooth functions in the interior with at worst polynomially logarithmic divergences along the boundary, i.e.~those which have the form \eqref{eq:coord-log-function-interior} in some (and hence any) coordinate chart.
\end{theorem}

The rest of this subsection is devoted to the proof of \autoref{thm:injectivity-of-j*}. In fact we will formulate and prove a stronger statement (\autoref{prop:j^*-injective-induction} below) that gives more precise control over the behaviour of the coefficients, allowing us to prove the theorem by induction on the number of boundary components.

The statement is local, so it suffices to prove it for the stalk at the origin in the manifolds with log corners
\begin{equation*}
    \Sig_{m,n} \defas \RR^m \times [0,\infty)^n
\end{equation*}
for $m,n \ge 0$, with coordinates $r_1,\ldots,r_n$ on the factor $\halfspc{n}$.

To formulate the stronger statement, we adopt the following terminology.  If  $I = (I_1,\ldots,I_n)$ is a multi-index, its \defn{support} is the collection of variables $r_l$ such that $I_l\neq 0$, and its \defn{vanishing set} $V(I)=\{I_1 r_1=\cdots=I_n r_n=0\}$ is the locus where all coordinates in the support vanish. If $h \in \Cinf{\Sig_{m,n},0}$ is a germ of a smooth function, we say that $h$ is \defn{$I$-flat} if we have
\begin{equation*}
\big(\cvf{r_1}^{a_1}\cdots \cvf{r_n}^{a_n} h \big)\Big|_{V(I)} = 0
\end{equation*}
for every multi-index $A = (a_1,\ldots,a_n)\in\NN^n$ whose support is contained in that of $I$. This always includes $A =0$, where the statement amounts to $h|_{V(I)} = 0$.

The condition to be $I$-flat becomes stronger for smaller support of $I$. The strongest case is $I=0$ (empty support), where $h$ is $I$-flat if and only if $h=0$, because $V(I)=\Sig_{m,n}$. The weakest condition arises when $I$ has full support $\{r_1,\ldots,r_n\}$; then an $I$-flat germ $h$ is only constrained near the codimension $n$ corner $\{r_1=\cdots=r_n=0\}$ of $\Sig_{m,n}$. Note, however, that for \emph{any} $I$, being $I$-flat implies that the Taylor expansion of $h$ at the origin is identically zero.
\begin{lemma}
    Suppose that $I$ is a multi-index with support $\{r_1,\ldots,r_j\}$ for some $j \ge 0$, and $h$ is an $I$-flat smooth function.  Then there exist smooth functions $h_1,\ldots,h_j$ such that $h=h_1+\cdots + h_j$ and $h_i$ is $r_i$-flat for all $i$.  In particular, $h_i \log(r_i)$ is smooth and $r_i$-flat.
\end{lemma}

\begin{proof}
    We proceed by induction on $j$ and show that all functions $h_i$ can in fact be chosen to vanish wherever $h$ does. Let $\psi\in\Cinf{}(\RR)$ denote a smooth function of one variable, with compact support, that is equal to $1$ in a neighbourhood of $0$. Note that a partition of unity $h=\psi(z) h+(1-\psi(z)) h$, with $z=r_i$ or $z=x_i$, preserves $I$-flatness. We can thus assume that $h$ has compact support.
    
    In the cases $j=0$ ($h=0$) and $j=1$ ($h=h_1$) there is nothing to prove.  For $j>1$, let
    \begin{equation*}
    a_n(r_1,\ldots,r_{j-1}) = \frac{1}{n!}(\cvf{r_j}^n h) \big|_{r_j=0}
    \end{equation*}
    be the $n$th Taylor coefficient of $h$ with respect to $r_j$; it is a smooth function of all remaining variables and it is flat with respect to $\{r_1,\ldots,r_{j-1}\}$ by assumption.  Hence by induction we may write
    \[
    a_n = a_{n,1} + \cdots + a_{n,j-1}
    \]
    where $a_{n,i}$ is $r_i$-flat for $1 \le i < j$. Following a proof of Borel's lemma, we may find a sequence of positive numbers $C_n$ such that for all $i<j$, the sum
    \begin{equation*}
        h_i \defas \sum_{n} r_j^n  a_{n,i}  \psi(r_jC_n) 
    \end{equation*}
    and all its partial derivatives are absolutely convergent. 
 It follows that the derivative of $h_i$ with respect to $r_i$ can be computed termwise, and hence the $r_i$-flatness of $a_{n,i}$ implies that $h_i$ is $r_i$-flat. 

 Now let $h_j = h - h_1 - \cdots - h_{j-1}$.  Then by construction, the Taylor expansion of $h_j$ along $r_j=0$ is identically zero.  Hence $h_1,\ldots,h_j$ give the desired functions. 

 Finally, observe that since $h_i$ is $r_i$-flat, all partial derivatives of $h_i \log(r_i)$ exist and are smooth, even $r_i$-flat.
\end{proof}

The following is a strengthened version of \autoref{thm:injectivity-of-j*}.
.
\begin{proposition}\label{prop:j^*-injective-induction}
Suppose that $f_{I} \in \Cinf{\Sig_{m,n},0}$ are germs of smooth functions at the origin, set
\[
    \uf \defas \sum_{I} f_{I}(x,r)\ulog^I(r)
    \in \Cinflog{\Sig_{m,n},0}
\]
and denote the germ of the corresponding function on the interior as
\[
    f \defas j^*\uf = \sum_{I} f_{I}(x,r)\log^I(r)
    \in \big(j_*\Cinf{\intSig_{m,n}}\big)_0
    .
\]
We write $I>J$ to denote that $I_k\ge J_k$ for all indices $1\le k \le n$, with at least one of these inequalities being strict ($I_k > J_k$).

If $f=0$ then the following statements hold:
\begin{itemize}
     \item[$(\alpha)$] We have $\uf = 0 \in \Cinflog{\Sig_{m,n},0}$.
     \item[$(\beta)$] If $J$ is such that $f_I=0$ for all $I$ with $I>J$, then the function $f_{J}$ is $J$-flat.
 \end{itemize}
\end{proposition}

\begin{proof}
We first show that $(\alpha)$ follows from $(\beta)$. Namely, note that for a multi-index $J$ as in $(\beta)$, since $f_{J}$ is $J$-flat, we may write $f_{J} = f_{J,j_1} + \cdots + f_{J,j_k}$ where the indices $j_l$ range over the support of $J$ and $f_{J,j_l}$ is $r_{j_l}$-flat.
Therefore, using the second defining relation for $\Cinflog{}$ from \autoref{def:log-functions}, we may replace $f_{J}\ulog^J(r)$ in the expression for $\uf$ by a sum of monomials of strictly smaller degree in the symbols $\ulog r_i$. By induction over the degree in $\ulog(r)$, we conclude that $\uf$ has a representative without logarithms. But then, $\uf=f_0$ is $0$-flat by $(\beta)$, and hence identically zero.

We now prove $(\beta)$ by induction over the support of $J$. The base case is $J=0\in\NN^n$ with empty support, so we must show that $f_0=0$. Since we have $I>J$ for all $I\neq 0$, and thus $f_I=0$, we have $f_0|_{\intSig_{m,n}}=f=0$. By continuity, $f_0=0$ and so $f_0$ is indeed $J$-flat.

From now on we will assume that we are given $J$ as in $(\beta)$, with non-empty support, and that $(\beta)$ holds in all cases with smaller support. Label one of the coordinates in the support of $J$ as $r_n$, and consider the embedding
\[
i\colon \Sig_{m,n-1} \times [0) \hookrightarrow \Sig_{m,n}
\]
of the vanishing set of $r_n$. Furthermore, we denote the inclusion of the relative interior of this facet as
\begin{equation*}
    j_n \colon \intSig_{m,n-1}  
    \hookrightarrow \Sig_{m,n-1} 
    .
\end{equation*}
We denote the $r_n$-derivaties of the coefficients of $\ulog^l(r_n)$ in $\uf$ as
\begin{equation}\label{eq:r_n-expansion}
    \ug_l^{(k)} \defas \sum_{I\colon I_n=l} \left(\cvf{r_n}^k f_{I}\right)_{r_n=0} \ulog^{I'}(r)
    \in \Cinflog{\Sig_{m,n-1},0}
    ,
\end{equation}
where we write $I'=(I_1,\ldots,I_{n-1})$. We will now show, by induction on $k$, that
\begin{equation*}
    j_n^* \left(\ug_l^{(k)} \right)
    =\lim_{r_n=0} \cvf{r_n}^k \left(\sum_{I\colon I_n=l} j^*f_I \log^{I'}(r) \right)
    = 0
\end{equation*}
for all $k$ and $l$. Namely, by Hadamard's lemma there are smooth functions $h_I^{(k)}$ for each $I$ and $k$ such that
\begin{equation*}
    f_I=\sum_{a=0}^{k} \frac{r_n^a}{a!} \left(\cvf{r_n}^a f_I \right)_{r_n=0}
    +r_n^{k+1} h_I^{(k)}
\end{equation*}
Therefore, once $j_n^*(\ug_l^{(0)})=\cdots=j_n^*(\ug_l^{(k-1)})=0$ is established, it follows from
\begin{equation*}
    0 = \frac{f}{r_n^k} = \frac{1}{k!}\sum_l j_n^*\left(\ug_l^{(k)}\right) \log^l(r_n) + \sum_l r_n \log^l(r_n) \sum_{I\colon I_n=l} h_I^{(k)} \log^{I'}(r) 
\end{equation*}
by considering the limit $r_n\rightarrow 0$ that we must have, for all $l$, that $j_n^*\big(\ug_l^{(k)}\big)=0$.

Hence, by the induction hypothesis, we conclude in particular that for $l=J_n$ and for every $k$, the coefficient $\cvf{r_n}^k f_J|_{r_n=0}$ of $\ulog^{J'}(r)$ in $\ug_l^{(k)}$ is $J'$-flat, where $J'=(J_1,\ldots,J_{n-1})$. Note that $J'$ fulfils the condition $(\beta)$ for $\ug_l^{(k)}$, since the coefficient of any multi-index $I'>J'$ in $\ug_l^{(k)}$ is $\cvf{r_n}^k f_I|_{r_n=0}$ where $I=(I',l)>J=(J',l)$ and thus $f_I=0$. We have thus proved that $f_J$ is $J$-flat (recall that $l=J_n>0$ since $r_n$ is in the support of $J$).
\end{proof}

\subsection{Regularized limits} Using scales, we may assign finite values to divergent limits of logarithmic functions, by the following general recipe.  

First, note that if $\Psi$ is a manifold with log corners, and $s : \Psibas \to \Psi$ is a scale on $\Psi$, given in local coordinates by $s = \sum_i s^i(r)\cvf{t_i}$, then the pullback $s^* : \Cinflog{\Psi} \to \Cinflog{\Psibas}$ acts as the identity on the subsheaf of basic logarithms, and sends the phantom logarithms to the functions
\[
s^*\log(t_i) = \log(s^i(r)) \in \Cinflog{\Sigbas}.
\]
Then if $\phi : \Psi \to \Sig$ is any virtual morphism, and $g \in \Cinflog{\Sig}$ is a logarithmic function, we may define its \defn{regularized pullback} to be the function
\[
\reg[s] \phi^*(f) := s^*\phi^*f \in \Cinflog{\Psibas},
\]
which is smooth function in the interior of $\uPsi$ but may have logarithmic divergences on the boundary.

In particular, in the case in which $\phi$ is the natural immersion $\bdSig \to \Sig$ of the boundary of a basic manifold with log corners, we obtain a \defn{regularized restriction}
\[
f|_{\bdSigbas}^{\reg[s]} \in \Cinflog{\bdSigbas}
\]
This construction recovers the usual restriction when the latter makes sense:

\begin{lemma}\label{lem:continuous-restriction}
    If $f \in \Cinflog{\Sigbas}$ extends to a continuous function on $\Sigbas$, then the regularized restriction $f|_{\bdSigbas}^{\reg[s]}$ agrees with the ordinary restriction  as a function.
\end{lemma}

\begin{proof}
    For smooth functions the statement is vacuous.  In general, if we have an expansion $f = \sum f_I\log^I(r)$, and $f$ is continuous, then arguing as in \autoref{lem:flat-times-log}, we see that the coefficient $f_I$ must be divisible by $r_i$ whenever $r_i$ is in the support of $I$.  But 
    \[
    r_i\log(r_i)|^{\reg[s]}_{r_i=0} := s^*(0 \cdot \log(t_i)) = 0,
    \]
    so the non-smooth terms in $f$ do not contribute to the regularized restriction.
\end{proof}

The regularized restriction defined in this way recovers the classical notion of regularized limit, as follows.
\begin{example}
    Let $\Sig = \halfspc{}$ and $\Psi = [0)$ be equipped with the standard coordinates $r$ and $t = r|_{[0)}$, respectively.  Define a scale on $[0)$ by $s = c \,\cvf{t}$ for some $c>0$.  If $f = \sum f_j(r)\log^j(r) \in \Cinflog{\halfspc{}}$, its regularized restriction to zero is given by
    \[
    f|^{\reg[s]}_0 = s^*(f|_{[0)}) = s^*\rbrac{\sum_j f_j(0)\log^j(t)} = \sum_j f_j(0)\log^j(c)
    \]
    When $c= 1$ is the unit scale, we get
    \[
     f|^{\reg[s]}_0 = f_0(0) =: \reglim_{\epsilon \to 0} \sum_j f_j(\epsilon)\log^j(\epsilon),
    \]
    which is the standard definition of the regularized limit.
\end{example}

\begin{example}
    Let $\Sig = \halfspc{2}$ with basic coordinates $(r_1,r_2)=(x,y)$.  Let $u$ and $v$ be the corresponding phantom coordinates on the boundary components $x=0$ and $y=0$, respectively.  A scale on the boundary then has the form
    \[
    s = 
    \begin{cases}
      e^{g(x)} x^a \cvf{v} & \textrm{on } y=0 \\
      e^{h(y)} y^b \cvf{u} &\textrm{on } x=0
    \end{cases}
    \]
    where $g,h$ are smooth functions on $\halfspc{}$ and $a,b \in \ZZ$.  If $f = \log(x)\log^2(y)$, then the regularized restriction of $f$ to the boundary is given by
    \[
     f|^{\reg[s]}_{\bdSig} = \begin{cases}
        \log(x)(g+a\log(x))^2 &\textrm{on }y=0 \\
        (h+b\log(y))\log^2(y) & \textrm{on } x=0
    \end{cases}
    \]
    exhibiting the explicit dependence on all components of the scale.
\end{example}

\section{de Rham theory for manifolds with log corners}
    
\subsection{Vector fields}\label{sec:vector-fields}

The natural notion of vector field in the context of functions with logarithmic singularities is as follows.
\begin{definition} A \defn{logarithmic vector field on $\Sig$}   is a derivation of the $\RR$-algebra $\Cinflog{}(\Sig)$ of global logarithmic functions on $\Sig$.
\end{definition}

The same argument as in classical differential geometry shows that logarithmic vector fields on $\Sig$ are local operators, so that they form a sheaf.
 
\begin{definition}
    We denote by $\VFlog{\Sig}$ the sheaf of logarithmic vector fields on $\Sig$.
\end{definition}

Concretely, in local coordinates $(r,t)$, define derivations 
\begin{align}
\logcvf{r_1},\ldots,\logcvf{r_n}, \logcvf{t_1},\ldots,\logcvf{t_k}\in \VFlog{\Sig} \label{eq:coord-vfs}
\end{align}
by the ``obvious'' formulae, as follows. Suppose that $f \in \Cinflog{\Sig}$.  Then by \autoref{lem:phantom-log-functions} and \autoref{thm:injectivity-of-j*} we may write $f$ uniquely in the form
\[
f = \sum_J f_J (\log t)^J
\]
where $f_J \in \Cinflog{\Sigbas}$ is a smooth function in the interior that 
can be written in the form $f_J = \sum_I f_{I,J} (\log r)^I$ for some smooth functions $f_{I,J}$.  We then set
\[
\logcvf{r_i} f \defas \sum_J (\logcvf{r_i} f_J) (\log t)^J
\]
and define $\logcvf{t_i}$ be the unique $\Cinflog{\Sigbas}$-linear derivation such that
\[
\logcvf{t_i} (\log t_j) = \begin{cases}
1 & i=j \\
0 & i \neq j
\end{cases}
\]
It is straightforward to check that these operations are well-defined derivations.

\begin{proposition}\label{prop:T-basis}
The derivations \eqref{eq:coord-vfs} form a local basis for $\VFlog{\Sig}$ as a $\Cinflog{\Sig}$-module.  Hence if $\Sig$ is a manifold with log corners of dimension $(n,k)$, the sheaf $\VFlog{\Sig}$ is a locally free $\Cinflog{\Sig}$-module of rank $n+k$.
\end{proposition}

\begin{proof}[Proof of \autoref{prop:T-basis}]
    By \autoref{lem:phantom-log-functions}, the algebra $\Cinflog{\Sig}$ is freely generated over $\Cinflog{\Sigbas}$  by $\log t_1,\ldots,\log t_j$.    Thus the operators $\logcvf{t_1},\ldots,\logcvf{t_m}$ form a basis for the $\Cinflog{\Sigbas}$-linear derivations, and hence it suffices to show that $\logcvf{r_1},\ldots,\logcvf{r_l}$ form a basis for the derivations of $\Cinflog{\Sigbas}$.  
    
    We first claim that any derivation $Z$ of $\Cinflog{\Sigbas}$ is uniquely determined by its action on smooth functions. Indeed, let $h \in \Cinf{\Sig}$ be any smooth function that is nonvanishing in the interior and such that $h \log r_i$ is smooth.  Then 
    \[
    Z(h \log r_i) = Z(h) \log r_i + h Z(\log r_i)
    \]
    so that
    \[
    Z(\log r_i) = h^{-1}\cdot (Z(h \log r_i) - Z(h)\log r_i).
    \]
    gives a formula for $Z(\log r_i)$ in terms of the action of $Z$ on smooth functions, for every $i$.  By linearity and the Leibniz rule, this determines the action of $Z$ on any element of $\Cinflog{\Sig}$ from the action of $Z$ on smooth functions.  

    But the action of $Z$ on smooth functions is a derivation from $\Cinf{\Sig}$ to $\Cinflog{\Sigbas}$, and must therefore have the form
    \[
    Z = \sum_{i=1}^n Z^i \cvf{r_i}
    \]
    for some coefficient functions $Z^i\in \Cinflog{\Sigbas}$. It remains to show that such a derivation extends to $\Cinflog{\Sigbas}$ if and only if $Z^i$ is divisible by $r_i$ for all $i$.  Equivalently, we must show that the expression for $Z(\log r_i)$ above is a logarithmic function if and only if $Z^i$ is divisible by $r^i$.  Let $h = h(r_i)$ be any function of $r_i$ such that $h\log r_i$ is smooth. Then $Z$ acts as $Z^i\cvf{r_i}$ on both $h \log r_i$ and $h$. Hence we have
    \begin{align*}
     Z(\log r_i) = h^{-1}\cdot (Z(h \log r_i) - Z(h)\log r_i) &= Z^i\cvf{r_i}\log r_i = \frac{Z^i}{r_i} 
    \end{align*}
    which lies in $\Cinflog{\Sigbas}$ if and only if $Z^i$ is divisible by $r_i$.
\end{proof}

\subsection{The de Rham complex}
Let $\Sig$ be a manifold with log corners.  The \defn{logarithmic cotangent sheaf $\coVFlog{\Sig}$} is the dual $\Cinflog{\Sig}$-module of $\VFlog{\Sig}$.  There is thus a natural derivation
\[
\dd\colon \Cinflog{\Sig} \to \coVFlog{\Sig}
\]
sending a logarithmic function $f$ to the functional $\dd f \colon Z \mapsto Z(f)$ on logarithmic vector fields.  If $(r_1,\ldots,r_n,t_1,\ldots,t_k)$ is a local system of coordinates on $\Sig$, then by \autoref{prop:T-basis} the elements
\[
\dd \log(r_i) = \dlog{r_i} \qquad  \mbox{ and } \qquad \dd\log(t_j) =: \dlog{t_j},
\]
for $1 \le i \le n$ and $1 \le j \le k$, form a local basis of $\coVFlog{\Sig}$.

\begin{definition}
    Let $\Sig$ be a manifold with log corners.  The \defn{sheaf of logarithmic $j$-forms on $\Sig$} is 
    \[
    \cAlog[j]{\Sig} \defas {\bigwedge}^j_{\Cinflog{\Sig}}\left(\coVFlog{\Sig}\right).
    \]
\end{definition}
The differential $\dd$ extends uniquely to a graded derivation 
\[
\dd\colon \cAlog{\Sig} \to \cAlog[\bullet+1]{\Sig}
\]
such that $\dd^2=0$, given by the usual formula for the de Rham differential. 

\begin{lemma}
The sheaf $(\cAlog{\Sig},\dd)$ of commutative differential graded algebras is called the \defn{de Rham complex of $\Sig$}.
\end{lemma}

Note that $\cAlog[j]{\Sig}$ is a sheaf of $\Cinf{\Sig}$-modules for all $j\ge0 $, hence soft.  The hypercohomology of $(\cAlog{\Sig},\dd)$ thus reduces to the cohomology of the complex $\cAlog{}(\Sig) \defas \sect{\Sig,\cAlog{\Sig}}$ of global sections, and similarly for the hypercohomology with compact supports. 

\begin{definition}
    The \defn{de Rham cohomology} of a manifold with log corners $\Sig$ is the cohomology of the complex of forms with logarithmic singularities, denoted
    \[
    \HdR{\Sig} \defas \coH{\cAlog{}(\Sig),\dd}.
    \]
    The \defn{compactly supported de Rham cohomology of $\Sig$} is the cohomology
    \[
    \HdRc{\Sig}\defas \coH{\Alogc{\Sig},\dd}
    \]
    of the complex of compactly supported sections of $\cAlog{\Sig}$.
\end{definition}

Given a virtual morphism $\phi \colon \Sig \to \Psi$, the pullback on logarithmic functions given by \autoref{lem:functoriality-of-log-functions} induces a pushforward map $\VFlog{\Sig} \to \phi^*\VFlog{\Psi}$ on logarithmic vector fields, and therefore extends uniquely to a map $\phi^* : \uphi^{-1}\cAlog{\Psi} \to \cAlog{\Sig}$ of sheaves of commutative differential graded algebras, so that the de Rham complex and its cohomology are functorial for virtual morphisms, and the compactly supported versions are functorial for virtual morphisms whose underlying map of manifolds is proper.

\subsection{Homotopy invariance and the log de Rham theorem}\label{sec:poincare}

We now show that our de Rham cohomology agrees with the ordinary de Rham cohomology defined using the smooth forms $\cA{\uSig}$, and hence inherits its usual topological properties.  The key is to directly prove the ``homotopy invariance'' $\HdR{\Sig} \cong \HdR{\Sig\times[0)} \cong \HdR{\Sig\times\halfspc{}}$ in the spirit of the classical argument, e.g.~as presented in \cite[\S{}I.4]{Bott1982}, for which we make essential use of virtual morphisms.

    Let  $\Sig$ be a manifold with log corners.  Let $r$ and $t$ be the standard coordinates on $\halfspc{}$ and $[0)$, respectively.    Consider the natural maps
  \begin{equation}\label{eq:homotopy-setup}
\begin{tikzcd}
    \Sig \times [0) \ar[rr,"i",hook]\ar[rd,"p"',twoheadrightarrow] & &\Sig\times \halfspc{} \ar[ld,"\tp",twoheadrightarrow] \\
    & \Sig
\end{tikzcd} 
\end{equation}
induced by the projection to a point and the canonical embedding $[0) \to \halfspc{}$.  Each of these (ordinary) morphisms has a canonical left or right inverse given by a virtual morphism: we have a commutative diagram
\begin{equation}
\begin{tikzcd}
    \Sig \times [0)  & &\Sig\times \halfspc{} \ar[ll,"q"',twoheadrightarrow]  \\
    & \Sig \ar[ul,"s",hook'] \ar[ur,"\ts"',hook]
\end{tikzcd} \label{eq:homotopy-inverse-setup}
\end{equation}
where $s$ is defined by $s^*(t)=1$, $q$ is defined by $q^*(t)=r$, and $\tilde s\defas is$.  We thus have the compositions
\begin{align*}
q i &= \id{\Sig \times [0)} & ps =\tilde p\tilde s &= \id{\Sig}. 
\end{align*}
so that the maps in \eqref{eq:homotopy-inverse-setup} are one-sided inverses for the maps in \eqref{eq:homotopy-setup}.  We claim that on the level of cohomology, these maps becomes two-sided inverses. Indeed, we have the following stronger statement.
\begin{theorem}\label{thm:deRham-homotopy}
    The three operators
    \[
    p^*s^* \in \End{\cAlog{}(\Sig\times [0))}
    \]
    and
    \[
    \tp^*\ts^*,q^* i^* \in \End{\cAlog{}(\Sig\times\halfspc{})}
    \]
    are canonically cochain homotopic to the identity, and hence the morphisms \eqref{eq:homotopy-setup} and \eqref{eq:homotopy-inverse-setup} give a commutative diagram of mutually inverse isomorphisms
    \[
\begin{tikzcd}
    \HdR{\Sig \times [0)}\ar[rr,"q^*"]\ar[dr,"s^*"]  & &\HdR{\Sig\times \halfspc{}} \ar[ll,bend right=20,"i^*"']  \ar[dl,"\ts^*"]\\
    & \HdR{\Sig} \ar[ul,"p^*",bend left = 30] \ar[ur,"\tp^*"',bend right = 30]
\end{tikzcd}
\]
\end{theorem}
Before proving the theorem, let us remark on some immediate consequences of this result.

First, by repeated application of the theorem, we deduce that the cohomology of $\Sig \times \model{n}{k}$ reduces to that of $\Sig$.  In particular, taking $\Sig = \RR^m$ we obtain the following logarithmic version of the Poincar\'e lemma.
\begin{corollary}
    The log de Rham cohomology of $\RR^m \times \model{n}{k}$ is given by
    \[
    \HdR[j]{\RR^m\times \model{n}{k}} \cong \begin{cases}
        \RR & j=0 \\
        0 & j \neq 0
    \end{cases}
    \]
\end{corollary}

Since every point in a manifold with log corners has a basis of neighbourhoods isomorphic to $\RR^m\times\model{n}{k}$ for some $n,k$, we deduce the following logarithmic counterpart of de Rham's theorem, with or without compact suport.

\begin{corollary}\label{cor:log-de-Rham}
The inclusions 
\[
\RR_{\uSig} \hookrightarrow \cA{\uSig} \hookrightarrow \cAlog{\Sigbas} \hookrightarrow \cAlog{\Sig}\]
are quasi-isomorphisms.  Hence they induce natural isomorphisms  
\[
\mathsf{H}^\bullet_{\mathrm{sing}}(\uSig;\RR) \cong \HdR{\uSig} \cong \HdR{\Sigbas} \cong \HdR{\Sig}.
\]
of graded commutative algebras, and natural isomorphisms
 \[
   \mathsf{H}^\bullet_{\mathrm{sing,c}}(\uSig;\RR)\cong \HdRc{\uSig} \cong \HdRc{\Sigbas} \cong \HdRc{\Sig}.
    \]
of their graded modules.
\end{corollary}

The classical K\"unneth formula then gives the following.
\begin{corollary}
    If $\uSig$ is of finite type (e.g.~compact), then the natural map
    \[
    \HdR{\Sig} \otimes_\RR \HdR{\Psi} \to \HdR{\Sig\times\Psi}
    \]
    is an isomorphism of graded commutative algebras.
\end{corollary}

We now turn to the proof of \autoref{thm:deRham-homotopy}.  We will deal with each of the three operators in the statement in \autoref{sec:deRham-htpy1} through \autoref{sec:deRham-htpy3} below.

\subsubsection{Contracting homotopy for $p$ and $s$}\label{sec:deRham-htpy1}
We have an isomorphism of commutative differential graded $\Alog{\Sig}$-algebras:
\[
\Alog{\Sig\times[0)} \cong \Alog{\Sig}\left[\log(t), \dd\log(t)\right].
\]
The graded $\Alog{\Sig}$-linear operator
\[
h \colon \Alog{\Sig\times[0)} \to \Alog[\bullet-1]{\Sig\times[0)} \; , \; \begin{cases} \log^j(t)\mapsto 0 \\ \log^j(t)\,\dd\log(t) \mapsto \tfrac{1}{j+1}\log^{j+1}(t) \end{cases}
\]
is easily seen to satisfy
\[
\dd h + h \dd = \id{} - p^*s^*
\]
since $s^*$ sends $\log(t)$ and $\dd\log(t)$ to zero.

\subsubsection{Contracting homotopy for $q$ and $i$}\label{sec:deRham-htpy2}

Since $i^*q^* = \id{}$, we have a splitting
\[
\Alog{\Sig \times \halfspc{}} \cong \Alog{\Sig\times[0)} \oplus \ker i^*
\]
where
\[
\ker i^* \subset \Alog{\Sig\times\halfspc{}}
\]
is the subcomplex of forms that vanish on $[0)$. Under the splitting, the operator $q^*i^*$ corresponds to the projection onto $\Alog{\Sig\times[0)}$ and hence it suffices to produce a contracting homotopy for the complex $\ker i^*$.

To this end, note that since $i^* (\dd\log(r))=\dd\log(t),$ any form $\omega \in \ker i^*$ has no poles in $r$, i.e.~it can be written as 
\[
\omega = \omega_0 + \omega_1\wedge \dd r
\]
where $\omega_0$ and $\omega_1$ are logarithmic forms that do not involve $\dd r$, although their coefficient functions may depend smoothly on $r$ and polynomially on $\log r$.  Since $\log^j(r)\dd r$ is absolutely integrable near $r=0$ for all $j\ge 0$, we may define an operator
\[
h' \colon \ker i^* \to \ker i^*[-1]
\]
by the formula
\[
h'\omega \defas  \int_0^r \omega_1(r') \dd r'
\]
so that
\[
(\dd h' + h' \dd )\omega = \omega,
\]
as desired.  

\subsubsection{Contracting homotopy for $\tp,\ts$}\label{sec:deRham-htpy3}

Since $\ts = is$ and $\tp = p q$, we may compose the homotopy equivalences from the previous two subsections to obtain a homotopy between $\tp^*\ts^*$ and the identity.  This completes the proof of \autoref{thm:deRham-homotopy}.

\subsection{Relative de Rham cohomology}
\label{sec:relative-cohomology}

We may also define a version of de Rham cohomology relative to the boundary.  Here, as for classical manifolds with corners, a complication arises: since the topological boundary is not itself a submanifold, we need to replace it with its natural simplicial resolution to obtain a sensible de Rham complex.

\subsubsection{Symmetric semi-simplicial objects}\label{sec:semi-simplicial-objects}
As explained in \cite{ChanGalatiusPayne} in the algebro-geometric context, the combinatorics of boundary strata are most naturally organized using a variant of semi-simplicial sets, which they refer to as symmetric $\Delta$-complexes.  We recall the basics here to set our notation and terminology, which differs somewhat from \emph{op.~cit.}.

Let $\ascat$ denote the category of finite sets and injective maps, and $\scat\subset\ascat$ the full subcategory of nonempty sets.  For a category $\mathscr{C}$, a \defn{symmetric semi-simplicial object in $\mathscr{C}$} is a functor $Y\colon \scat^\op \to \mathscr{C}$.  From such a functor we may extract the objects $Y_n \defas Y(\{1,\ldots,n\})$, which by functoriality carry an action of the symmetric groups $\symgp{n}$, and a collection of morphisms $d_j^n \colon Y_n \to Y_{n-1}$ for $j \in \{1,\ldots, n\}$, induced by  the unique increasing map $\{1,\ldots,n-1\} \to \{1,\ldots,n\}$ whose image omits the element $j$. Since all injective maps are conjugate to increasing maps by permutations, this data determines $Y$ up to isomorphism.  Thus, as a shorthand, we denote the $\scat$-object by 
\[
Y_\bullet = \rbrac{\simp{Y_1}{Y_2}{Y_3}}
\]
where the names of the maps $d_j^n$ and symmetric group actions are left implicit.

Similarly, functors $(\ascat)^\op \to \mathscr{C}$, called \defn{augmented symmetric semi-simplicial objects}, are determined by the data
\[
(X,Y_\bullet) = \rbrac{\augsimp{X}{Y_1}{Y_2}{Y_3}}
\]
where $X$ is the image of the empty set, $Y_\bullet$ encodes the $\scat$-object defined by the restriction of the functor to $\scat\subset\ascat$, and the morphism $Y_1 \to X$, which could be denoted $d^1_1$, is obtained by functoriality from the inclusion of the empty set in $\{1\}$.

\begin{remark}
    Our indexing differs from the standard indexing of semi-simplicial sets, in that we start from one instead of zero.  This convention is chosen to better match the indexing of boundary strata.
\end{remark}

\begin{remark}
    The notion of a symmetric semi-simplicial object is a version ``without degeneracies'' of the notion of symmetric simplicial object \cite{grandis}, also known as ``oversimplicial object'' \cite{saitowild}. 
\end{remark}

\subsubsection{Symmetric semi-simplicial manifolds with log corners}

Specializing to the case in which $\mathscr{C}$ is the category of manifolds with log corners with virtual morphisms, we have the following definition.

\begin{definition}
An \defn{$\scat$-manifold with log corners} is a symmetric semi-simplicial object $\Psi_\bullet$ in the category of manifolds with log corners and virtual morphisms. 

Similarly, an \defn{$\ascat$-manifold with log corners} is an augmented symmetric semi-simplicial object $(\Sigma,\Psi_\bullet)$ in the category of manifolds with log corners and virtual morphisms.
\end{definition}

We emphasize that the structure maps of an $\scat$- or $\ascat$-manifold with log corners are virtual morphisms by default in this definition. When we want to explicitly say that the  morphisms are virtual (resp. ordinary) we will talk about virtual (resp. ordinary) $\scat$-manifolds with log corners and similarly for the augmented case.

Underlying every $\scat$-manifold with log corners $\Psi_\bullet$ is an $\scat$-manifold with corners $\underline{\Psi}_\bullet$, and similarly in the augmented case.  Our primary source of (ordinary) $\ascat$-manifolds with log corners is given by the following construction.

\begin{example}
    If $\Sig$ is a manifold with log corners then its boundary strata assemble into an ordinary $\scat$-manifold with log corners, which we denote by 
    \[
    \bdSig[\bullet] \defas \rbrac{\simp{\bdSig}{\bdSig[2]}{\bdSig[3]}}
    \]
    and an ordinary $\ascat$-manifold with log corners
    \[
    (\Sig,\bdSig[\bullet]) = \rbrac{\augsimp{\Sig}{\bdSig}{\bdSig[2]}{\bdSig[3]}}
    \]
    where the augmentation is the natural immersion $\bdSig \to \Sig$ of the boundary.
\end{example}

\subsubsection{The relative de Rham complex}\label{sec:relative-de-Rham}
If $\Psi_\bullet$ is an $\scat$-manifold with log corners, the logarithmic de Rham complexes of its components assemble into a symmetric semi-cosimplicial dg algebra, and we may define the logarithmic de Rham complex of $\Psi_\bullet$ to be its totalization, given by the homotopy limit
\begin{align*}
\Alog{\Psi_\bullet} &\defas \holim_{\scat} \Alog{\Psi_\bullet}.
\end{align*}
In concrete terms, we have a canonical quasi-isomorphism
\[
\Alog{\Psi_\bullet} \cong \bigoplus_{n \ge 1} \rbrac{\Alog{\Psi_n} \otimes \sgn{n}}^{\symgp{n}}[1-n]
\]
where $[-]$ is the usual degree shift functor, the differential is the alternating sum of the de Rham differential and the pullbacks along the face maps, $\sgn{n}$ denotes the sign representation of $\symgp{n}$, and $(-)^{\symgp{n}}$ denotes $\symgp{n}$-invariants.

Similarly for a virtual $\ascat$-manifold with log corners $(\Sig,\Psi_\bullet)$ with augmentation map denoted $i \colon \Psi_1 \to \Sig$, we define its relative logarithmic de Rham complex
\[
\Alog{\Sig,\Psi_\bullet} \defas \fibre\rbrac{i^*\colon  \Alog{\Sig} \to \Alog{\Psi_\bullet}}
\]
Concretely, we have
\[
\Alog{\Sig,\Psi_\bullet} \cong \Alog{\Sig} \oplus \bigoplus_{n \ge 1}(\Alog{\Psi_n} \otimes \sgn{n})^{\symgp{n}}[-n]
\]
where once again the differential is the sum of the de Rham and \v{C}ech differentials.

We also have the de Rham complex $\A{\uPsi_\bullet}$ and $\A{\uSig,\uPsi_\bullet}$ of the underlying symmetric semi-simplicial manifolds with corners, and the inclusion of smooth forms into logarithmic forms give quasi-isomorphisms
\[
\begin{tikzcd}
    \A{\uPsi_\bullet} \ar[r,hook,"\sim"] &\Alog{\Psi_\bullet} &\A{\uSig,\uPsi_\bullet} \ar[r,hook,"\sim"] &\Alog{\Sig,\Psi_\bullet}
\end{tikzcd}
\]
which also induce quasi-isomorphisms of the compactly supported forms.

In the special case in which $\Psi_\bullet = \bdSig[\bullet]$ is the semi-simplicial boundary, a standard inclusion/exclusion argument gives the following.
\begin{proposition}\label{prop:BdR-comparison-manifolds-with-log-corners}
    The de Rham cohomology of $\bdSig[\bullet]$ with/without compact supports is naturally isomorphic to the corresponding singular cohomology groups of the topological boundary $\partial_{\mathrm{top}}\uSig$: 
    \begin{align*}
    \HdR{\partial^\bullet\Sig}&\cong \mathsf{H}^\bullet_{\mathrm{sing}}(\partial_{\mathrm{top}}\uSig;\RR) &
    \HdRc{\partial^\bullet\Sig}&\cong \mathsf{H}^\bullet_{\mathrm{sing,c}}(\partial_{\mathrm{top}}\uSig;\RR).
    \end{align*}
    and similarly for the relative cohomology
    \begin{align*}
    \HdR{\Sig,\partial^\bullet\Sig}&\cong \mathsf{H}^\bullet_{\mathrm{sing}}(\uSig,\partial_{\mathrm{top}}\uSig;\RR) &
    \HdRc{\Sig,\partial^\bullet\Sig}&\cong \mathsf{H}^\bullet_{\mathrm{sing,c}}(\uSig,\partial_{\mathrm{top}}\uSig;\RR).
    \end{align*}
\end{proposition}

There is also a smaller model for the relative cohomology.  Namely, for $p \ge 0$, let 
\[
\Alog[p]{\Sig,\bdSig[\bullet]}_0 := \ker (i^* \colon \Alog[p]{\Sig} \to \Alog[p]{\bdSig})
\]
be the set of $p$-forms that vanish identically on the boundary. Note that since $i^*$ commutes with $\dd$, we may identify $(\Alog[\bullet]{\Sig,\bdSig[\bullet]},\dd)$ with a subcomplex of the total complex computing the relative cohomology.  We then have the following.
\begin{lemma}
\label{rem:continuous-relative-cohomology}
The \v{C}ech complex
\[
\Alog[p]{\Sig,\bdSig[\bullet]} = \rbrac{\Alog[p]{\Sig} \to \Alog[p]{\bdSig} \to \Alog[p]{\bdSig[2]} \to \cdots}
\]
is a resolution of $\Alog[p]{\Sig}_0 \subset \Alog[p]{\Sig}$. Hence the inclusion of $\Alog{\Sig,\bdSig[\bullet]}_0$ in the total complex $\Alog{\Sig,\bdSig[\bullet]}$ is a quasi-isomorphism.
\end{lemma}
\begin{proof}
The problem is local, so it reduces to the case in which $\Sig$ is an open subset of $\halfspc{n}\times[0)^k$.  In this case, the pullback to boundary faces identifies the standard coordinate basis of $\Alog{\Sig}$ with the corresponding bases of $\Alog{\bdSig[j]}$ for all $j > 0$.  Hence these bases decompose the \v{C}ech complex  $\Alog[p]{\Sig,\bdSig[\bullet]}$ as a direct sum of ${n+k \choose p}$ copies of the \v{C}ech complex $\Alog[0]{\Sig,\bdSig[\bullet]} = \Cinflog{}(\bdSig[\bullet])$ of the algebras of logarithmic functions.  By \autoref{lem:boundary-log-principal-ideal}, the latter is identified with the \v{C}ech complex of the regular sequence of elements $r_1,\ldots,r_n$ in the ring $\Cinflog{}(\Sig)$, which is therefore a resolution by standard commutative/homological algebra.
\end{proof}

\section{Regularized integration}
\label{sec:integration}

\subsection{Overview of the strategy}
In this section, we use the cohomological formalism developed thus far to define the regularized integral of forms with logarithmic singularities.  The strategy is as follows. 

\subsubsection{Classical case}\label{sec:classical-integral} Recall that if $\uSig$ is an \emph{oriented} manifold with corners, and $\uomega$ is a compactly supported, smooth, top-degree form on $\uSig$, the integral $\int_{\uSig} \uomega$ can be defined locally using the Lebesgue measure in charts, and then globalized using a partition of unity. 

For degree reasons, such a form $\uomega$ is closed, and its pullback to $\ubdSig$ is zero, so that it defines a class in the relative cohomology $\HdRc[n]{\uSig,\ubdSig}$.  By Stokes' formula, the integral only depends on this class, i.e.~it descends to a functional
\[
\int_{\uSig} : \HdRc[n]{\uSig,\ubdSig}  \to \RR,
\]
which is, moreover, an isomorphism if $\uSig$ is connected.  

Indeed, we could alternatively \emph{define} the integral as the functional on $\Ac{\uSig}$ obtained by first applying the projection $\Ac{\uSig} \to \HdRc[n]{\uSig,\ubdSig}$, then applying de Rham's isomorphism $\HdRc[n]{\uSig,\ubdSig}\cong \mathsf{H}^\bullet_\mathrm{c}(\uSig,\partial_{\mathrm{top}}\uSig;\RR)$, and finally applying the functional $\mathsf{H}^\bullet_\mathrm{c}(\uSig,\partial_{\mathrm{top}}\uSig;\RR) \to \RR$ given by pairing against the fundamental class of $(\uSig,\ubdSig)$ in Borel--Moore homology.  The advantage of this more abstract definition is that it makes sense for any cochain complex computing the relative cohomology.

For instance, using the description of the relative cohomology via totalization of the  $C^\infty$ de Rham complexes on strata $\Ac{\uSig,\ubdSig[\bullet]}$ similar to \autoref{sec:relative-de-Rham}, a relative cohomology class may be represented by a tuple $(\uomega_j)_{j \ge0}$ with $\uomega_j \in (\Ac[n-j]{\ubdSig[j]}\otimes\sgn{j})^{\symgp{j}}$ a smooth form of top degree on the $j$th boundary, equivariant for the sign-twisted action of the symmetric group.  The integral of such a tuple is expressed in terms of the integrals over strata by the formula
\[
\int_{\uSig}\uomega_\bullet= \sum_{j \ge 0} \frac{(-1)^j}{j!} \int_{\ubdSig[j]}\uomega_j,
\]
which reduces to the integral of the top form $\omega = \omega_0$ if $\omega_j = 0$ for $j>0$.  
\subsubsection{Logarithmic case}
Now if $\Sig$ is a manifold with log corners and $\omega$ is a form with logarithmic singularities, it is not clear how to define the integral in coordinates due to possible divergences, but the logarithmic de Rham theorem still gives an isomorphism $\HdRc{\Sig,\bdSig} \cong \mathsf{H}^\bullet_{\mathrm{c}}(\uSig,\partial_{\mathrm{top}}\uSig;\RR)$, suggesting that the cohomological definition of the integral still makes sense.  However, the presence of phantoms on the boundary means that a top-degree form need no longer vanish on the boundary, e.g.~the form $\dlog{r}$ on $\halfspc{}$ pulls back to the phantom form $\dlog{t}\neq 0$ on $\partial\halfspc{} = [0)$.  Hence \emph{a priori} not all logarithmic forms represent relative cohomology classes, so that the cohomological approach only gives an obvious definition of the integral for logarithmic forms whose pullback to the boundary is zero.   In fact, as we explain  in \autoref{prop:convergent-integrals} below, these are exactly the forms for which the  ordinary integral of $\omega$ over the interior $\intSig$ converges absolutely, and in this case, the cohomological formalism recovers this absolutely convergent integral.  This extends the cohomological definition of the integral only to forms which are absolutely integrable but not smooth, such as $\log r\, \dd r$. 

Therefore, in order to tackle (possibly divergent) integrals of arbitrary logarithmic forms, the relative cohomology group $\HdRc{\Sig,\bdSig}$ is not the right object to consider. This should not be surprising since as explained in the introduction, the value of a ``regularized'' integral depends on extra tangential data on the boundary of $\Sig$, namely a scale on $\partial\Sigma$ in the sense of \autoref{def:scale}. 
Moreover, to have a good theory, with Stokes formula, etc., we need a compatible family of scales along boundary strata of all codimensions.  We package these data and compatibilities below in the notion of a ``regularization''.  Using the latter, we define a diagram of manifolds with log corners and virtual morphisms whose log de Rham cohomology still compares to $\mathsf{H}^\bullet_{\mathrm{c}}(\uSig,\partial_{\mathrm{top}}\uSig;\RR)$ and which contains the classes of \emph{all} logarithmic forms of degree $\dim \uSig$. This allows us to  define their regularized integrals consistently. For forms that are absolutely integrable, we recover the classical integral, which is independent of the regularization.  However, in general, the regularized integral does depend on the choice of regularization.

\subsection{Regularization via scales}
The precise definition of a regularization is as follows.

\begin{definition}\label{def:regularization}
Let $\Sig$ be a manifold with log corners. A \defn{regularization of $\Sig$} is a tuple $s = (s_{\bdSig[j]})_{j \ge 0}$ consisting of a scale $s_{\bdSig[j]}$ on $\bdSig[j]$ for each $j \ge 0$,  which is invariant under the natural action of $\symgp{j}$ and such that all the natural morphisms in the diagram
\[
\begin{tikzcd}[column sep = .8cm, row sep = 1cm,baseline=-0.9cm]
(\Sig,\bdSig[\bullet]): &[-.5cm] \Sig &[.5cm] \bdSig \ar[l,"i"'] &[.4cm] \bdSig[2] \ar[l,shift left] \ar[l,shift right]  &[.5cm] \bdSig[3] \ar[l,shift left = 2]\ar[l,shift left = 0]\ar[l,shift left = -2] &[.5cm] \cdots \ar[l,shift left = 3]\ar[l,shift left = 1]\ar[l,shift left = -1]\ar[l,shift left = -3] 
\end{tikzcd}
\]
are scale-preserving. It is \defn{nondegenerate} if all of the scales $s_{\partial^j\Sig}$ are.  
A \defn{regularized manifold with log corners} is a pair $(\Sig,s)$ where $\Sig$ is a manifold with log corners and $s$ is a regularization of $\Sig$.
\end{definition}

Let us unpack this definition by applying \autoref{def:scale-preserving} and the discussion following it to every morphism in the diagram $(\Sig,\bdSig[\bullet])$. We see that a regularization $s$ of $\Sig$ gives rise to a diagram which we denote by
\[
\begin{tikzcd}[column sep = .4cm, row sep = 1cm,baseline=-0.9cm]
(\Sigbas,\bas{(\bdSig[\bullet])};s): &[-.5cm] \Sigbas &[.3cm] \bas{(\bdSig)}\ar[l,"i'"'] &[.4cm] \bas{(\bdSig[2])} \ar[l,shift left] \ar[l,shift right]  &[.4cm] \bas{(\bdSig[3])} \ar[l,shift left = 2]\ar[l,shift left = 0]\ar[l,shift left = -2] &[.4cm] \cdots \ar[l,shift left = 3]\ar[l,shift left = 1]\ar[l,shift left = -1]\ar[l,shift left = -3] 
\end{tikzcd}
\]
where each structure morphism $\bas{(\bdSig[j+1])}\to \bas{(\bdSig[j])}$ is induced by the scale $s_{\bdSig[j+1]}$ as in \eqref{eq:scale-preserving-concrete}. This diagram is a $\ascat$-manifold with log corners, and the scales $s_{\bdSig[j]}$ assemble into a morphism of $\ascat$-manifolds with log corners
\begin{equation}
\begin{tikzcd}[column sep = .5cm, row sep = 1cm]
(\Sig,\bdSig[\bullet]): &[-.5cm] \Sig &[.3cm] \bdSig \ar[l,"i"'] &[.4cm] \bdSig[2] \ar[l,shift left] \ar[l,shift right]  & \bdSig[3] \ar[l,shift left = 2]\ar[l,shift left = 0]\ar[l,shift left = -2] \cdots \\ 
(\Sigbas,\bas{(\bdSig[\bullet])};s): \ar[u,"s_\bullet"]& \Sigbas \ar[u,"s_\Sig"'] & \bas{(\bdSig)} \ar[l,"i'"']\ar[u,"s_{\bdSig}"'] & \bas{(\bdSig[2])}\ar[l,shift left] \ar[l,shift right]\ar[u,"s_{\bdSig[2]}"']  & \bas{(\bdSig[3])} \ar[l,shift left = 2]\ar[l,shift left = 0]\ar[l,shift left = -2]\ar[u,"s_{\bdSig[3]}"']  \cdots 
\end{tikzcd}
\label{eq:regularization-diagram}
\end{equation}
The datum of the regularization $s$ of $\Sig$ is equivalent to the structure of $\ascat$-manifold with log corners $(\Sigbas,\bas{(\bdSig[\bullet])};s)$ and the morphism $s_\bullet$.
In particular, note that the boundary of a regularized manifold with log corners admits a unique regularization such that the canonical maps $\partial^j\bdSig \to \bdSig[j+1]$ are scale-preserving. Similarly, if $(\Sig,s)$ and $(\Sig',s')$ are regularized manifolds with log corners, then the product inherits a canonical regularization which we denote by $(\Sig\times\Sig',s\times s')$.

\begin{example}
For $\Sig=\halfspc{}$, since $\Sig=\Sigbas$ and $\partial^j\Sig=\varnothing$ for $j\ge 2$, a regularization for $\Sig$ is the same thing as a scale for $\partial\Sig=[0)$, i.e.~a virtual morphism $s:\{0\}\to [0)$, i.e.~a tangential basepoint $c\,\partial_r|_0$ with $c>0$, by \autoref{prop:tangential-basepoints-as-weak-morphisms}. It gives rise to the following morphism of $\ascat$-manifolds with log corners.
\[
\begin{tikzcd}[row sep = 1cm,baseline=-0.9cm]
(\Sig,\partial^\bullet\Sig): &[-.5cm] {[0,\infty)} &[.5cm] {[0)} \ar[l,"i"'] &[.4cm] \varnothing \ar[l,shift left] \ar[l,shift right] \cdots \\
(\Sigbas,\bas{(\bdSig[\bullet])};s): \ar[u,"s_\bullet"]& {[0,\infty)} \ar[u,equal] & \{0\} \ar[l,"i'=i\circ s"']\ar[u,"s"'] & \varnothing \ar[l,shift left] \ar[l,shift right] \cdots
\end{tikzcd}\qedhere
\]
\end{example}

When unpacked in an example, the compatibility conditions relating the scales on the various boundary faces exhibit some subtleties; these are visible already in the case of a quadrant, as follows.
  \begin{figure}
        \centering
        \begin{subfigure}[t]{0.3\textwidth}
        \centering
        \def\La{0.3}
        \def\Lb{0.4}
        \begin{tikzpicture}[scale=1]
            \draw[white,fill=white!90!black] (0,0) rectangle (2.1,2.1);
            \draw[thick,->] (0,0) -- (2.2,0);
            \draw[thick,->] (0,0) -- (0,2.2);
            \draw[fill] (0,0) circle (0.05);
            \draw[dashed] (\La,\Lb) -- (\La,2.1);
            \draw[dashed] (\La,\Lb) -- (2.1,\Lb);
            \foreach \x in {0.6,0.9,...,2.1} {
            \draw[blue,thick,->] (\x,0) -- (\x,\Lb);
            \draw[blue,thick,->](0,\x) -- (\La,\x);
            }
            \draw[red,thick,->] (0,0) -- (\La,\Lb);
        \end{tikzpicture}
\caption{Nondegenerate}\label{fig:quadrant-regularization-nondegen}
        \end{subfigure}
        \begin{subfigure}[t]{0.3\textwidth}
        \centering
        \def\La{0.4}
        \def\Lb{0.3}
        \begin{tikzpicture}[scale=1]
            \draw[white,fill=white!90!black] (0,0) rectangle (2.1,2.1);
            \draw[thick,->] (0,0) -- (2.3,0);
            \draw[thick,->] (0,0) -- (0,2.3);
            \draw[fill] (0,0) circle (0.05);
            \draw[dashed] (0,0) -- (2,\Lb/\La*2);
            \draw[red,thick,->] (0,0) -- (\La,\Lb);
            \clip (0,0) rectangle (2.1,2.1);
            \foreach \x in {0.5,0.8,...,2.4} {
            \draw[blue,thick,->] (\x,0) -- (\x,{\Lb/\La*\x});
            \draw[blue,thick,->](0,\x) -- ({\La/\Lb*\x},\x);
            }
        \end{tikzpicture}
        \caption{Linear}\label{fig:quadrant-regularization-linear}
        \end{subfigure}
        \begin{subfigure}[t]{0.3\textwidth}
        \centering
        \def\La{1}
        \def\Lb{1}
        \def\Aa{2}
        \def\Ab{2}
        \begin{tikzpicture}[scale=2]
            \draw[white,fill=white!90!black] (0,0) rectangle (1.05,1.05);
            \draw[thick,->] (0,0) -- (1.1,0);
            \draw[thick,->] (0,0) -- (0,1.1);
            \draw[fill] (0,0) circle (0.025);
            \foreach \x in {0.3,0.5,...,1} {
            \draw[blue,thick,->] (\x,0) -- (\x,{1/(1+\x^4)*\Lb/\La^\Aa*\x^\Ab});
            \draw[blue,thick,->](0,\x) -- ({1/(1+\x^2)*\La/\Lb^\Ab*\x^\Aa},\x);
            }
            \draw[red,thick,->] (0,0) -- (\La,\Lb);
    \draw[dashed,domain=0:1.1,smooth,variable=\t] plot ({\t},{1/(1+\t^4)*\Lb/\La^\Aa*\t^\Ab});
    \draw[dashed,domain=0:1,smooth,variable=\t] plot ({1/(1+\t^2)*\La/\Lb^\Ab*\t^\Aa},\t);

        \end{tikzpicture}    \caption{Nonlinear}\label{fig:quadrant-regularization-nonlin}
        \end{subfigure}
        \caption{Regularizations of the quadrant $\Sig = \halfspc{2}$, showing the scale on  $\bdSig$ in blue, and the scale on $\bdSig[2]$ in red.}
        \label{fig:quadrant-regularization}
   \end{figure}

\begin{example}
Consider the standard corner $\Sig \defas \halfspc{2}$ with coordinates $(r_1,r_2)$.  Let $t_1,t_2$ be the phantom coordinates obtained by restricting $r_1,r_2$ to their vanishing loci. The boundary inclusions have the form
\[
\augsimpshort{\halfspc{2}_{r_1,r_2}}{{[0)_{t_1}\times\halfspc{}_{r_2} \sqcup \halfspc{}_{r_1}\times[0)_{t_2}}}
{{[0)^2_{t_1,t_2} \sqcup [0)^2_{t_2,t_1}}}{\varnothing}
\]
where we have used a subscript to denote the coordinates on each space.  Then as illustrated in \autoref{fig:quadrant-regularization}, a regularization of  $\Sig$ consists of the following data:
\begin{itemize}
    \item A scale on $\Sig$;   this must be the identity map since $\Sig=\Sigbas$ is basic.
    \item A scale on $\bdSig = {[0)_{t_1}\times\halfspc{}_{r_2} \sqcup \halfspc{}_{r_1}\times[0)_{t_2}}$, given by 
    \[
    s_{\bdSig} = \begin{cases}
f_2(r_2)r_2^{a_2}\cvf{t_1} & \textrm{on }[0)_{t_1} \times \halfspc{}_{r_2} \\
 f_1(r_1)r_1^{a_1}\cvf{t_2} & \textrm{on }[0,\infty)_{r_1}\times [0)_{t_2}
 \end{cases}
    \]
    where $f_1,f_2$ are strictly positive functions on $\halfspc{}$ and $a_1,a_2 \in \ZZ$.  Thus $s_{\bdSig}^* t_1 = f_2(r_2) r_2^{a_2}$ and $s_{\bdSig}^*t_2 = f_1(r_1) r_1^{a_1}$ on the respective components.
    \item A scale on $\bdSig[2] = [0)^2_{t_1,t_2} \sqcup [0)^2_{t_2,t_1}$ that is invariant under the $\symgp{2}$-action induced by the involution that identifies the two components.  It is thus given on both components by 
    \[
    s_{\bdSig[2]} = \lambda_1 \cvf{t_1} + \lambda_2 \cvf{t_2}
    \]
    where $\lambda_1,\lambda_2 > 0$, so that $s_{\bdSig[2]}^*t_j = \lambda_j$.
\end{itemize}
These data must satisfy the consistency condition that the boundary inclusions are scale-preserving.  For the map $\bdSig \to \Sig$ this is vacuous, but for the maps $i \colon \bdSig[2]\to\bdSig$ this is a nontrivial condition; we require $i^*s_{\bdSig}^*t_j = s_{\bdSig[2]}^*i^*t_j$ for $j=1,2$.  This is equivalent to the following equation of phantom tangent vectors:
\[
f_1(0) \lambda_1^{a_1}\cvf{t_2} + f_2(0)\lambda_2^{a_2}\cvf{t_1} = \lambda_1 \cvf{t_1} + \lambda_2 \cvf{t_2},
\]
By linear independence of $\cvf{t_1},\cvf{t_2}$, this is equivalent to the equations 
\begin{align*}
f_1(0) &= \frac{\lambda_2}{\lambda_1^{a_1}}  &
f_2(0) = \frac{\lambda_1}{\lambda_2^{a_2}}
\end{align*}
of positive real numbers, or equivalently to the linear system
\[
\begin{pmatrix}
    \log f_1(0) \\ \log f_2(0)
\end{pmatrix} = \begin{pmatrix}
    -a_1 & 1 \\
    1 & -a_2
\end{pmatrix}\begin{pmatrix}
    \log \lambda_1 \\ 
    \log \lambda_2
\end{pmatrix}
\]
for their logarithms.  Since $\bdSig[3]=\varnothing$, there are no further constraints. 

Note that if $a_1a_2\neq 1$, the constants $\lambda_1,\lambda_2$ are uniquely determined by $f_1(0),f_2(0)$, for any values of the latter.  This applies, in particular, to the case $a_1=a_2=0$ in which the regularization is nondegenerate (\autoref{fig:quadrant-regularization-nondegen}); in this case, $s_{\bdSig}$ is an arbitrary positive trivialization of the normal bundle of $\ubdSig$, and it completely determines the regularization.  

On the other hand, if $a_1=a_2 = \pm 1$, the equations have a solution if and only if $f_1(0)f_2(0)^{\pm 1}=1$; when this condition holds, the constant $\lambda_1$ can be chosen arbitrarily, and $\lambda_2 = \lambda_1^{\pm 1}f_1(0)$.  Hence, in this case, the scale on $\bdSig$ is not arbitrary, and does not completely determine the regularization.  Geometrically, the scales bisect the quadrant into triangles and the tangential basepoint at zero points along the ``diagonal'' edge; see \autoref{fig:quadrant-regularization-linear} for the case $(a_1,a_2)=(1,1)$.
\end{example}

The (non)degeneracy of a regularization will not play any role in our main results below.    However, the above example illustrates a useful aspect of nondegerate regularizations: they are specified by the scales assigned to faces of codimension $\le 1$, making them relatively easy to construct, as follows.
\begin{proposition}
Let $\Sig$ be a manifold with log corners, and let $N$ denote the normal line bundle of the immersion $\ubdSig \to \uSig$.  Then there is a natural bijection between nondegenerate regularizations of $\Sig$ and pairs $(s,s')$ where $s$ is a nondegenerate scale on $\Sig$ and $s'$ is a positive section of $N$. 
\end{proposition}

\begin{proof}
Given a pair $(s,s')$ we construct a scale on all iterated boundaries $\bdSig[k]$ as follows.  Note that the phantom tangent bundle is given by
\[
\tbphan{\bdSig[k]} \cong i^* \tbphan{\Sig}\oplus i_1^* N \oplus \cdots \oplus  i_k^* N
\]
where $i \colon \ubdSig[k] \to \uSig$  and $i_1,\ldots,i_k \colon \ubdSig[k] \to \ubdSig$ are the canonical immersions.  The section $s_k \defas i^* s + i_1^*s' + \cdots + i_k^* s'$ then gives a nondegenerate scale on $\bdSig[k]$ such that the structure maps of the diagram $(\Sig,\partial^\bullet\Sig)$ are scale-preserving, so that the tuple
$(s_0,s_1,s_2,\ldots)$ defines a regularization.  It is, moreover, the unique regularization having  $s_0 =s$ and $s_1 = i^*s + s'$ as the scales on $\Sig$ and $\partial^\bullet\Sig$, respectively.
\end{proof}

\begin{corollary}
Every manifold with log corners admits a nondegenerate regularization.
\end{corollary}

\subsection{The regularized integral} 
Let $(\Sig,s)$ be a regularized manifold with log corners; we make no assumption about the (non)degeneracy of $s$.  We have an induced morphism of $\ascat$-manifolds with log corners
\[
s_\bullet \colon (\Sigbas,\bas{(\bdSig[\bullet])};s) \to (\Sig,\bdSig[\bullet]),
\]
and arguing as in \S\ref{sec:relative-de-Rham}, we obtain an isomorphism at the level of compactly supported relative de Rham cohomology
\[
\HdRc{\Sigbas,\bas{(\bdSig[\bullet])};s} \cong \HdRc{\uSig,\ubdSig}.
\]
We use this isomorphism to define regularized integral, as follows.

Let $n \defas \dim \uSig$ and let
\[
i' := p_{\Sig} \circ i \circ s_{\bdSig} : \bas{(\bdSig)} \to \Sigbas
\]
be the immersion induced by the scale $s_{\bdSig}$ on $\bdSig$, the projection $p_\Sig : \Sig \to \Sigbas$ and the natural immersion $i : \bdSig \to \Sig$.  If $\omega \in \Alogc[n]{\Sig}$ is a compactly supported form on $\Sig$, then the form
\[
\omega' \defas s_\Sig^*\omega \in \Alogc[n]{\Sigbas}
\]
satisfies
\[
\dd \omega' =0 \qquad \textrm{and} \qquad (i')^*\omega'= 0
\]
for degree reasons.  Hence, it defines a cocycle in the compactly supported relative de Rham double complex $\Alogc{\Sigbas,\bas{(\bdSig[\bullet])};s}$, giving a class in relative cohomology which we denote by
\[
[\omega,s] \in \HdRc[n]{\uSig,\ubdSig}.
\]

\begin{definition}\label{def:regularized-integral}
    Let $(\Sig,s)$ be an oriented and regularized manifold with log corners whose underlying manifold with corners $\uSig$ has dimension $n$.  The \defn{regularized integral} is the linear functional
    \[
    \int_{(\Sig,s)} \colon \Alogc[n]{\Sig} \to \RR
    \]
    defined by 
    \[
    \int_{(\Sig,s)} \omega \defas \int_{\uSig}[\omega,s]
    \]
    for all $\omega \in \Alogc[n]{\Sig}$, where $\int_{\uSig}$ is the integration functional on $\HdRc{\uSig,\ubdSig}$ recalled in \autoref{sec:classical-integral}.      If $\omega$ is a form of degree different from $n$, we set  $\int_{(\Sig,s)}\omega \defas 0$ as usual.
\end{definition}

Unpacking the definition, this means that the regularized integral can be computed as follows.  In general, an element in $\Alogc[n]{\Sig,\bas{(\bdSig[\bullet])};s}$ is a tuple $(\omega_j)_{j \ge 0}$, where $\omega_j \in (\Alogc[n-j]{\bas{(\bdSig[j])}}\otimes \sgn{j})^{\symgp{j}}$ is a top-degree form on the $j$th basic boundary $\bas{(\bdSig[j])}$, equivariant for the sign-twisted action of $\symgp{j}$.  In particular, we may view our given form $\omega \in \Alogc[n]{\Sig}$ as such a tuple with $\omega_0 = \omega$ and $\omega_j = 0$ for $j > 0$.

The logarithmic de Rham theorem implies that any such tuple is cohomologous to a tuple $(\uomega_j)_{j \ge0}$ lying in the subcomplex of smooth forms $\Ac{\uSig,\partial^\bullet\uSig}$, i.e.~that there exists a tuple $(\eta_j)_{j \ge 0}$ with $\eta_j \in  (\Alogc[n-j-1]{\bas{\bdSig[j]}}\otimes \sgn{j})^{\symgp{j}}$ such that
\begin{align*}
\omega_0 &= \uomega_0 + \dd \eta_0  \\
\omega_1 &= \uomega_1 + \dd \eta_1 + (i'_1)^*\eta_0 \\
\omega_2 &= \uomega_2 + \dd \eta_2 + (i'_2)^*\eta_1 \\
&\ \ \vdots
\end{align*}
where $i'_j : \bas{(\bdSig[j])} \to \bas{(\bdSig[j-1])}$ is the virtual morphism induced by the boundary immersion and the scale on $\bdSig[j]$.  Then by definition of the integration function on $\HdRc{\uSig,\ubdSig}$, we have
\[
\int_{(\Sig,s)} \omega = \int_{\uSig} \uomega = \sum_j \frac{(-1)^j}{j!} \int_{\ubdSig[j]} \uomega_j
\]
where the right hand side is the ordinary integral of compactly supported smooth forms over the boundary strata with their induced orientations.

Note that this formula requires computing integrals along strata of all codimensions $j$ for which $\uomega_j \neq 0$.  The simplest case is when $\uomega_j =0$ for $j >0$, for which we obtain the following.

\begin{proposition}\label{lem:how-to-compute-regularized-integral}
Let $\Sig = \Sigbas$ be a basic manifold with log corners of dimension $n$, let $s$ be a regularization of $\Sig$, and let $i' : \bas{\bdSig} \to \Sig$ denote the virtual morphism induced by the boundary immersion and the scale $s_{\bdSig}$. Let $\omega \in \Alogc[n]{\Sig}$ be a compactly supported logarithmic form of top degree. Suppose that $\eta\in \Alogc[n-1]{\Sigbas}$ is such that $(i')^*(\eta)=0$ and $\uomega := \omega + \dd \eta$ is smooth. Then we have the equality
        \[
\int_{(\Sig,s)}\omega = \int_{\uSig}\underline{\omega},
        \]
        i.e.\ the regularized integral of $\omega$ is the ordinary integral of $\uomega$.
\end{proposition}

\begin{remark}
Note that if a form $\eta$ as in \autoref{lem:how-to-compute-regularized-integral} can be found, then the regularized integral depends only on the scale $s_{\bdSig}$ on the boundary.  Since degree considerations force $(i')^*\omega = 0$, the existence of such $\eta$ is immediate whenever the natural inclusion $(\ker (i')^*,\dd) \hookrightarrow  \Alogc{\Sig,\bdSig[\bullet]}$ is a quasi-isomorphism.  For instance, we were able to establish this quasi-isomorphism  under the additional assumption that the regularization is nondegenerate.  However, for degenerate regularizations, we do not know whether a form $\eta$ as in the proposition can always be found.  The difficulty is that the natural analogue of \autoref{rem:continuous-relative-cohomology} can fail: the  \v{C}ech complex associated with the $\ascat$-manifold with log corners $(\Sig,\bdSig[\bullet];s)$  depends on the choice of regularization, and can fail to be a resolution when the regularization is degenerate.  This highlights a virtue of the full double complex model for the relative cohomology: it gives a systematic theory of regularized integration that treats all regularizations in the same manner.
\end{remark}

Since the isomorphism between smooth and logarithmic de Rham cohomology is natural and compatible with products, we immediately deduce that the basic identities of integration remain true for our regularized integral.

\begin{corollary}[Change of variables]
    For an open embedding $j : \Psi \hookrightarrow \Sig$  and a form $\omega \in \Alogc{\Psi}$,  let   $j_*\omega \in \Alogc{\Sig}$ be the extension by zero of $\omega$.  Then
    \[
    \int_{(\Psi,j^*s)} \omega = \int_{(\Sig,s)}j_*\omega.
    \]
\end{corollary}
\begin{corollary}[Fubini's formula]
    Suppose that $(\Sig,s)$ and $(\Sig',s')$ are oriented, regularized manifolds with log corners, $\omega \in \Alogc{\Sig}$ and $\omega' \in \Alogc{\Sig'}$.  Then
    \[
\int_{(\Sig\times\Sig',s\times s')}\omega \wedge \omega' = \rbrac{\int_{(\Sig,s)} \omega } \cdot \rbrac{\int_{(\Sig',s')}\omega'}
    \]
\end{corollary}

\begin{corollary}[Stokes' formula]\label{cor:reg-stokes}
    If $\eta \in \Alogc[\bullet]{\Sig}$ then 
    \[
    \int_{(\Sig,s)} \dd \eta = \int_{\partial(\Sig,s)} \eta.
    \]
\end{corollary}

\begin{corollary}
    The regularized integral defines a map of cochain complexes
    \[
    \int_{(\Sig,s)} \colon \Alogc{\Sig} \to \RR[-n]
    \]
    If, in addition, $\Sig$ is connected, this gives an isomorphism $\HdRc[n]{\Sig} \cong \RR.$
\end{corollary}

\begin{example}\label{ex:regint-interval}
    Let $a>0$, let $r$ be the standard coordinate on $\Sig = [0,a]$, and consider a logarithmic one-form of the form
    \[
    \omega = f(r)\dlog{r} = f(0)\dlog{r} + \tf(r)\dd r
    \]
    where $f(r) = f(0) + r \tf(r)$ for a smooth function $\tf$.     If $f(0)=0$, the form $\omega$ is smooth and we have
    \[
    \int_{(\Sig,s)} \omega = \int_0^a \tf(r)\,\dd r 
    \] 
    for any choice of regularization $s$ of $\Sig$.

    On the other hand, if $f(0) \neq 0$, the integral is divergent and will depend on the regularization.  The latter is determined by a scale on the boundary, or equivalently virtual morphisms $s_0 = \lambda\,\cvf{r}|_0 \colon \{0\} \to [0,a]$ and $s_a = -\mu\,\cvf{r}\colon \{a\} \to [0,a]$  for some $\lambda,\mu>0$.  We then have the regularized pullbacks
   \begin{align*}
   s_0^* \log(r) &= \log(\lambda),  & s_a^*(\log(r)) &= \log(a),
   \end{align*}
   where the second of these is independent of the scale since $\log(r)$ is smooth at $a$.  
   The divergent part of the regularized integral can be computed using the regularized Stokes formula  
   \begin{align*}
   \int_{(\Sig,s)} \dlog{r} = \int_{(\Sig,s)} \dd(\log(r)) 
   = \int_{\partial(\Sig,s)}  \log(r)
   = s_a^*\log(r)-s_0^*\log(r)
   = \log(a/\lambda),
   \end{align*}
   so that by linearity of the integral we have
   \[
   \int_{(\Sig,s)}\omega = f(0) \int_{(\Sig,s)} \dlog{r} + \int_{(\Sig,s)} \tf(r)\,\dd r = f(0) \log(a/\lambda) + \int_0^a \tf(r)\,\dd r 
   \]   
   Classically, the same result would be obtained by introducing a cutoff parameter $\epsilon>0$, computing the convergent integral
   $$\int_{\lambda\epsilon}^{a-\mu\epsilon} \dlog{r} = \log(a/\lambda)-\log(\epsilon) + O(\epsilon),$$
   and formally discarding $\log(\epsilon)$ in the limit as $\epsilon\to 0$.  Either way, the result depends on the parameter $\lambda$ determining the scale at $r=0$, but is independent of the parameter $\mu$ determining the scale at $r=a$.  The reason is that $\omega$ has a pole at $r=0$ but is smooth at $r=a$. 
\end{example}

\begin{example}
Let $\Sig = \model{}{}$ with coordinates $(r,u)$ and let $(t,u)$ be the induced phantom coordinates on $\bdSig = [0)^2$.  For a smooth function $f(r)$ with compact support, an integral of the form
\[
\int_{\Sig} f(r)\dlog{u}
\]
has no classical meaning due to the phantom $u$; rather, the definition of the regularized integral in this context requires us to first convert the phantoms to functions using a scale on $\Sig$, as follows.

A regularization of $\Sig$ consists of a scale on $\Sig$ and a compatible scale on $\bdSig$; these must have the form
\[
s_\Sig = g(r) r^j \cvf{u} \qquad s_{\bdSig} = \lambda \cvf{t} + g(0)\lambda^j \cvf{u}
\]
where $g > 0$ is a positive smooth function, $j \in \ZZ$, and $\lambda > 0$ is a constant.  By definition, the regularized integral is given by
\[
\int_{(\Sig,s)}f(r)\dlog{u} = \int_{(\Sigbas,\bas{s})} s_\Sig^*\rbrac{f\dlog{u}} =  \int_{(\halfspc{},s)} f(r) \rbrac{j \dlog{r} + \dlog{g(r)}}
\]
The rightmost integral is a classical regularized integral, and can be computed as in \autoref{ex:regint-interval}, taking the length of the interval to $\infty$. 
\end{example}

\subsection{Non-smooth convergent integrals}

As we have just seen, the regularized integral reduces to the usual integral when the integrand is a smooth form.  We will now prove a stronger statement: it reduces to the ordinary integral whenever the latter converges absolutely.

\begin{proposition}\label{prop:convergent-integrals}
    Let $\Sig = \Sigbas$ be a basic, oriented manifold with log corners of dimension $n$, and let $\omega \in \Alogc[n]{\Sig}$ be a top-degree form.  Let $j \colon \uSigo \to \Sig$ be the inclusion of the interior (a smooth manifold), and let $i \colon \bdSig \to \Sig$ be the canonical immersion of the boundary.
    Then the following statements are equivalent:
    \begin{enumerate}
        \item The integral $\int_{\uSigo} j^*\omega$ converges absolutely.
        \item The form $\omega$ has no poles on $\bdSig$.
        \item We have $i^*\omega = 0 \in \Alog[n]{\bdSig}$.
    \end{enumerate}
    Moreover, under these conditions, we have
    \[
    \int_{(\Sig,s)}\omega = \int_{\uSigo} j^*\omega
    \]
    for any regularization $s$ of $\Sig$.
\end{proposition}

\begin{proof}
    The equivalence of (1) and (2) is the content of \cite[Lemma~4.9]{Brown2009}; the statement in \emph{op.~cit.}~refers to analytic forms, but the argument only uses the existence of an expansion in powers of logarithms, and corresponding bounds on the integral, and hence it applies equally well in our setting.  To see the equivalence with (3), note that $\omega$ can be written in local coordinates $(r_1,\ldots,r_n)$ as 
    \[
    \omega = \sum_I \omega_I(r)\log^I(r) \dlog{r_1} \wedge \cdots\wedge \dlog{r_n}
    \]
    for some smooth functions $\omega_I$.  Then $\omega$ is free of poles if and only if each $\omega_I$ is divisible by $r_1\cdots r_n$.  But $i^*\omega$ is computed along the boundary stratum $r_l = 0$ by setting $r_l=0$ in each coefficient function $\omega_I$, and making the substitutions $\log(r_l) \mapsto \log(t_l)$ and $\dlog{r_l}\mapsto \dlog{t_l}$ where  $t_l = i^*r_l|_{r_l=0}$ is the corresponding phantom.  Hence $i^*\omega = 0$ if and only if $\omega_I$ is divisible by $r_l$ for all $I$ and $l$, as desired.

    Now suppose that the equivalent conditions (1)--(3) hold.   
 By \autoref{rem:continuous-relative-cohomology}, it follows that there exists a form $\eta \in \Alogc[n-1]{\Sig}$ and a smooth form $\uomega \in \Ac[n]{\uSig}$ such that $i^*\eta = 0 \in \Alogc[n-1]{\bdSig}$ and $\omega - \underline{\omega} = \dd\eta$. We will deduce the equality $\int_{(\Sig,s)}\omega = \int_{\uSigo} j^*\omega$ by computing both sides using different versions of Stokes' formula.

 On the one hand, by our regularized Stokes formula (\autoref{cor:reg-stokes}) we have
 \begin{align}
 \int_{(\Sig,s)}\omega = \int_{(\Sig,s)}\underline{\omega} + \int_{\partial(\Sig,s)}i^*\eta = \int_{\uSig}\underline{\omega}+0 = \int_{\uSigo}j^*\underline{\omega},\label{eq:convergent-eq1}
 \end{align}
 where in the last step we use that the boundary has measure zero.  

 On the other hand, we claim that $\eta$ extends continuously to $\ubdSig$ and the resulting continuous form $\underline{i^*\eta}$ on $\ubdSig$ is zero.  Indeed, in local coordinates, $\eta$ has the form
 \begin{align*}
 \eta = \sum \eta_{I,k}(r) \log^I(r) \dlog{r_1}\wedge  \cdots \wedge \widehat{\dlog{r_k}} \wedge \cdots \wedge \dlog{r_n} 
 \end{align*}
 for some smooth functions $\eta_{I,k}$, and  the condition $i^*\eta=0$ is equivalent to the condition that each function $\eta_{I,k}(r)$ vanishes on the boundary, from which the claim follows immediately.  It then follows by applying the classical Stokes formula to the continuous form $\eta$ as in \cite[Theorem 4.11]{Brown2009}, that the ordinary integral is given by
 \[
 \int_{\uSigo}j^*\omega = \int_{\uSigo}j^*{\underline{\omega}} + \int_{\ubdSig} \underline{i^*\eta} = \int_{\uSigo}j^*{\underline{\omega}} + 0 = \int_{\uSigo}j^*{\underline{\omega}},
 \]
 and hence it agrees with the regularized integral by \eqref{eq:convergent-eq1}.
\end{proof}

\begin{example}
    For $a > 0$, let $\Sig = [0,a]$ with coordinate $r$ and consider the logarithmic one-form $\omega = \log(r)\dd r$, which is absolutely integrable despite the singularity at $r=0$; the classical argument is to introduce a cutoff around zero, compute the integral $\int_{\epsilon}^a \log(r)\dd r$ using the fundamental theorem of calculus, and take the limit as $\epsilon \to 0$. 
 Hence for any choice of regularization of $\Sig$ we have    
    \[
    \int_{(\Sig,s)} \omega = \int_0^a \log(r) \dd r = a \log(a) - a
    \]
     Alternatively, this can be derived using our regularized Stokes's formula.  Indeed, the logarithmic function $\eta \defas r \log(r) - r \in \Cinflog{\Sig}$ is a primitive for $\omega$.  Adopting the notation of \autoref{ex:regint-interval}, the regularized Stokes formula gives
    \begin{align*}
    \int_{(\Sig,s)} \omega &= \int_{\partial(\Sig,s)}\eta\\
    &= s_a^* (r\log(r) -r ) - s_0^*(r\log(r) - r) \\
    &= (a\log(a) -a ) - (0\log(\lambda) - 0) \\
    &= a\log(a) - a,
    \end{align*}
    explicitly exhibiting the independence of  the choice of regularization.
\end{example}

\section{Periods of logarithmic varieties}
\label{sec:periods}
We now turn to the application of our results to the study of period integrals on logarithmic algebraic varieties.  In this section we assume basic familiarity with logarithmic algebraic geometry, as treated for instance in \cite{Kato1989a,KatoNakayama,Ogus}.

\subsection{Varieties with log corners}  Here and throughout, by a \defn{variety}, we mean a separated scheme of finite type over the field $\KK =\RR$ or $\CC$.  A \defn{log variety} is a tuple $X = (\uX,\cM_X,\alpha)$ where $\uX$ is a variety, $\cM_X$ is a sheaf of monoids in the \'etale topology on $\uX$, and $\alpha \colon \cM_X \to \cO{X}$ is a morphism of sheaves of monoids that identifies $\alpha^{-1}(\cOx{X})$ with $\cOx{X}$.    We denote by $\uX(\KK)$ the set of $\KK$-points of $\uX$, equipped with the classical analytic topology.

If $Y$ is a variety and $D \subset Y$ is a divisor, we denote by $X = Y \log D$ the \defn{divisorial log variety}, for which $\uX =Y$ and $\alpha \colon \cM_{Y \log D} \to \cO{Y}$ is the inclusion of the subsheaf of regular functions on $Y$ that are invertible on $Y \setminus D$.  If $Z \to Y$ is a locally closed immersion (which may have components in common with $D$) we endow it with the restricted log structure, giving a log variety we denote by $Z \log D$.

\begin{example}
Let $z$ be the standard coordinate on $\AF^1$.  The log structure on $\logcyl$ is given by the monoid $\cM_{\logcyl}=\cOx{\AF^1} z^\NN \subset \cO{\AF^1}$ of monomials in $z$ with invertible coefficients.  Its restriction to the origin gives the log variety $\logcirc$ with $\cM_{\logcirc} = \KK^\times w^\NN$ where $w = z|_{\logcirc}$ is a phantom; this log variety is often called the ``standard log point''. 
\end{example}

\begin{definition}
    A log variety $X$ is a \defn{variety with log corners} if \'{e}tale locally it admits a strict \'{e}tale morphism to $(\logcyl)^n \times (\logcirc)^k$ for some $n,k \in \NN$.
\end{definition}
Thus a variety with log corners is covered by \'etale charts consisting of functions $z_1,\ldots,z_n \in \cO{\uX}$ cutting out a normal crossing $\uD\subset \uX$, together with phantom elements $w_1,\ldots,w_k \in \cM_X$.  Globally, a variety with log corners is a log variety that is isomorphic to one of the form $X = Z \log D$ where $Z \to Y$ is the immersion of a union of strata of a normal crossing divisor $D$ in a smooth variety $Y$.

Ordinary and virtual morphisms $X \to Y$ of varieties with log corners are defined in the same fashion as for manifolds with log corners. An ordinary morphism consists of a map $\uphi \colon  \uX \to \uY$ of varieties over $\KK$ together with a pullback morphism of sheaves of monoids $\phi^* : \uphi^{-1}\cM_{Y} \to \cM_X$, which we require to be compatible with $\alpha_X$ and $\alpha_Y$. A virtual morphism consists of a map $\uphi\colon \uX\to \uY$ of varieties over $\KK$ together with a pullback morphism of sheaves of groups $\phi^*\colon \uphi^{-1}\cMgp_Y\to \cMgp_X$ which behaves as expected on $\cOx{Y} \subset \cMgp_Y$. The notion of virtual morphism was first defined in logarithmic algebraic geometry by Howell \cite{Howell2017}; see \cite{DPP:WeakMorphisms} for more details.

Every variety with log corners $X$ has a \defn{boundary $\bdX$}, given by the pullback of the log structure along the map $\widetilde{\uD} \to\uX$, where $\widetilde{\uD}$ is the normalization of $\uD$.  It comes equipped with a canonical morphism $\bdX \to X$, so that the iterated boundaries $\bdX[\bullet]$ form an (ordinary) $\ascat$-variety with log corners.

\subsection{Kato--Nakayama spaces}\label{sec:KN}

In \cite{KatoNakayama}, Kato--Nakayama associate topological spaces to a class of log varieties over $\CC$, which includes all varieties with log corners. In \cite{GillamMolcho}, Gillam--Molcho explained how to endow these spaces with differentiable positive log structures. As we explain in \cite{DPP:WeakMorphisms}, this construction is functorial for virtual morphisms.  The prototype is the following example.

\begin{example}
    The Kato--Nakayama space of $\logcyl$ is the real-oriented blowup of $\AF^1(\CC) = \CC$ at the origin; it is a manifold with boundary equipped with its basic log structure.  If $z$ is the standard coordinate on $\AF^1$, then the polar coordinates $r = |z|$ and $\theta = \arg z$ give an isomorphism of manifolds with log corners
    \[
    \KN{\logcyl} \cong \halfspc{} \times \unitcirc.
    \]
    The inclusion $\logcirc \hookrightarrow \logcyl$ then gives an isomorphism
    \[
    \KN{\logcirc} \cong \partial\, \KN{\logcyl}  \cong [0) \times \unitcirc
    \]
    of manifolds with log corners.
\end{example}

In general, if $X$ is a variety with log corners and $z_1,\ldots,z_n,w_1,\ldots,w_k$ are coordinates identifying an analytic open set $U$ in $X$ with an  analytic open set in $(\logcyl)^n\times(\logcirc)^k$, then $\KN{U}$ is identified with the corresponding open set in the manifold with log corners $(\halfspc{} \times \unitcirc)^n \times ([0)\times\unitcirc)^k$.  The functions $r_i = |z_i|$, $\theta_i = \arg(z_i)$ and $\phi_i = \arg(w_i)$ give basic coordinates, while $t_i \defas |w_i|$ are phantom coordinates on the factor $[0)^k$. We therefore have the following.

\begin{proposition}
    If $X$ is a variety with log corners of dimension $(n,k)$, then $\KN{X}$ is a manifold with log corners of dimension $(2n+k,k)$, and the natural map $\KN{X} \to \uX(\CC)$ is a morphism of manifolds with log corners, where $\uX(\CC)$ is viewed as a smooth manifold with the trivial positive log structure.
\end{proposition}

    \begin{figure}
    \centering
    \begin{subfigure}[t]{.32\linewidth}
    \centering\begin{tikzpicture}[scale=0.8]
    \centering
    \draw[white] (-2,-2) -- (2,2);
    \draw[thick] (-2,0) -- (2,0);
    \draw[blue,thick] plot[mark=x] (0,0);
    \end{tikzpicture}
    \caption{$X(\RR)$}
    \end{subfigure}
    \hfill
    \begin{subfigure}[t]{.32\linewidth}
    \centering
    \begin{tikzpicture}[scale=0.8]
    \draw[white,fill=white!90!black] (-2,-2) rectangle  (2,2);
    \draw[blue,thick] plot[mark=x] (0,0);
    \end{tikzpicture}
    \caption{$X(\CC)$}
    \end{subfigure}
    \hfill
    \begin{subfigure}[t]{.32\linewidth}
    \centering
    \begin{tikzpicture}[scale=0.8]
    \draw[white,fill=white!90!black] (-2,-2) rectangle  (2,2);
    \draw[thick] (0.5,0) -- (2,0);
    \draw[thick] (-2,0) -- (-0.5,0);
    \draw[blue] plot[mark=x,thick] (0.5,0);
    \draw[blue] plot[mark=x,very thick] (-0.5,0);
    \draw[blue,fill=white,thick] (0,0) circle (0.5);
    \end{tikzpicture}
    \caption{Kato--Nakayama spaces, real and complex}\label{fig:a1-real-KN-log-corner}
    \end{subfigure}
    \caption{Geometry of the log scheme $X = \logcyl$, which is defined over any ring $\KK$.  The portions of the diagram shown in blue correspond to the log subscheme $\bdX = \logcirc$, which gives the boundary of the Kato--Nakayama space.}\label{fig:a1-real-KN}
    \label{fig:enter-label}
\end{figure}

\subsection{Real points}\label{sec:real-KN} We now discuss a version of Kato--Nakayama's construction for varieties over $\RR$.  If $X$ is a variety with log corners defined over $\RR$, then we have an inclusion $\uX(\RR)\subset \uX(\CC)$ expressing the real points as the fixed locus of the anti-holomorphic involution given by complex conjugation.  This lifts to an involution $\sigma \colon \KN{X} \to \KN{X}$ of manifolds with log corners.  \autoref{thm:fixed-locus} then allows us to make the following construction.

\begin{definition}
    Let $X$ be a variety with log corners over $\RR$.  The \defn{real Kato--Nakayama space of $X$} is the manifold with log corners $\KNR{X} \subset \KN{X}$ defined as the fixed locus of the complex conjugation involution.
\end{definition}

Concretely, let $z_i,w_j$ be coordinates on $X$ defined over $\RR$, and  $(r_i,\theta_i,t_j,\phi_j)$ be the induced coordinates on $\KN{X}$ as above, with the angles $\theta_i$ and $\phi_j$ defined modulo $2\pi$. Then the conjugation is given by $\theta_i \mapsto - \theta_i$ and $\phi_i \mapsto -\phi_i$, so that the fixed locus is given by $\theta_i,\phi_j \in \ZZ\pi$, with coordinates induced by the basic radial coordinates $r_i$ and the phantom radial coordinates $t_j$.  From this we deduce the following.

\begin{corollary}
    If $X$ is a variety with log corners of dimension $(n,k)$ over $\RR$, then $\KNR{X}$ is a manifold with log corners of dimension $(n,k)$.
\end{corollary}

    Note that the proper morphism $\KN{X}\to \uX(\CC)$ restricts to a finite morphism of manifolds with log corners $\KNR{X}\to \uX(\RR)$, whose fibers are of the form $(\mathsf{S}^0)^r$.

\begin{example}
    The log variety $X = \logcyl$ is defined over $\RR$.  Its real Kato--Nakayama space is the fixed locus of complex conjugation on the real oriented blowup of $\CC$ at the origin.  It therefore consists of two semi-infinite intervals, which map to the non-negative and non-positive real axes, giving an isomorphism
    \[
    \KNR{\logcyl} \cong (-\infty,0] \sqcup [0,\infty)
    \]
    of manifolds with log corners, with boundary given by two standard ends
    \[
    \partial\,\KNR{\logcyl} \cong \KNR{\logcirc} \cong (0] \sqcup [0).
    \]
    See \autoref{fig:a1-real-KN-log-corner} for an illustration.
\end{example}

\subsection{Tangential basepoints: algebraic vs. differential geometry}
\label{sec:KN-basepoints}
Let $Y$ be a smooth variety over $\CC$ and $D \subset Y$ a normal crossing divisor.  If $p \in Y$ is a point, a \defn{tangential basepoint} at $p$, in the sense of Deligne~\cite[\S15]{Deligne1989}, is a choice of a nonzero normal vector for each local irreducible component of $D$ passing through $p$.  We refer to such basepoints as \defn{algebraic} to distinguish them from the $C^\infty$ tangential basepoints for manifolds with corners in this paper. They correspond to virtual morphisms from a point to the log scheme $Y\log D$ via the algebro-geometric analogue of \autoref{prop:tangential-basepoints-as-weak-morphisms}; see \cite{DPP:WeakMorphisms}.

\begin{figure}
    \begin{subfigure}{0.4\textwidth}
    \centering\begin{tikzpicture}
    \draw[white,fill=white!90!black] (-2,-2) rectangle (2,2);
    \draw[->,thick] (0,0)--(0.5,0);
    \draw[->,thick] (0,0) -- (65:1);
    \draw[blue,thick] plot[mark=x] (0,0);
    \end{tikzpicture}
    \caption{$\logcyl$}
    \end{subfigure}
    \begin{subfigure}{0.4\textwidth}
    \centering\begin{tikzpicture}
    \draw[white,fill=white!90!black] (-2,-2) rectangle (2,2);
    \draw[thick,blue,fill=white] (0,0) circle (0.5);
    \draw[->,thick] (0.5,0)--(1,0);
    \draw[->,thick] (65:0.5) -- (65:1.5);
    \end{tikzpicture}
    \caption{$\KN{\logcyl}$}
    \end{subfigure}
    \caption{Algebraic vs.~$C^\infty$ tangential basepoints at the origin in $\AF^1$.}
    \label{fig:baspoints-alg-vs-diff}
\end{figure}

The differential of the canonical map $\KN{Y\log D} \to Y(\CC)$ gives a bijection between $C^\infty$ tangential basepoints on the Kato--Nakayama space and algebraic tangential basepoints for $(Y,D)$, illustrated in \autoref{fig:baspoints-alg-vs-diff}.  Concretely, using the relation $z = r e^{\iu \theta}$ between holomorphic coordinates on $Y$ and polar coordinates on $\KN{Y\log D}$, the correspondence is given by
\[
c\, \cvf{z}|_{z=0} \qquad \longleftrightarrow \qquad |c| \, \cvf{r}|_{(r=0,\theta=\arg c)}
\]
for $c\in \CC^\times$.  We thus have the following:

\begin{proposition}\label{prop:basepoints-alg-vs-diff}
For a normal crossing divisor $D$ in a smooth variety $Y$ over $\CC$, the following are in canonical bijection:
    \begin{itemize}
        \item virtual morphisms $\Spec{\CC} \to Y \log D$ of varieties with log corners;
        \item virtual morphisms $* \to \KN{Y\log D}$ of manifolds with log corners;
        \item algebraic tangential basepoints for $(Y,D)$;
        \item $C^\infty$ tangential basepoints  for $\KN{Y \log D}$.
    \end{itemize}
\end{proposition}

A similar correspondence holds for varieties with log corners defined over $\RR$, replacing $\Spec{\CC}$ with $\Spec{\RR}$ and $\KN{Y\log D}$ with $\KNR{Y\log D}$. 
 
 \begin{remark}\label{rmk:basepoints-over-Z}
     More generally, one may consider log schemes over a base ring $\KK\subset \CC$, in which case the resulting tangential basepoints must be defined over $\KK$ as well, which may greatly rigidify the geometry.  For instance, if $X = \logcyl$ with coordinate $z$, defined over $\KK = \ZZ$, the only tangential basepoints at $0$ are $\pm\cvf{z}$; this gives a natural notion of a tangential basepoint having ``unit length''. See \cite{DPP:WeakMorphisms} for a discussion of this general notion of ``virtual point'' of log schemes.
 \end{remark}

\subsection{Betti and de Rham cohomology}
For a class of log varieties over $\CC$, Kato--Nakayama~\cite{KatoNakayama} defined the \defn{Betti cohomology} of $X$ to be the singular cohomology of the Kato--Nakayama space
\[
\HB{X} \defas \mathsf{H}^\bullet_{\mathrm{sing}}(\KN{X};\ZZ),
\]
and the \defn{(algebraic) de Rham cohomology} to be the hypercohomology (in the Zariski or \'{e}tale topology) of the algebraic logarithmic de Rham complex $(\forms{X},\dd)$, generated by the logarithmic derivatives of elements of $\cM_X$:
\[
\HdR{X} \defas \HH{\uX,(\forms{X},\dd)}.
\]
Furthermore, they established a \defn{comparison isomorphism}
\begin{align}
\HdR{X}  \stackrel{\sim}{\longrightarrow} \HB{X} \otimes_\ZZ \CC, \label{eq:comparison}
\end{align}
under certain assumptions on $X$ that are satisfied in the case of a variety with log corners. In the special case of varieties with log corners, this boils down to the classical fact that the log complex $\forms{Y}(\log D)$ of a normal crossing divisor $D \subset Y$ computes the cohomology of $Y \setminus D$.  We prove in \cite{DPP:WeakMorphisms} that these cohomology groups are functorial for virtual morphisms, and the comparison isomorphism is natural. 

Concretely, the real and imaginary parts of every algebraic log form $\omega \in \forms{X}$ define $C^\infty$ log forms $\Re \omega, \Im \omega \in \cAlog{\KN{X}}$ on the Kato--Nakayama space, viewed as a manifold with log corners.  This gives a canonical map
\[
\HdR{X} \to \HdR{\KN{X}}\otimes_\RR \CC
\]
which induces the isomorphism \eqref{eq:comparison} via our log de Rham theorem (\autoref{cor:log-de-Rham}).

\begin{remark}\label{rmk:dR-over-K}
If $X$ is the extension of scalars of a log variety $X_\KK$ defined over a subfield $\KK\subset \CC$, then the algebraic de Rham cohomology of $X_\KK$ is a $\KK$-structure on the algebraic de Rham cohomology of $X$: $\HdR{X}\cong \HdR{X_\KK}\otimes_\KK\CC$.
\end{remark}

\subsection{Logarithmic periods}

We now turn to the definition and cohomological interpretation of regularized period integrals in logarithmic algebraic geometry.

By a \defn{(virtual) $\ascat$-variety with log corners} we mean an $\ascat$-object 
\[
(X,Y_\bullet) = \rbrac{\augsimp{X}{Y_1}{Y_2}{Y_3}}
\]
in the category of virtual morphisms of varieties with log corners over $\CC$.  By the discussion above, the Betti and de Rham cohomology of log schemes extend immediately to such objects by totalizing the relevant symmetric coaugmented cosimplicial complexes as in \autoref{sec:relative-de-Rham}, and the comparison isomorphism gives a canonical pairing
\begin{equation}\label{eq:BdR-pairing-diagrams-varieties-log-corners}
\abrac{-,-} \colon \HlgyB{X,Y_\bullet} \otimes_\ZZ \HdR{X,Y_\bullet} \to \CC
\end{equation}
which becomes nondegenerate after tensoring with $\CC$.

Note that this construction includes the absolute cohomology $\coH{X}$ as the special case $Y_j = \varnothing$ for all $j > 0$, and the relative cohomology $\coH{X,Y}$ of a morphism $Y\to X$ as the special case $Y_1=Y$ and $Y_j = \varnothing$ for $j > 1$.

\begin{definition}\label{def:log-cycle}
    A \defn{logarithmic cycle in $(X,Y_\bullet)$} is the data of a compact oriented regularized manifold with log corners $(\Sig,s)$ and a morphism of $\ascat$-manifolds with log corners
    \[
    \phi \colon (\Sigbas,\bas{(\bdSig[\bullet])};s) \to \KN{X,Y_\bullet}.
    \]
\end{definition}

A logarithmic cycle has an underlying map of symmetric semi-simplicial spaces $\uphi:(\uSig,\partial^\bullet\uSig) \to \KN{\uX,\uY_\bullet}$ inducing a morphism in homology
\[
\underline{\phi}_* \colon \mathsf{H}^{\mathrm{sing}}_\bullet(\uSig,\partial_{\mathrm{top}}\uSig) \to \HlgyB{X,Y_\bullet}
\]
Since $\uSig$ is compact and oriented of dimension $n$, the image of its fundamental class gives the \defn{cycle class}, which we denote simply by
\[
[\phi] \defas \uphi_*[\uSig] \in \HlgyB[n]{X,Y_\bullet}.
\]
Such cycles arise naturally as integration domains in many interesting situations. 

As for integrands, let $\omega\in \sect{X,\Omega^n_X}$ be a global closed logarithmic $n$-form on $X$ whose pullback to $Y_\bullet$ is zero. It defines a class $[\omega]\in\HdR[n]{X,Y_\bullet}$, and we denote by 
\[
\int_\phi \omega \defas \abrac{[\phi],[\omega]} \in \CC
\]
the corresponding period, induced by the Betti--de Rham pairing \eqref{eq:BdR-pairing-diagrams-varieties-log-corners}. It can be computed as a regularized integral as in \autoref{sec:integration}, as the following result shows.

\begin{proposition}\label{prop:how-to-compute-log-periods}
The pairing \eqref{eq:BdR-pairing-diagrams-varieties-log-corners} between the cycle class $[\phi]\in\HlgyB[n]{X,Y_\bullet}$ and the class $[\omega]\in \HdR[n]{X,Y_\bullet}$ equals the regularized integral
\[
\int_\phi \omega = \int_{(\Sig,s)}\phi^*\omega.
\]
\end{proposition}

\begin{proof}
This follows from the functoriality of the Betti--de Rham comparison for $\ascat$-manifolds with log corners (\autoref{prop:BdR-comparison-manifolds-with-log-corners}) and the definition of the regularized integral.  
\end{proof}

\begin{remark}\label{rmk:periods-over-K}
    If $(X,Y_\bullet)$ is the extension of scalars of an $\ascat$-log scheme $(X_\KK,Y_{\KK\bullet})$ over a subfield $\KK\subset \CC$, then since the Betti homology is a finite rank $\ZZ$-module, the pairing \eqref{eq:BdR-pairing-diagrams-varieties-log-corners} restricts to a pairing $\HlgyB{X,Y_\bullet} \otimes_\ZZ \HdR{X_\KK,Y_{\KK\bullet}} \to \CC$ whose image is a finite-dimensional $\KK$-vector subspace of $\CC$, where we have used the base change for de Rham cohomology from \autoref{rmk:dR-over-K}.  In this way, we obtain arithmetically interesting periods from regularized integrals on log varieties defined over $\KK$.  Note that the fact that the object $(X,Y_\bullet)$ is defined over $\KK$ imposes a sort of $\KK$-rationality constraint on the induced maps of Kato--Nakayama spaces, and hence also on the scales involved in any regularized cycle. For example, let $X_\KK=C\,\log\,\{p\}$,  where $C$ is a smooth curve over $\KK$ and $p\in C$ is a closed point, let $Y_{\KK\bullet}=\Spec{\KK}$, and consider the diagram $Y_{\KK\bullet}\to X_\KK$ corresponding to a $\KK$-rational tangential basepoint $\vec{v}\in (T_pC)^\times$.  Then a regularized cycle $([0,1],\bas{\partial[0,1]};s) \to (\KN{X},\KN{Y_\bullet})$ must map the tangential basepoints at $0$ and $1$ (defining the scale on $\partial[0,1]$) to the tangential basepoint of $\KN{X}$ corresponding to the $\KK$-rational tangential basepoint $\vec{v}$. See \autoref{rmk:kummer-over-K} below for an explicit example, which illustrates how this constrains the possible periods of the diagram. 
\end{remark}

\subsection{Examples of logarithmic periods}\label{sec:period-examples}
We now explain how the constructions in this section recover and unify some classical examples of periods.

\subsubsection{The residue theorem with ``radius zero''}
Let $z$ be the standard coordinate on $X = \logcyl$. The class of the logarithmic form $\tfrac{\dd z}{z}$ is a basis of the first de Rham cohomology group $\HdR[1]{X}$. The Kato--Nakayama space of $X$ is $\KN{X} \cong \halfspc{} \times \unitcirc$ with radial coordinate $r = |z|$ and angular coordinate $\theta = \arg z$. For $\epsilon \ge 0$, let  
\[
\mapdef{\underline{\gamma_\epsilon}}{\unitcirc}{\KN{X}\cong [0,\infty)\times\unitcirc}{
e^{\iu\theta}}{(\epsilon, e^{\iu\theta})}
\]
denote the circle of radius $\epsilon$ in the Kato--Nakayama space, oriented counterclockwise. The homology classes of these cycles are all equal and form a basis of $\HlgyB[1]{X}$. To compute the period pairing $\HlgyB[1]{X}\otimes_\ZZ \HdR[1]{X}\to\CC $ one can therefore assume $\varepsilon>0$ and we get the usual result
\[
\abrac{[\underline{\gamma_\epsilon}],\sbrac{\tfrac{\dd z}{z}}} = \int_{\underline{\gamma_\epsilon}} \dlog{z} = \tipi.
\]
However, this computation does not make sense classically for $\epsilon=0$ since $\tfrac{\dd z}{z}=\tfrac{\dd r}{r}+\iu\,\dd\theta$ is ill defined at $r=0$.

This issue is solved using our formalism by lifting each $\gamma_\epsilon$ to a logarithmic cycle (\autoref{def:log-cycle}) with domain $\Sig=\unitcirc$, i.e.~a virtual morphism
\[
\gamma_\epsilon\colon\unitcirc\to \KN{X}\cong [0,\infty)\times\unitcirc.
\]
For $\epsilon>0$ there is nothing to add to the datum of $\underline{\gamma_\epsilon}$, but for $\varepsilon=0$, since $\underline{\gamma_0}$ lands in the boundary $\{0\}\times\unitcirc$, one also needs to specify the pullback by $\gamma_0$ of the coordinate $r$, which may be any positive smooth function $\lambda(e^{\iu\theta})$ on $\unitcirc$. (The choice of the constant function $\lambda=1$ is somewhat canonical, but our formalism allows more flexibility.) This choice can be thought of as a family of tangential basepoints at $0$ on $[0,\infty)$ indexed by $\unitcirc$, or alternatively as a scale for the manifold with log corners $\partial\KN{X}\cong [0)\times \unitcirc$. Using our notation for tangential basepoints from \autoref{sec:tangential-basepoints} we write
\[
\mapdef{\gamma_0}{\unitcirc}{\KN{X}\cong [0,\infty)\times\unitcirc}{
e^{\iu\theta}}{(\,\lambda(e^{\iu\theta})\,\partial_r|_0 \, , \, e^{\iu\theta}\, ).}
\]
The pullback of $\dlog{z}$ via $\gamma_0$ is now a well-defined smooth $1$-form on $\unitcirc$,
\[
\gamma_0^*\left(\mathrm{dlog}(z)\right)=\gamma_0^*\left(\mathrm{dlog}(r)+\iu\,\dd\theta\right)=\mathrm{dlog}(\lambda(e^{\iu\theta}))+\iu\,\dd\theta,
\]
and \autoref{prop:how-to-compute-log-periods} implies that the pairing between $[\gamma_0]$ and $[\tfrac{\dd z}{z}]$ is equal to
\[
\int_{\gamma_0}\dlog{z} = \int_{\unitcirc}\mathrm{dlog}(\lambda(e^{\iu\theta}))+\iu\,\dd \theta =2\pi\iu,
\]
which shows that $[\gamma_0]$ is independent of the choice of the function $\lambda(e^{\iu\theta})$.

\subsubsection{The logarithm as a regularized Kummer period}\label{sec:kummers} 
For real numbers $0 < \epsilon < a$, consider the integral
\[
I(\epsilon,a) = \int_\epsilon^a \dlog{z} = \log(a) - \log(\epsilon) = \log(a/\epsilon).
\]
It is classically interpreted as a period of the ``Kummer motive'' $\coH[1]{\AF^1\setminus \{0\},\{\epsilon,a\}}$, whose algebraic de Rham cohomology and Betti homology are given by
\[
\HdR[1]{\AF^1\setminus\{0\},\{\epsilon,a\}} = \CC \cdot \cbrac{[\dd z],\sbrac{\tfrac{\dd z}{z}}} \;\; \mbox{ and } \;\;
\HlgyB[1]{\AF^1\setminus\{0\},\{\epsilon,a\}} = \ZZ \cdot \{ [\gamma_\epsilon], [\eta_\epsilon] \}
\]
respectively, where $\gamma_\epsilon$ is the circle of radius $\epsilon$ as before and $\eta_\epsilon$ is the interval $[\epsilon,a]$.  We then have $I(\epsilon,a) = \int_{\eta_\epsilon}\dlog{z}$, and more generally we have the period matrix
\[
\begin{pmatrix}
\int_{\eta_\epsilon} \dd z & \int_{\eta_\epsilon} \dlog{z} \\
\int_{\gamma_\epsilon} \dd z & \int_{\gamma_\epsilon} \dlog{z}
\end{pmatrix} = \begin{pmatrix}
    a-\epsilon & \log(a/\epsilon)\\
    0 & \tipi
\end{pmatrix}
\]
which determines the Betti--de Rham pairing completely.

In the limit $\epsilon \to 0$, the integral $I(0,a)=I_1$ is the divergent integral discussed at the beginning of the paper, which must be regularized by choosing a tangential basepoint of $\AF^1$ at $0$. For simplicity, we choose a tangential basepoint that points in the positive real direction, i.e.\ a tangent vector 
\[
\vec{v} = \lambda \, \cvf{z}|_{z=0} \in (\tb[0]{\AF^1})^\times
\]
with $\lambda > 0$, and view it as a virtual morphism $\{0\}\to \logcyl$  by \autoref{prop:basepoints-alg-vs-diff}. Combined with the ordinary inclusion of $a \in \AF^1 \setminus \{0\}$ we obtain a diagram
\begin{equation}\label{eq:diagram-regularized-kummer}
\begin{tikzcd}
\logcyl & \{0\} \sqcup \{a\} \ar[l,shift right] \ar[l,shift left]
\end{tikzcd}
\end{equation}
which we view as a (virtual) $\ascat$-variety with log corners. Its relative cohomology
\[
\coH[1]{\logcyl,\{\vec{v},a\}}
\]
deserves the name ``regularized Kummer motive''.
The classes of $\dd z$ and $\tfrac{\dd z}{z}$ still form a basis of relative de Rham cohomology. In order to describe a basis of Betti homology, we equip the interval $[0,a]$ with the regularization $s$ given by the tangential basepoint $\lambda \, \cvf{r}|_{r=0}$ at $0$, and any tangential basepoint at $a$. This gives the following logarithmic cycle (\autoref{def:log-cycle}) for \eqref{eq:diagram-regularized-kummer}, denoted by $\eta_0$.
\[
\begin{tikzcd}
\KN{\logcyl} & \{0\} \sqcup \{a\} \ar[l,shift right] \ar[l,shift left] \\
{[0,a]} \ar[u,"\eta_0"] & \{0\} \sqcup \{a\} \ar[u,equal] \ar[l,shift right] \ar[l,shift left]
\end{tikzcd}
\]
and the relative Betti homology is given by
\[
\HlgyB[1]{\logcyl,\{\vec{v},a\}} = \ZZ \cdot \cbrac{[\gamma_0],[\eta_0]}
\]
where $\gamma_0$ is the class of the boundary circle as above; see \autoref{fig:Kummer-Betti}.  By \autoref{prop:how-to-compute-log-periods}, the period corresponding to $I(0,a)$ is then the classical regularized integral computed in \autoref{ex:regint-interval},
\[
\int_{\eta_0}\dlog{z} = \int_{([0,a],s)}\dlog{r} = \log(a/\lambda).
\] 
It gives the value $I_1 = \log(a)$ from the introduction when $\lambda=1$, corresponding to the case in which the tangential basepoint is defined over $\ZZ$ as in \autoref{rmk:basepoints-over-Z}.  The full period matrix of $\coH[1]{\logcyl,\{\vec{v},a\}}$ is given by
\[
\begin{pmatrix}
\int_{\eta_0} \dd z & \int_{\eta_0} \dlog{z} \\
\int_{\gamma_0} \dd z & \int_{\gamma_0} \dlog{z}
\end{pmatrix} = \begin{pmatrix}
    a & \log(a/\lambda)\\
    0 & \tipi
\end{pmatrix}.
\]
\begin{figure}
\begin{subfigure}{0.4\textwidth}
    \centering\begin{tikzpicture}
    \draw[white,fill=white!90!black] (-0.5,-2) rectangle (3,2);
    \draw plot[mark=x] (0.5,0);
    \draw[->] (0.5,0) -- (1,0);
    \draw[fill=black] (2,0) circle (0.05);
    \draw (0.5,-0.35) node {$\vec v$};
    \draw (2,-0.35) node {$a$};
    \end{tikzpicture}
    \caption{The pair $(\logcyl,\{\vec v,a\})$}
    \end{subfigure}
    \begin{subfigure}{0.4\textwidth}\begin{tikzpicture}
    \draw[white,fill=white!90!black] (-1.5,-2) rectangle (3,2);
    \draw[thick,fill=white,decoration={markings, mark=at position 0.5 with {\arrow{>}}},
        postaction={decorate}] (0,0) circle (0.5);
    \draw[thick,decoration={markings, mark=at position 0.5 with {\arrow{>}}},
        postaction={decorate}] (0.5,0) -- (2,0);
    \draw[fill=black] (0.5,0) circle (0.05);
    \draw[fill=black] (2,0) circle (0.05);
    \draw (1.25,0.25) node {$\eta_0$};
    \draw (0,0.75) node {$\gamma_0$};
    \draw (2,-0.3) node {$a$};
    \end{tikzpicture}
    \caption{Betti chains}
    \end{subfigure}
    \caption{The geometry of the ``regularized Kummer motive'' $\coH[1]{\logcyl,\{\vec v,a\}}$}
    \label{fig:Kummer-Betti}
\end{figure}

\begin{remark}\label{rmk:kummer-over-K}
    Suppose that $\KK \subset \CC$ is a subfield, and that the tangential basepoint $\vec v = \lambda \cvf{z}|_{z=0}$ and the point $a$ are also defined over $\KK$, i.e. $\lambda,a \in \KK$. Then the pair $(\AF^1\log \{0\},\{\vec v,a\})$ lifts to a pair defined over $\KK$. Note that $\lambda$ also determines the scale $\lambda\cvf{t}$ on $[0)$ required to make the interval $[0,a]$ into a regularized cycle.  The resulting collection of periods is the $\KK$-vector space
    \[
    \KK +  \tipi\KK + \log(a/\lambda) \KK \subset \CC
    \]
    spanned by the entries of the period matrix, illustrating the general principles discussed in \autoref{rmk:periods-over-K}.
\end{remark}

\subsection{Single-valued integration and the double-copy formula}

The periods we have considered so far involve the integration of algebraic log forms on a complex variety with log corners $X$ over subspaces of $\uX(\CC)$.  In many applications, one is interested in integrals of products of holomorphic and antiholomorphic forms over $\uX(\CC)$ itself, like the integral $I_2$ from the introduction.  Such integrals can be reduced to holomorphic periods of $X$ by way of the ``double copy'' formula for single-valued integration from \cite{BrownDupont}.  We now explain how that recipe can be recovered using our formalism.

\subsubsection{Doubling and the twisted diagonal}
The key point is that the integrals in question can be thought of in purely holomorphic/algebraic terms, as the integral of holomorphic forms on $X \times \overline{X}$ over the diagonal copy of $X$, where $\overline{X}$ denotes the complex conjugate of $X$.  Thus $\overline{X}$ is given by the same underlying log variety, but with the conjugate complex structure.  Equivalently, we replace the structure map $X \to \Spec{\CC}$ with its complex conjugate.

Note that $X$  is not a complex subvariety of $X \times \cX$; rather it is the fixed locus of the antiholomorphic involution of $X \times \overline{X}$ which interchanges the factors, and is thus totally real.  Note further that if $X$ is disconnected, the product $X \times\overline{X}$ will have connected components that do not intersect the diagonal; these may be ignored for the purposes of studying such integrals.  This motivates the following definition
 \begin{definition}
     Let $X$ be a complex variety with log corners whose connected components are denoted by $X_i$. The \defn{double of $X$} is the complex variety with log corners defined as the disjoint union  $\bigsqcup_i X_i\times \overline{X_i} \subset X\times\cX$.
 \end{definition}
 
  The diagonal $X(\CC) \to X(\CC) \times \cX(\CC)$ lifts canonically to a morphism of manifolds with log corners
 \[
 \KN{X} \to \KN{\Dbl{X}} \subset   \KN{X} \times \KN{\cX},
 \]
 which we call the \defn{twisted diagonal}; it identifies $\KN{X}$ with the fixed points of the involution that swaps the factors in the product.

\begin{remark}
More abstractly, we may consider the Weil restriction $W(X)$; it is a variety with log corners over $\RR$ whose $\RR$-points are in bijection with the $\CC$-points of $X$.  We have $W(X)\times_\RR\CC \cong X \times \cX$, so that the twisted diagonal is the inclusion $\KNR{W(X)} \to \KN{W(X)}$ of the real Kato--Nakayama space in the sense of \autoref{sec:real-KN}.
\end{remark}

\subsubsection{Polar smooth forms}
We now restrict to the case in which $X = Y \log D$ is the variety with log corners associated to a connected smooth proper complex variety $Y$ of dimension $n$ and a normal crossing divisor $D$.  Applying the doubling construction to $X$ and its iterated boundaries, we obtain an ordinary $\ascat$-variety with log corners $(\Dbl{X},\Dbl{\partial^\bullet X})$.  The twisted diagonal then gives a morphism of $\ascat$-manifolds with log corners
\[
\KN{X,\bdX[\bullet]} \to \KN{\Dbl{X,\bdX[\bullet]}}
\]
whose class in Betti homology we denote by
\[
[\KN{X}] \in \HlgyB[2n]{\Dbl{X},\Dbl{\partial^\bullet X}}.
\]

To understand the periods we must examine the differential forms on the open subscheme $\Dbl{X} \subset X\times\overline{X}$ relative to $\Dbl{\partial X} \subset \bdX\times\cbdX$.  For this note that since the restriction map $\forms{X\times \cX} \to \forms{\partial X \times \overline{\partial X}}$ is surjective, the relative de Rham complex is modelled by its kernel
\[
\forms{\Dbl{X},\Dbl{\bdX}} \subset \forms{\Dbl{X}} = \left.\forms{X\times\cX}\right|_{\Dbl{X}}.
\]
Following \cite{BrownDupont}, we use the following terminology:
\begin{definition}A form $\omega \in \forms{\Dbl{X}}$ is \defn{polar smooth} if its pullback to $\KN{X}$ is smooth, i.e.~it lies in the sheaf $\cA{\underline{\KN{X}}}$ of $C^\infty$ forms on the underlying manifold with corners.
\end{definition}

\begin{lemma}
For a section $\omega \in \sect{\Dbl{X},\forms[2n]{\Dbl{X}}}$, the following are equivalent.
\begin{enumerate}
    \item  The form $\omega$ lies in the subsheaf $\forms[2n]{\Dbl{X},\Dbl{\bdX}}$.
    \item The form $\omega$ is polar smooth.
    \item For every open subset $U \subset \KN{X}$ with compact closure, the integral $\int_U\omega$ converges absolutely.
    \item If $z \in \cO{Y}$ is a local defining equation for any irreducible component of $D$, then the double residue along $z=\zb=0$ vanishes:
    \[
    \res_{z=0}\res_{\zb=0}\omega = 0.
    \]
\end{enumerate}
\end{lemma}

\begin{proof}
    The problem is local and invariant under taking products  with smooth varieties, so we may assume without loss of generality that $X = (\logcyl)^n$ with coordinates $z_j$, so that $\forms{X\times\cX}$ is generated by $\dlog{z_j}$ and $\dlog{\zb_j}$.  A general element of $\forms[2n]{X\times\cX}$ thus has the form
    \[
    \omega = f \dlog{z_1} \wedge \dlog{\zb_1} \wedge \cdots \wedge \dlog{z_n}\wedge\dlog{\zb_n} 
    \]
    for a polynomial $f \in \CC[z_1,\zb_1,\ldots,z_n,\zb_n]$.
    
    The log variety $\Dbl{\partial X}$  has $n$ connected components $Z_1,\ldots,Z_j$, each of codimension two, identified with the loci $z_j=\zb_j = 0$ for $1 \le j \le n$. 
  The restriction of $\omega$ to such a component is given by
    \[
    \omega|_{Z_j} = f|_{z_j=\zb_j=0} \dlog{t_j}\wedge\dlog{\overline{t_j}} \wedge \dlog{z_1}\wedge\dlog{\zb_1}  \wedge \cdots \wedge \widehat{\dlog{z_j}}\wedge\widehat{\dlog{\zb_j}} \wedge \cdots\wedge \dlog{z_n}\wedge\dlog{\zb_n}
    \]
    where $t_j,\overline{t}_j$ are the phantom coordinates corresponding to $z_j,\zb_j$.  From this we deduce that $\omega|_{Z_j} = 0$ if and only if $f$ vanishes on the linear subspace $z_j=\zb_j = 0$ in $\AF^n$ or equivalently the double residue $\res_{z_j=0}\res_{\zb_j=0} \omega$ is zero.  This gives the equivalence of conditions (1) and (4).

   On the other hand, the pullback to $\KN{X}$ is computed by converting to polar coordinates $z_j = r_j e^{\iu \theta_j}$ and $\zb_j = r_j e^{-\iu\theta_j}$, giving
    \[
    \omega = (-2 \iu)^n f \dlog{r_1}\wedge \dd\theta_1 \wedge \cdots \wedge \dlog{r_n}\wedge \dd\theta_n.
    \]
    Considering the behaviour of $\omega$ as $r_j \to 0$, we see that $\omega$ is smooth on $\KN{X}$ if and only if the pullback of $f$ vanishes on the boundary component $r_j=0$ for every $j$, but this is evidently equivalent to the vanishing of $f$ when $z_j=\zb_j = 0$, and also to the absolute integrability, as desired.
\end{proof}

\subsubsection{The double copy formula}
Now suppose further that the divisor $D \subset Y$ is decomposed as union $D = A \cup B$, where $A$ and $B$ have  no common irreducible components.  We have canonical morphisms of varieties with log corners
\[
\begin{tikzcd}
 X_A := Y \log A & X_D := Y \log D \ar[r]\ar[l] & X_B := Y \log B 
\end{tikzcd}
\]
induced by the inclusions of divisors.  On the level of forms, these maps induce the inclusion of forms with logarithmic poles on $A$ or $B$ into the forms with poles on $D = A \cup B$.   Applying the doubling construction, we obtain a morphism of $\ascat$-varieties with log corners
\begin{align}
(\Dbl{X_D},\Dbl{\bdX[\bullet]_D}) \to \rbrac{X_A\times \cX_B, \partial^\bullet(X_A \times \cX_B)}
\label{eq:double-copy-pairs}
\end{align}
We also have maps
\[
A_B := A \log B \to X_B \qquad B_A := B \log A \to X_A
\]
induced by the pullback of log structures along the normalization maps of $A$ and $B$.  Since $D = A \cup B$,  Alexander--Lefschetz--Poincar\'e duality implies that the intersection pairing
\[
\HlgyB[n]{X_A,B_A} \times \HlgyB[n]{X_B,A_B} \to \HlgyB[2n]{X_D,\bdX_D} \cong \ZZ
\]
is perfect.  Let
\begin{align*}
\{\gamma_i\}_i \subset \HlgyB[n]{X_A,B_A} \qquad \{\gamma_i^\vee\}_i \in \HlgyB[n]{X_B,A_B}
\end{align*}
be dual bases.  The  cycle class of the twisted diagonal in the relative homology $\HlgyB[2n]{X_A\times\cX_B,\partial^\bullet(X_A\times\cX_B)}$ is then given, under the K\"unneth decomposition, by
\[
[\KN{X_D}] = \sum \gamma_i \otimes \gamma_i^\vee,
\]
from which we deduce the following.
\begin{proposition}[{Double copy formula, \cite[Corollary 1.5]{BrownDupont}}]
    If $\omega \in \sect{Y,\forms[n]{Y \log A}}$ and $\nu \in \sect{Y,\forms[n]{Y \log B}}$ are global logarithmic forms, then $\omega \wedge \overline{\nu}$ is polar smooth, and its period over the twisted diagonal is given by
    \[
    \abrac{[\KN{X_D}],[\omega\wedge\overline{\nu}]} = \int_{Y(\CC)}\omega \wedge \overline{\nu} = \sum_i \int_{\gamma_i}\omega \int_{\gamma_i^\vee} \overline{\nu}.
    \]
\end{proposition}

\begin{example}We now explain how to treat the integral $I_2$ from the introduction as a logarithmic period. 
 Let $Y = \PP^1$, suppose that $a \in \PP^1 \setminus \{0,1,\infty\}$ and consider the divisors
\[
A = \{1,a\} \qquad B = \{0,\infty\} \qquad D = A \cup B = \{0,1,a,\infty\}.
\]
Setting $\omega = \tfrac{\dd z}{z-a} - \tfrac{\dd z}{z-1}$ and $\nu = \dlog{z}$, we have
\[
\iint_{Y(\CC)} \omega \wedge \overline{\nu} = \iint_{\PP^1(\CC)}\rbrac{\frac{\dd z}{z-a} - \frac{\dd z}{z-1}} \wedge \dlog{\zb} = I_2
\]
As explained in the introduction, we can compute its value by applying the regularized Stokes formula on $\Sig = \KN{Y\log D}$ (with respect to any choice of regularization), yielding $I_2 = \tipi \log|a|^2$. This kind of computation has appeared in the literature under the name ``Cauchy--Stokes formula'', see e.g.~\cite{KhesinRosly}. On the other hand in \cite[Example 1.6]{BrownDupont}, the same result is obtained by the double copy formula, which reduces the computation to the periods of the Kummer motives from \autoref{sec:kummers}.
\end{example}

\bibliographystyle{hyperamsalpha}
\bibliography{log-comp-biblio}

\end{document}

%% file: log-comp-preamble.tex
\usepackage{mathscinet}
\usepackage{tikz-cd}
\usepackage{tikz}
\usetikzlibrary{decorations.markings,shapes.misc,plotmarks}
\usepackage{caption}
\usepackage{subcaption}
\usepackage{amsthm,amsfonts,amssymb,amsmath} 

\usepackage{aliascnt} 
\usepackage{mathrsfs} 
\usepackage{xargs,ifthen} 
\usepackage{enumitem}

\newcommand{\mapdef}[5]{
\begin{array}{cccc}
#1 \colon & #2 &\longrightarrow & #3 \\
& #4 &\longmapsto& #5
\end{array}
}

\usepackage{bigints}
\usepackage[Smaller]{cancel}
\usepackage{color}

\usepackage{todonotes}

\usepackage{hyperref}

\definecolor{urlcolor}{rgb}{.2,.2,.6}
\definecolor{linkcolor}{rgb}{.1,.1,.5}
\definecolor{citecolor}{rgb}{.4,.2,.1}
\hypersetup{colorlinks=true, urlcolor=urlcolor, linkcolor=linkcolor, citecolor=citecolor}

\newcommandx{\thdef}[2]{
	\newaliascnt{#1}{theorem}  
	\newtheorem{#1}[#1]{#2}
	\aliascntresetthe{#1}  
	\expandafter\newcommand\expandafter{\csname #1autorefname\endcsname}{#2}
}

\newtheorem{theorem}{Theorem}[section]

\thdef{lemma}{Lemma}
\thdef{corollary}{Corollary}
\thdef{conjecture}{Conjecture}
\thdef{proposition}{Proposition}
\thdef{theoremdefn}{Theorem/Definition}

\theoremstyle{definition}
\thdef{definition}{Definition}

\theoremstyle{remark}
\thdef{remarkx}{Remark}
\thdef{examplex}{Example}

\newenvironment{example}
  {\pushQED{\qed}\examplex}
  {\popQED\endexamplex}

\newenvironment{remark}
  {\pushQED{\qed}\remarkx}
  {\popQED\endremarkx}
  
\newcommand{\defn}[1]{\textbf{\textit{#1}}} 

\newcommand{\CC}{\mathbb{C}}
\newcommand{\KK}{\mathbb{K}}
\newcommand{\RR}{\mathbb{R}}
\newcommand{\QQ}{\mathbb{Q}}
\newcommand{\NN}{\mathbb{N}}
\newcommand{\ZZ}{\mathbb{Z}}
\newcommand{\AF}{\mathbb{A}}
\newcommand{\PP}{\mathbb{P}}

\newcommand{\model}[2]{\halfspc{#1}\times[0)^{#2}}
\newcommand{\halfspc}[1]{[0,\infty)^{#1}}

\newcommand{\logcyl}{\AF^1 \log\, \{0\}}
\newcommand{\logcirc}{\{0\} \log\, \{0\}}

\newcommand{\id}[1]{\mathrm{id}_{#1}}

\renewcommand{\le}{\leqslant}
\renewcommand{\ge}{\geqslant}

\newcommand{\iu}{\mathrm{i}}
\newcommand{\ipi}{\pi\iu }
\newcommand{\tipi}{2\ipi}
\newcommand{\ulog}{\operatorname{\widetilde{\log}}}
\newcommand{\uf}{\widetilde{f}}
\newcommand{\ug}{\widetilde{g}}
\newcommand{\tp}{\tilde p}
\newcommand{\ts}{\tilde s}
\newcommand{\uomega}{\underline{\omega}}

\newcommand{\defas}{\mathrel{\mathop:}=}

\newcommand{\zb}{\overline{z}}
\newcommand{\tf}{\widetilde{f}}

\renewcommand{\epsilon}{\varepsilon}

\newcommand{\set}[2]{\left\{#1\,\middle|\,#2\right\}}
\newcommand{\rbrac}[1]{\left(#1\right)}
\newcommand{\abrac}[1]{\left\langle#1\right\rangle}
\newcommand{\sbrac}[1]{\left[#1\right]}
\newcommand{\cbrac}[1]{\left\{#1\right\}}

\newcommand{\cO}[1]{\mathcal{O}_{#1}}
\newcommand{\cOx}[1]{\cO{#1}^\times}

\newcommand{\cM}{\mathscr{M}}
\newcommand{\cMgp}{\mathscr{M}^{\mathrm{gp}}}
\newcommand{\cMbas}{\mathscr{M}^{\mathrm{bas}}}
\newcommand{\cMphan}{\mathscr{M}^{\mathrm{phan}}}

\newcommand{\Cinf}[1]{\mathscr{C}^{\infty}_{#1}}
\newcommand{\Cinfpos}[1]{\mathscr{C}^{\infty, > 0}_{#1}}
\newcommand{\Cinfge}[1]{\mathscr{C}^{\infty, \ge 0}_{#1}}
\newcommand{\Cinflog}[1]{\mathscr{C}^{\infty,\log}_{#1}}

\newcommand{\dd}{\mathrm{d}}
\newcommand{\dlog}[1]{\frac{\dd#1}{#1}}
\newcommand{\cA}[2][\bullet]{\mathscr{A}^{#1}_{#2}} 
\newcommand{\A}[2][\bullet] {\mathscr{A}^{#1}\rbrac{#2}} 
\newcommand{\Ac}[2][\bullet]{\mathscr{A}_c^{#1}\rbrac{#2}} 
\newcommand{\cAlog}[2][\bullet]{\mathscr{A}^{#1,\log}_{#2}} 
\newcommand{\Alog}[2][\bullet]{\mathscr{A}^{#1,\log}(#2)} 
\newcommand{\Alogc}[2][\bullet]{\mathscr{A}^{#1,\log}_c(#2)} 

\newcommand{\VFlog}[1]{\mathcal{T}^{\log}_{#1}}
 \newcommand{\coVFlog}[1]{\mathcal{T}^{\vee,\log}_{#1}}
\newcommand{\cvf}[1]{\partial_{#1}}
\newcommand{\logcvf}[1]{#1\partial_{#1}}

\newcommand{\reg}[1][]{\mathrm{reg}_{#1}}
\DeclareMathOperator{\reglim}{reglim}
\DeclareMathOperator*{\res}{Res}
\newcommand{\Sig}{\Sigma}
\newcommand{\uSig}{\underline{\Sigma}}
\newcommand{\uSigo}{\underline{\Sigma}^\circ}
\newcommand{\uphi}{\underline{\phi}}
\newcommand{\uPsi}{\underline{\Psi}}
\newcommand{\bas}[1]{{#1}^{\mathrm{bas}}}
\newcommand{\phant}{\mathrm{phan}}

\newcommand{\Sigbas}{\bas{\Sig}}
\newcommand{\ubdSig}[1][]{\underline{\partial^{#1}\Sigma}}
\newcommand{\bdSig}[1][]{\partial^{#1}\Sigma}

\newcommand{\bdSigbas}[1][]{\partial^{#1}\Sigbas}

\newcommand{\intSig}{\Sigma^\circ}

\newcommand{\Psibas}{\bas{\Psi}}

\newcommand{\tb}[2][]{T_{#1}#2}
\newcommand{\tbphan}[2][]{T_{#1}^{\phant}#2}
\newcommand{\tbpos}[2][]{T^{>0}_{#1}#2}
\newcommand{\tbposphan}[2][]{T^{\phant,>0}_{#1}#2}
\newcommand{\tbge}[2][]{T^{\ge0}_{#1}#2}
\newcommand{\tbgephan}[2][]{T^{\phant,\ge0}_{#1}#2}
\newcommand{\ctb}[2][]{T^\vee_{#1}#2}

\newcommand{\cbdX}{\overline{\partial X}}
\newcommand{\bdX}[1][]{\partial^{#1} X}
\newcommand{\cX}{\overline{X}}
\newcommand{\uX}{\underline{X}}
\newcommand{\uY}{\underline{Y}}

\newcommand{\uD}{\underline{D}}

\newcommand{\KN}[1]{\mathrm{KN}(#1)}
\newcommand{\KNR}[1]{\mathrm{KN}_\RR\rbrac{#1}}
\newcommand{\Spec}[1]{\mathsf{Spec}(#1)}
\newcommand{\Dbl}[1]{\mathrm{Dbl}(#1)}

\newcommand{\sect}[1]{\mathsf{\Gamma}(#1)}
\newcommand{\coH}[2][\bullet]{\mathsf{H}^{#1}\rbrac{#2}}

\newcommand{\HH}[2][\bullet]{\mathbb{H}^{#1}\rbrac{#2}}
\newcommand{\HdR}[2][\bullet]{\mathsf{H}_{\mathrm{dR}}^{#1}(#2)}
\newcommand{\HdRc}[2][\bullet]{\mathsf{H}_{\mathrm{dR},\mathrm{c}}^{#1}(#2)}
\newcommand{\HB}[2][\bullet]{\mathsf{H}_{\mathrm{B}}^{#1}(#2)}
\newcommand{\HlgyB}[2][\bullet]{\mathsf{H}^{\mathrm{B}}_{#1}(#2)}

\newcommand{\forms}[2][\bullet]{\Omega^{#1}_{#2}}

\newcommand{\End}[2][]{\mathsf{End}_{#1}\rbrac{#2}}
\newcommand{\GL}[2]{\mathrm{GL}_{#1}(#2)}
\newcommand{\symgp}[1]{\mathfrak{S}_{#1}}
\newcommand{\sgn}[1]{\mathrm{sgn}_{#1}}
\newcommand{\unitcirc}{\mathsf{S}^{1}}

\newcommand{\simp}[3]
{
\begin{tikzcd}[ampersand replacement=\&]
     {#1} \& {#2} \ar[l,shift left = 1]\ar[l,shift right=1] \& {#3} \ar[l,shift left = 2]\ar[l,shift left = 0]\ar[l,shift left = -2] \& \ar[l,shift left = 3]\ar[l,shift left = 1]\ar[l,shift left = -1]\ar[l,shift left = -3] \cdots
    \end{tikzcd}
}

\newcommand{\augsimp}[4]
{
\begin{tikzcd}[ampersand replacement=\&]
     #1 \& #2 \ar[l] \& #3
     \ar[l,shift left = 1]\ar[l,shift right=1] \& #4 \ar[l,shift left = 2]\ar[l,shift left = 0]\ar[l,shift left = -2] \& \ar[l,shift left = 3]\ar[l,shift left = 1]\ar[l,shift left = -1]\ar[l,shift left = -3] \cdots
    \end{tikzcd}
}

\newcommand{\augsimpshort}[4]
{
\begin{tikzcd}[ampersand replacement=\&]
     #1 \& #2 \ar[l] \& #3
     \ar[l,shift left = 1]\ar[l,shift right=1] \& \ar[l,shift left = 2]\ar[l,shift left = 0]\ar[l,shift left = -2]  #4
    \end{tikzcd}
}

\newcommand{\scat}{\mathscr{I}}
\newcommand{\ascat}{\mathscr{I}_+}
\newcommand{\op}{\mathrm{op}}
\DeclareMathOperator*{\holim}{holim}
\DeclareMathOperator*{\fibre}{fibre}